\documentclass[11pt]{amsart}

\usepackage{amsmath}
\usepackage{amssymb}
\usepackage{mathtools}
\usepackage{microtype}
\usepackage{enumitem}
\usepackage{mathrsfs}
\usepackage{booktabs}
\usepackage{array}
\usepackage{graphicx}
\usepackage{xcolor}
\usepackage{tikz}
\usepackage{subcaption}

\usetikzlibrary{decorations.pathreplacing}
\usepackage[hidelinks]{hyperref}
\hypersetup{pdftitle={Structure of ENO Entropy Dissipation: Parity Dichotomy for the ENO-TV Conjecture and Shift Cohomology}, pdfauthor={Zhuoyun Li and Kailiang Wu}, pdfsubject={Coercivity, ENO null modes, and entropy-pair transfer for ENO and TeCNO schemes}, pdfkeywords={ENO reconstruction, ENO-TV conjecture, entropy dissipation, parity dichotomy, ENO null modes, shift cohomology}}
\usepackage[
paperwidth=7in,
paperheight=10in,
left=1in,
right=1in,
top=0.86in,
bottom=0.76in,
headheight=8pt,
headsep=14pt
]{geometry}

\allowdisplaybreaks
\theoremstyle{plain}
\newtheorem{theorem}{Theorem}[section]
\newtheorem{lemma}[theorem]{Lemma}
\newtheorem{proposition}[theorem]{Proposition}
\newtheorem{corollary}[theorem]{Corollary}
\theoremstyle{definition}
\newtheorem{definition}[theorem]{Definition}
\theoremstyle{remark}
\newtheorem{remark}[theorem]{Remark}
\newtheorem*{remark*}{Remark}
\numberwithin{equation}{section}

\theoremstyle{plain}
\newtheorem{lettertheorem}{Theorem}

\newcommand{\D}{\mathrm D}

\DeclareMathOperator{\Range}{Range}

\title[Structure of ENO Entropy Dissipation]{Structure of ENO Entropy Dissipation:\\
	Parity Dichotomy for the ENO--TV Conjecture and Shift Cohomology}

\author[Z. Li]{Zhuoyun Li}
\address[Zhuoyun Li]{Department of Mathematics, Southern University of Science and Technology, Shenzhen, Guangdong 518055, China}

\author[K. Wu]{Kailiang Wu}
\address[Kailiang Wu]{Department of Mathematics and Shenzhen International Center for Mathematics, Southern University of Science and Technology, Shenzhen, Guangdong 518055, China}
\email[Kailiang Wu]{wukl@sustech.edu.cn}
\thanks{Kailiang Wu is the corresponding author. The authors were supported, in part, by the Science Challenge Project (No.~TZ2025007) and the Shenzhen Science and Technology Program (Nos.~JCYJ20250604144300001 and RCJC20221008092757098).}

\subjclass[2020]{Primary 35L65, 65M12, 41A15; Secondary 26D10, 65M08}
\date{}

\usepackage{etoolbox}
\patchcmd{\abstract}{\leftmargin3pc}{\leftmargin1pc}{}{}

\begin{document}

\begin{abstract}
\renewcommand{\baselinestretch}{1.08}\selectfont
Essentially non-oscillatory (ENO) reconstruction provides a key mechanism for designing high-order entropy-stable schemes for hyperbolic conservation laws, with its sign property [Fjordholm, Mishra, and Tadmor, {\em Found.~Comput.~Math.} 13 (2013)] ensuring nonnegative local dissipation for a prescribed entropy. However, two fundamental questions concerning convergence remain open: whether this dissipation provides the coercivity required for weak-BV compactness, as posited by the long-standing ENO--TV conjecture, and whether entropy stability transfers from the prescribed entropy pair to additional pairs. 

This paper resolves the ENO--TV conjecture by establishing a sharp parity dichotomy: it holds if and only if the reconstruction order $k=2$ or $k$ is odd, and fails for all even orders $k\ge 4$.
 The key to our proof is a localization principle that eliminates dependence on nonlinear adaptive stencil selection, establishing a two-sided equivalence between ENO dissipation and a canonical finite-difference functional. For odd orders, the conjecture is proved via a hidden quadratic energy and novel discrete Gagliardo--Nirenberg inequalities obtained through cardinal B-spline lifting. Conversely, for even orders $k\ge 4$, ENO null modes, on which ENO dissipation vanishes, yield counterexamples built from alternating Euler-polynomial blocks that disprove the conjecture. 
 This dichotomy extends to quasi-uniform meshes, but for every $k\ge 2$,  the conjectured coercivity can fail on non-quasi-uniform meshes. 
Addressing the above second open question, we discover on ENO null modes that local entropy transfer is governed by the first cohomology of a unipotent shift. Using apolar duality and binary covariants from classical invariant theory, we compute the dimensions of the associated cohomology subspaces and prove that smooth local entropy transfer encounters generic obstructions for every order $k\ge 4$. By revealing how ENO null modes link global coercivity and local entropy compatibility, this work provides a structural foundation for the compactness and convergence analysis of high-order entropy-stable discretizations.
\end{abstract}

\maketitle

\newpage

\tableofcontents

\section{Introduction}\label{sec:intro}
\phantomsection\label{sec:intro-enotv}

In the theory and approximation of nonlinear hyperbolic conservation laws, entropy serves both as an admissibility principle for weak solutions and as a foundation for convergence analysis. 
Consider the Cauchy problem for the scalar conservation law
\begin{equation}\label{eq:intro-conservation-law}
	u_t+f(u)_x=0
\end{equation}
with $u(\cdot,0)=u_0\in L^\infty(\mathbb R)$ and $f\in C^1(\mathbb R)$. Even for smooth initial data, solutions can develop shocks in finite time. Although the weak formulation accommodates discontinuities satisfying the Rankine--Hugoniot jump relations, conservation alone does not ensure uniqueness. Kruzhkov's theory \cite{Kruzkov} singles out the unique bounded entropy solution through a family of convex entropy inequalities. Entropy further provides quantitative compactness under suitable flux nondegeneracy assumptions \cite{LPT,TadmorTao}, controls measure-valued limits \cite{DiPernaMV}, and guarantees weak--strong uniqueness for systems with a uniformly convex entropy \cite{BLS}.

At the discrete level, a scheme is termed \emph{entropy-stable} if it satisfies a consistent cellwise entropy inequality for a prescribed convex entropy. Yet uniform boundedness under mesh refinement does not ensure strong compactness, nor does weak convergence permit passage to the limit across the nonlinear flux. For high-order schemes, a central question is whether entropy dissipation can provide the quantitative control needed for convergence without compromising accuracy in smooth regions. A natural candidate is the dissipation generated by classical essentially non-oscillatory (ENO) reconstruction. Introduced by Harten, Engquist, Osher, and Chakravarthy \cite{Harten}, ENO addresses a fundamental problem of nonlinear approximation: retaining high-order accuracy in smooth regions while suppressing Gibbs-type spurious oscillations near discontinuities whose locations are not known a priori. At each cell, it reconstructs a polynomial on a stencil selected adaptively through comparisons of divided differences, favoring locally smoother stencils.

The connection between ENO approximation and discrete entropy stability rests on the \emph{sign property} proved by Fjordholm, Mishra, and Tadmor (FMT) \cite{FMT}: the interface jump produced by cell-average reconstruction or point-value interpolation cannot have the opposite sign from the corresponding jump in the input data. Their framework \cite{FMTtecno}, termed TeCNO, combines high-order entropy-conservative fluxes with dissipation generated by ENO reconstruction in entropy variables to construct entropy-stable schemes of arbitrarily high order for hyperbolic systems. While this sign property ensures nonnegative local entropy dissipation for a prescribed entropy, yielding a semidiscrete cellwise entropy inequality at arbitrary reconstruction order, it does not quantify how strongly this dissipation controls cell-to-cell variation. 
Two fundamental open questions thus arise in the convergence analysis of entropy-stable ENO schemes:
\begin{enumerate}[
	label=\textup{(Q\arabic*)},
	leftmargin=*,
	topsep=0.4em,
	itemsep=0.4em,
	parsep=0pt
	]
	\item\label{q:coercivity}
	{\em Does ENO entropy dissipation provide the quantitative
	coercivity needed for weak-BV estimates?}
	
	\item\label{q:entropy-transfer}
	{\em Can entropy stability extend from the prescribed entropy
	pair to additional pairs without modifying the underlying scheme?}
\end{enumerate}

Question~\ref{q:coercivity} is posited by the long-standing \emph{ENO--TV conjecture} \cite{FjordholmThesis,FjordholmHandbook}, which is a key \emph{weak-BV estimate} expected in 
the convergence framework of Fjordholm, K\"appeli, Mishra, and Tadmor
\cite{FKMT}. To formulate the conjecture, consider scalar ENO reconstruction of order $k\ge2$ on the unit grid $x_{i+1/2}=i+1/2$. Let $\bar u=(\bar u_i)_{i\in\mathbb Z}$ be a compactly supported sequence of cell averages, and write
$a_i:=\bar u_{i+1}-\bar u_i$. 
The reconstruction produces a polynomial $p_i^{(k)}$ of degree at most $k-1$ in each cell. Denote the jump between its interface traces by
$\Delta p_{i+1/2}^{(k)}
:=p_{i+1}^{(k)}(x_{i+1/2})-p_i^{(k)}(x_{i+1/2})$.
The FMT sign property \cite{FMT} gives
\begin{equation}\label{eq:intro-sign-property}
	a_i\Delta p_{i+1/2}^{(k)}\ge0
	\qquad (i\in\mathbb Z).
\end{equation}
The associated \emph{ENO dissipation functional} is
\begin{equation}\label{eq:intro-Qk}
	Q_k(\bar u):=\sum_{i\in\mathbb Z}
	a_i\Delta p_{i+1/2}^{(k)} \ge 0.
\end{equation}
The ENO--TV conjecture \cite{FjordholmThesis,FjordholmHandbook} asserts that, for each $k\ge2$, there exists a constant $C_k>0$, depending only on $k$, such that
\begin{equation}\label{eq:full-enotv-intro}
	\sum_{i\in\mathbb Z}|a_i|^{k+1}
	\le C_k\|\bar u\|_{\ell^\infty}^{k-1}Q_k(\bar u)
\end{equation}
for every compactly supported sequence $\bar u$.   The second-order case was established in \cite[Section~4.4]{FjordholmHandbook} and used in the convergence analysis of \cite{CFM2AN}. Before the present work, all higher orders $k\ge3$ remained open.

The challenges of the conjecture~\eqref{eq:full-enotv-intro} stem from the very properties that make ENO effective for high-order shock capturing: adaptive stencil selection and polynomial exactness. Because ENO selects its stencils through successive comparisons of neighboring divided differences, the reconstruction depends nonlinearly on the cell averages. Indeed, even the local FMT sign property~\eqref{eq:intro-sign-property} was established only two decades after ENO was introduced~\cite{Harten,FMT}. The conjecture goes further, asking whether these nonnegative interface contributions collectively control cell-to-cell variation. 
Polynomial exactness introduces an additional obstacle: because order-$k$ ENO reconstruction reproduces every polynomial of degree at most $k-1$ exactly, the reconstructed interface jumps vanish on polynomial data even though the cell jumps $a_i$ may not. The very property that ensures high-order accuracy therefore allows variation with no local dissipation. The central challenge is thus to establish global coercivity despite this local degeneracy and the nonlinear stencil selection.

In this paper, we resolve the conjecture~\eqref{eq:full-enotv-intro} completely by establishing a sharp parity dichotomy: the estimate holds for every odd $k\ge3$, whereas it fails for every even $k\ge4$. Together with the known second-order case, this completes the classification. 
\emph{Our proof rests on a key discovery: a two-sided equivalence between ENO dissipation and a canonical weighted finite-difference functional}. Specifically, writing $(\mathrm Db)_i:=b_{i+1}-b_i$ for the forward difference and defining
\begin{equation}\label{eq:intro-localized-functional}
	A_j^{(k)}:=\max_{0\le s<k}|a_{j+s}|,
	\qquad
	\mathcal E_k(a):=\sum_j A_j^{(k)}|(\mathrm D^{k-1}a)_j|,
\end{equation}
we prove that\footnote{We write $X\lesssim_{\mathcal A}Y$ (or $Y\gtrsim_{\mathcal A}X$) if
	$X\le C_{\mathcal A}Y$ for a constant depending only on $\mathcal A$,
	and $X\asymp_{\mathcal A}Y$ if both bounds hold.}
\begin{equation}\label{eq:intro-selection-free-source}
	Q_k(\bar u)\asymp_k\mathcal E_k(a),
\end{equation}
uniformly over all stencil patterns selected by the ENO rule, with comparison constants depending only on $k$. This equivalence eliminates stencil selection from the coercivity estimate and reduces the conjecture to a discrete interpolation problem on the integer lattice.

With this equivalence established, the parity dichotomy emerges from the highest finite difference. For odd orders $k=2m-1\ge3$, repeated summation by parts reveals a hidden quadratic energy, $\|\mathrm D^{m-1}a\|_{\ell^2}^2$, controlled by $\mathcal E_k(a)$. A novel discrete Gagliardo--Nirenberg inequality, obtained through cardinal B-spline lifting, combines this energy with the amplitude of $\bar u$ to establish the conjectured global estimate. 
Conversely, for even orders $k\ge4$, the highest difference has odd order, rendering this symmetric energy identity unavailable. Exploiting polynomial exactness and Euler-polynomial identities, we construct compactly supported counterexamples that disprove the conjecture~\eqref{eq:full-enotv-intro} for all even orders. This parity dichotomy extends to quasi-uniform meshes. Yet for every $k\ge2$, the conjectured coercivity does not hold, in general, on non-quasi-uniform meshes.

The ENO interpolation selects the same stencils as the ENO reconstruction applied to the same input values, and we prove that their dissipation functionals are comparable by positive constants depending only on $k$.  Consequently, the same parity classification and counterexamples apply to the ENO--TV conjecture for ENO interpolation.

Polynomial exactness also leads to the second question~\ref{q:entropy-transfer}: can entropy stability extend to additional entropies without modifying the underlying scheme? We study this \emph{entropy-transfer} problem for the scalar finite-difference TeCNO scheme on a uniform grid, equipped with the quadratic design entropy $u^2/2$. This problem is naturally tested on the order-$k$ \emph{ENO null modes}, defined as nodal sequences whose $k$th forward differences vanish at every grid index (Definition~\ref{def:point-value-null-mode}). On these modes, ENO dissipation vanishes and the quadratic entropy balance is exact, thereby isolating the compatibility of the fixed entropy-conservative flux with other entropies.

This compatibility problem exhibits an intrinsic algebraic structure: the \emph{first cohomology of the cell shift} characterizes the terms in the entropy balance near a constant state that no local change of entropy flux can cancel. Using classical invariant theory, we compute the dimensions of the resulting obstruction spaces and identify generic obstructions for all $k\ge4$. Under the assumptions detailed below, these obstructions prevent a smooth extension of the exact quadratic balance to cellwise inequalities for both signs of a small entropy perturbation throughout a common neighborhood of the constant state. Here, the admissible numerical entropy fluxes are assumed to be consistent and translation-invariant on a fixed finite stencil.

Together, these results decouple three distinct mathematical requirements: nonnegative local dissipation, global coercivity, and compatibility with additional entropies. For even orders $k\ge4$, dissipation is strictly positive on nonzero compactly supported data, yet the ENO--TV estimate fails. For odd orders $k\ge5$, coercivity holds, while the entropy-transfer obstruction persists. Polynomial exactness provides a unifying framework for both problems: splicing ENO null-mode blocks probes the strength of dissipation, while translating their entropy balances tests compatibility. This bridges nonlinear approximation and discrete interpolation with algebraic obstruction theory, suggesting a broader paradigm for analyzing high-order methods through the data on which numerical dissipation vanishes.

The even-order failure of the ENO--TV conjecture~\eqref{eq:full-enotv-intro} does not obstruct the convergence program for high-order entropy-stable ENO schemes, including the TeCNO framework of Fjordholm, Mishra, and Tadmor~\cite{FMTtecno,FjordholmThesis}. In the companion paper~\cite{LiWuCompanion}, we build on the localization developed here to obtain two weaker coercive estimates valid at every order, extending the reach of adaptive dissipation beyond the classical ENO--TV conjecture. These estimates establish the compactness framework for arbitrarily high-order schemes, proving convergence to the entropy solution in the scalar case and subsequential convergence to entropy measure-valued solutions for systems.

We now present the three principal results of this paper.

\subsection{A parity dichotomy on uniform grids}
\phantomsection\label{sec:intro-parity}

Our first theorem states the uniform-grid classification for cell-average reconstruction, with the notation $a_i=\bar u_{i+1}-\bar u_i$, $\Delta p_{i+1/2}^{(k)}$, and $Q_k$ introduced above.

\begin{lettertheorem}[Uniform-grid parity dichotomy]\label{thm:A}
	For each integer $k\ge2$, estimate~\eqref{eq:full-enotv-intro} holds for
	every compactly supported sequence of cell averages on the unit grid if
	and only if $k=2$ or $k$ is odd. More precisely, for $k=2$ and
	every odd $k\ge3$,
	\begin{equation}\label{eq:intro-odd-main}
		\sum_i|a_i|^{k+1}
		\le
		C_k\|\bar u\|_{\ell^\infty}^{k-1}Q_k(\bar u),
	\end{equation}
	where the constant $C_k>0$ is independent of the data $\bar u$. For every
	even $k\ge4$, there exist compactly supported data $\bar u^{(L)}$ such
	that
	\begin{equation}\label{eq:intro-even-ratio}
		\frac{\displaystyle\sum_i|a_i^{(L)}|^{k+1}}
		{\displaystyle
			\|\bar u^{(L)}\|_{\ell^\infty}^{k-1}Q_k(\bar u^{(L)})}
		\longrightarrow\infty
		\qquad\text{as }L\to\infty.
	\end{equation}
\end{lettertheorem}

For the same input values, ENO interpolation selects the same stencils as ENO reconstruction, and their dissipation functionals are comparable by positive constants depending only on $k$ (Proposition~\ref{prop:point-value-comparison}).  Consequently, the conclusions in Theorem \ref{thm:A} also apply to the ENO--TV conjecture for ENO interpolation (Corollary~\ref{cor:point-value-parity}).

\subsection{Mesh geometry and global control}

Our next theorem distinguishes global comparability of cell widths from local comparability among neighboring cells. Let $h_i>0$ be the cell widths of a mesh covering the real line, and let $\Delta p_{i+1/2}^{(k,h)}$ denote the order-$k$ reconstructed jump on that mesh. Retaining $a_i:=\bar u_{i+1}-\bar u_i$, we define the ENO dissipation functional by
\begin{equation}\label{eq:intro-nonuniform-source}
	Q_{k,h}(\bar u):=\sum_i a_i\Delta p_{i+1/2}^{(k,h)},
\end{equation}
where each summand is nonnegative by the sign property. For $\Lambda\ge1$, the mesh is $\Lambda$-quasi-uniform if all its cell widths are comparable by the global factor $\Lambda$: there exists $h_*>0$ such that
\[
h_*\le h_i\le\Lambda h_*\qquad(i\in\mathbb Z).
\]

\begin{lettertheorem}[Mesh dependence of ENO coercivity]
	\label{thm:mesh-main}
	Fix an integer $k\ge2$.
	\begin{enumerate}[label=\textup{(\roman*)},leftmargin=2.2em]
		\item The coercivity estimate holds if and only if $k=2$ or
		$k$ is odd: for every $\Lambda\ge1$, there exists a constant
		$C_{k,\Lambda}>0$ such that
		\begin{equation}\label{eq:intro-quasi-uniform-main}
			\sum_i|a_i|^{k+1}
			\le
			C_{k,\Lambda}
			\|\bar u\|_{\ell^\infty}^{k-1}
			Q_{k,h}(\bar u)
		\end{equation}
		for every $\Lambda$-quasi-uniform mesh and every compactly
		supported sequence of cell averages.
		
		\item For every integer $R\ge1$ and every $\Lambda>1$, there exists
		a fixed mesh on which the largest and smallest widths in every
		group of $2R+1$ consecutive cells differ by at most a factor
		$\Lambda$, together with a uniformly bounded sequence of compactly
		supported cell-average data $\bar u^{(n)}$ such that
		\[
		Q_{k,h}(\bar u^{(n)})\longrightarrow0,
		\qquad
		\inf_n\sum_i|a_i^{(n)}|^p>0
		\quad\text{for every }2\le p<\infty.
		\]
	\end{enumerate}
\end{lettertheorem}

Part~\textup{(i)} preserves the parity classification uniformly over each
quasi-uniform mesh class. Part~\textup{(ii)} shows that controlling cell
sizes within any fixed number of neighboring cells is insufficient, even
when the comparison factor is arbitrarily close to one. Local geometric
regularity alone therefore cannot guarantee global control of the jumps.
This counterexample does not imply failure on every non-quasi-uniform mesh.

\subsection{Compatibility with additional entropies and shift cohomology}
\phantomsection\label{sec:intro-cohomology}

For entropy transfer, we consider the scalar finite-difference TeCNO scheme
for nodal values $u_i$, constructed for the quadratic design entropy
$\eta_2(u)=u^2/2$ using the entropy-conservative flux of \cite{FMTtecno}. On a uniform grid of spacing $h>0$, fix a constant
state $u_*$ and a node $i$, and consider ENO null modes
\[
u_{i+\ell}=u_*+P(\ell),\qquad \deg P\le k-1.
\]
The reference constant state corresponds to $P=0$. On these data, ENO dissipation vanishes identically,
and shifting the base index from $i$ to $i+1$ replaces $P$ by
\begin{equation}\label{eq:intro-profile-shift}
	(TP)(t):=P(t+1).
\end{equation}
Any translation-invariant numerical entropy flux then takes the form $H(P)$. Its difference across cell $i$ is precisely $H\circ T-H$.

Expanding the entropy balance in the coefficients of $P$ near zero, the quotient modulo flux differences of the form $H\circ T-H$ defines the first cohomology of the shift at each Taylor degree (with degree referring to the coefficients of $P$, independently of $k$). A residual term can be cancelled by a local flux change if and only if its cohomology class vanishes. We compute the spaces spanned by the classes arising as the physical flux and the locally strictly convex entropy vary.

\begin{lettertheorem}[Shift cohomology and entropy transfer]
	\label{thm:cohomology-main}
	Fix a constant state $u_*$.
	\begin{enumerate}[label=\textup{(\roman*)},leftmargin=2.2em]
		\item For $k=2,3$, the obstruction spaces vanish identically at
		every Taylor degree. For every $k\ge4$, they vanish through degree
		four and have dimension $\lfloor(D-3)/2\rfloor$ at each degree
		$D\ge5$.
		
		\item Let $k\ge4$, let $f,\phi$ be smooth near $u_*$, and set
		\[
		\eta_\theta(u)=\frac12u^2+\theta\phi(u),
		\qquad q_\theta'=\eta_\theta'f'.
		\]
		Assume that these entropies remain strictly convex on a common interval
		around $u_*$ for all sufficiently small $|\theta|$. Keeping the
		underlying scheme fixed, if
		\[
		f''(u_*)\phi^{(4)}(u_*)\ne0,
		\]
		there is no family of numerical entropy fluxes $\mathcal Q^\theta$,
		jointly smooth in $\theta$ and the coefficients of $P$, translation-invariant and supported on a fixed finite stencil, that satisfies
		\begin{equation}\label{eq:intro-transfer-inequality}
			\frac{d}{dt}\eta_\theta(u_i)
			+\frac{\mathcal Q^\theta_{i+1/2}-\mathcal Q^\theta_{i-1/2}}{h}
			\le0
		\end{equation}
		for both signs of sufficiently small $\theta$ and all $P$ in a
		fixed neighborhood of zero, with equality there at $\theta=0$.
		The fluxes are required to be consistent with $q_\theta$, meaning
		that $\mathcal Q^\theta(u,\ldots,u)=q_\theta(u)$.
	\end{enumerate}
\end{lettertheorem}

The connection between the cohomology and the inequality~\eqref{eq:intro-transfer-inequality} stems from the exact balance at $\theta=0$: for a fixed ENO null mode, the left-hand side of~\eqref{eq:intro-transfer-inequality} attains a maximum of zero at $\theta=0$. Its derivative with respect to $\theta$ must therefore vanish, yet the non-trivial degree-five cohomology class prevents this cancellation. This obstruction concerns a smooth family through the exact quadratic balance. It does not exclude an inequality constructed separately for an isolated target entropy. Theorems~\ref{thm:cohomology-dim}, \ref{thm:complete-mismatch-rank}, and \ref{thm:no-two-sided-transfer} give the full algebraic and regularity statements.

\subsection{Organization of the paper}

Section~\ref{sec:eno-source} derives the localized representation of the cell-average ENO dissipation functional and characterizes its ENO null modes. Sections~\ref{sec:odd-coercivity} and~\ref{sec:even-obstruction} establish the uniform-grid parity classification, concluded by a numerical verification of the explicit fourth-order counterexamples. Section~\ref{sec:mesh-geometry} examines the role of mesh geometry, proving the persistence of coercivity on quasi-uniform meshes and its failure under purely local cell-width control. Section~\ref{sec:point-value-enotv} transfers the resolution of ENO--TV conjecture to the ENO interpolation by comparing the two ENO dissipations.  Turning to multi-entropy compatibility, Section~\ref{sec:shift-cohomology} formulates the cancellation of entropy residuals on ENO null modes as a problem in shift cohomology, while Section~\ref{sec:entropy-transfer} computes the resulting obstruction classes via invariant theory and develops their consequences for entropy transfer. Section~\ref{sec:conclusion} concludes with broader implications and perspectives for future work, and the appendices contain the technical mesh constructions and apolar calculations.

\section{Localized ENO dissipation and ENO null modes}
\label{sec:eno-source}

We begin by analyzing $Q_k$ as an ENO reconstruction functional. We derive
a localized representation that is uniform over all selected stencil
patterns, and subsequently characterize the ENO null modes.

\subsection{ENO reconstruction and stencil selection}
\label{sec:eno-reconstruction}

For $d\ge0$, let $\mathbb P_d$ denote the real vector space of univariate
polynomials of degree at most $d$. 
 The symbols $\mathbf1_E$ and $*$ denote
the indicator of $E$ and convolution, respectively.

Throughout this section, we work on the unit grid
\[
x_{i+1/2}=i+\frac12,
\qquad
I_i=(x_{i-1/2},x_{i+1/2}),
\qquad i\in\mathbb Z.
\]
Fix a reconstruction order $k\ge2$ and let
$(\bar u_i)_{i\in\mathbb Z}$ be a compactly supported sequence of cell
averages. For a sequence
$v=(v_i)$, set
\begin{equation}\label{eq:forward-difference}
	\mathrm Dv_i:=v_{i+1}-v_i,
	\qquad
	a_i:=\mathrm D\bar u_i,
	\qquad
	U:=\|\bar u\|_{\ell^\infty}.
\end{equation}
Throughout, $\mathrm D$ denotes the forward difference operator.

The divided differences for cell averages are defined by
\[
[\bar u_j]:=\bar u_j
\]
and, for $m>j$, by
\begin{equation}\label{eq:cell-average-divided-difference}
	[\bar u_j,\ldots,\bar u_m]
	:=
	\frac{[\bar u_{j+1},\ldots,\bar u_m]
		-[\bar u_j,\ldots,\bar u_{m-1}]}
	{x_{m+1/2}-x_{j-1/2}}.
\end{equation}
On the unit grid, these divided differences satisfy the exact normalization
\begin{equation}\label{eq:divided-difference-normalization}
	[\bar u_j,\ldots,\bar u_{j+\ell}]
	=\frac{\mathrm D^\ell\bar u_j}{(\ell+1)!}
	=\frac{\mathrm D^{\ell-1}a_j}{(\ell+1)!},
	\qquad \ell\ge1.
\end{equation}

We use the standard ENO recursion of \cite{FMT}. Let $r_i^{(\ell)}$ denote the left
endpoint of the $\ell$-cell stencil associated with cell $i$. Starting with
\begin{equation}\label{eq:eno-initial-endpoint}
	r_i^{(1)}=i,
\end{equation}
suppose that $r_i^{(\ell)}=r$ has been chosen for some $1 \le \ell < k$.
At level $\ell$, the ENO procedure compares the two candidate divided differences
\begin{equation}\label{eq:eno-candidates}
	L_i^{(\ell)}
	:=[\bar u_{r-1},\ldots,\bar u_{r+\ell-1}],
	\qquad
	R_i^{(\ell)}
	:=[\bar u_r,\ldots,\bar u_{r+\ell}],
\end{equation}
and sets
\begin{equation}\label{eq:eno-endpoint-recursion}
	r_i^{(\ell+1)}
	=
	\begin{cases}
		r_i^{(\ell)}-1,
		& |L_i^{(\ell)}|<|R_i^{(\ell)}|,\\[2mm]
		r_i^{(\ell)},
		& |L_i^{(\ell)}|\ge |R_i^{(\ell)}|.
	\end{cases}
\end{equation}
By \eqref{eq:divided-difference-normalization}, this choice depends only on the absolute values of two
adjacent entries of $\mathrm D^{\ell-1}a$. The final stencil is
\begin{equation}\label{eq:final-eno-stencil}
	\mathcal S_i^{(k)}
	:=\{r_i^{(k)},r_i^{(k)}+1,\ldots,r_i^{(k)}+k-1\},
\end{equation}
and the reconstruction polynomial $p_i^{(k)}$ on cell $I_i$ is the unique polynomial of degree at most $k-1$ such that
\begin{equation}\label{eq:reconstruction-cell-averages}
	\int_{I_j}p_i^{(k)}(x)\,dx=\bar u_j,
	\qquad \forall j\in\mathcal S_i^{(k)}.
\end{equation}

For later use, an \emph{admissible order-$k$ ENO stencil for cell $i$}
refers to any consecutive $k$-cell stencil containing cell $i$. Equivalently,
it is a set of the form
\[
\mathcal S_{i,\nu}^{(k)}
:=
\{i-\nu,i-\nu+1,\ldots,i-\nu+k-1\},
\qquad 0\le\nu\le k-1.
\]
The selected stencil $\mathcal S_i^{(k)}$ is one member of this family.

We represent cell-average reconstruction through nodal interpolation of
a discrete primitive by introducing
\begin{equation}\label{eq:discrete-primitive}
	V_{n+1/2}:=\sum_{m=-\infty}^{n}\bar u_m.
\end{equation}
Since $\bar u$ is compactly supported, these values are well defined. They represent
the interface values of an antiderivative of the piecewise-constant function
equal to $\bar u_n$ on $I_n$. In particular,
\[
V_{n+1/2}-V_{n-1/2}=\bar u_n.
\]
Let $P_i^{(k)}$ be the polynomial of degree at most $k$ interpolating $V$ at
the $k+1$ boundary points
\begin{equation}\label{eq:primitive-stencil}
	x_{r_i^{(k)}-1/2},x_{r_i^{(k)}+1/2},\ldots,
	x_{r_i^{(k)}+k-1/2}.
\end{equation}
We call this set of boundary points the \emph{interpolation stencil for
	$V$}. Since
\begin{equation}
	\int_{I_j}(P_i^{(k)})'(x)\,dx=\bar u_j,
	\qquad \forall j\in\mathcal S_i^{(k)},
\end{equation}
uniqueness of the reconstruction gives
\begin{equation}\label{eq:reconstruction-as-derivative}
	p_i^{(k)}=(P_i^{(k)})'.
\end{equation}

At the interface $x_{i+1/2}$, write
\[
p_{i+1/2}^{(k),-}:=p_i^{(k)}(x_{i+1/2}),
\qquad
p_{i+1/2}^{(k),+}:=p_{i+1}^{(k)}(x_{i+1/2})
\]
and set
\begin{equation}\label{eq:reconstructed-interface-jump}
	\Delta p_{i+1/2}^{(k)}
	:=p_{i+1/2}^{(k),+}-p_{i+1/2}^{(k),-}.
\end{equation}
The symbol $\Delta p$ is reserved for this reconstructed interface jump. The
sign property of Fjordholm, Mishra, and Tadmor \cite{FMT} states that
\begin{equation}\label{eq:fmt-sign-recalled}
	a_i\,\Delta p_{i+1/2}^{(k)}\ge0
	\qquad\text{for every }i\in\mathbb Z.
\end{equation}
Hence
\begin{equation}\label{eq:eno-source-section2}
	Q_k(\bar u)
	:=\sum_{i\in\mathbb Z}a_i\,\Delta p_{i+1/2}^{(k)}
\end{equation}
is nonnegative. 

While the sign property~\eqref{eq:fmt-sign-recalled} ensures that each interface
contribution to $Q_k$ is nonnegative, it does not quantify its magnitude
in terms of local finite differences of the data. To establish a quantitative
localized representation, we first analyze how the stencil
endpoints selected in adjacent cells are related.

\subsection{Stencil endpoints and the induced index partition}
\label{sec:eno-index-partition}

For $1\le\ell\le k$, define the half-open index interval
\begin{equation}\label{eq:selected-index-interval}
	B_i^{(\ell)}
	:=[r_i^{(\ell)},r_{i+1}^{(\ell)})\cap\mathbb Z,
\end{equation}
which may be empty.

\begin{lemma}[Monotonicity and nesting of stencil endpoints]
	\label{lem:endpoint-partition}
	For every $1\le\ell\le k$, the map
	$i\mapsto r_i^{(\ell)}$ is nondecreasing, and
	the intervals $B_i^{(\ell)}$ form a partition of $\mathbb Z$, with
	empty intervals allowed. Moreover,
	\begin{equation}\label{eq:endpoint-bounds}
		i-\ell+1\le r_i^{(\ell)}\le i,
	\end{equation}
	and
	\begin{equation}\label{eq:index-interval-locality}
		j\in B_i^{(\ell)}
		\quad\Longrightarrow\quad
		0\le i-j\le\ell-1.
	\end{equation}
	For $1\le\ell<k$, the intervals at consecutive levels satisfy
	\begin{equation}\label{eq:index-interval-nesting}
		B_i^{(\ell+1)}
		\subset B_i^{(\ell)}\cup\{r_i^{(\ell)}-1\}.
	\end{equation}
	For the same range of $\ell$, if $B_i^{(\ell)}$ is empty, then
	$B_i^{(\ell+1)}$ is likewise empty.
\end{lemma}

\begin{proof}
	The assertions are immediate at level $\ell=1$, where
	$r_i^{(1)}=i$ and $B_i^{(1)}=\{i\}$. Suppose that
	$i\mapsto r_i^{(\ell)}$ is nondecreasing. Each endpoint either stays fixed
	or moves one unit to the left. If
	$r_i^{(\ell)}<r_{i+1}^{(\ell)}$, these two possible updates cannot reverse
	their order. If the endpoints agree, the two adjacent cells face the same
	pair of candidates in~\eqref{eq:eno-candidates}. Hence the recursion~\eqref{eq:eno-endpoint-recursion} produces the same update in both
	cells. Thus $i\mapsto r_i^{(\ell+1)}$ is again nondecreasing.
	
	An endpoint can move to the left at most once at each stage, which proves
	\eqref{eq:endpoint-bounds}. In particular,
	$r_i^{(\ell)}\to\pm\infty$ as $i\to\pm\infty$. Monotonicity then shows that
	the consecutive half-open intervals in~\eqref{eq:selected-index-interval}
	cover every integer exactly once, with some intervals possibly empty.
	Furthermore, if $j\in B_i^{(\ell)}$, then
	\[
	j\ge r_i^{(\ell)}\ge i-\ell+1,
	\qquad
	j<r_{i+1}^{(\ell)}\le i+1.
	\]
	Since the indices are integers, these inequalities give
	\eqref{eq:index-interval-locality}.
	
	To prove the nesting property~\eqref{eq:index-interval-nesting}, put $L=r_i^{(\ell)}$ and
	$R=r_{i+1}^{(\ell)}$. At the next level,
	the new left endpoint lies in $\{L-1,L\}$ and the new right endpoint lies in
	$\{R-1,R\}$. Hence every integer in the new half-open interval either
	already belongs to $[L,R)$ or is the single possible new point $L-1$. This
	proves~\eqref{eq:index-interval-nesting}. If $L=R$, the two endpoints face
	the same comparison and undergo the same update, so the new interval remains
	empty.
\end{proof}

For $j\in\mathbb Z$, let
\[
\mathcal T_j^{(k)}
:=
\{x_{j-1/2},x_{j+1/2},\ldots,x_{j+k-1/2}\}
\]
denote the interpolation stencil for $V$ with left endpoint
$x_{j-1/2}$. The stencils used to construct $P_i^{(k)}$ and
$P_{i+1}^{(k)}$ are $\mathcal T_{r_i^{(k)}}^{(k)}$ and
$\mathcal T_{r_{i+1}^{(k)}}^{(k)}$, respectively. By
Lemma~\ref{lem:endpoint-partition},
$r_i^{(k)}\le r_{i+1}^{(k)}$. Each $j\in B_i^{(k)}$ therefore labels one
elementary shift
\[
\mathcal T_j^{(k)}
\longrightarrow
\mathcal T_{j+1}^{(k)}.
\]
Successive shifts carry the stencil defining $P_i^{(k)}$ to the
stencil defining $P_{i+1}^{(k)}$. If $B_i^{(k)}$ is empty, the two stencils
coincide. The intermediate stencils are introduced only for this
algebraic telescoping and need not themselves be selected by the ENO
recursion. Telescoping over these elementary shifts yields the localized formula for
$\Delta p_{i+1/2}^{(k)}$ established below.

\subsection{A localized form of the FMT identity}
\label{sec:localized-fmt}

The sign property and an accompanying upper bound were established in
\cite[Theorem~1.1]{FMT}. Telescoping the elementary adjacent-stencil shifts
described above yields the exact localized representation below, whose
termwise sign property is guaranteed by \cite[Lemma~2.2]{FMT}.

For later use, we set
\begin{equation}\label{eq:q-level-differences}
	q_j^{(\ell)}:=\mathrm D^{\ell-1}a_j,
	\qquad 1\le\ell\le k,
\end{equation}
and, for $0\le d\le k-1$, define
\begin{equation}\label{eq:localized-fmt-coefficient}
	\gamma_{k,d}
	:=\frac1{k!}
	\prod_{\substack{0\le q\le k-1\\q\ne d}}(d-q)
	=\frac{(-1)^{k-1-d}d!(k-1-d)!}{k!}.
\end{equation}

\begin{proposition}[Localized FMT dissipation formula]
	\label{prop:localized-fmt-source}
	For every $k\ge2$ and every interface $i+1/2$,
	\begin{equation}\label{eq:fmt-localized-expansion}
		\Delta p_{i+1/2}^{(k)}
		=\sum_{j\in B_i^{(k)}}
		\gamma_{k,i-j}\,\mathrm D^{k-1}a_j.
	\end{equation}
	Every summand satisfies the termwise sign inequality
	\begin{equation}\label{eq:fmt-termwise-sign}
		a_i\,\gamma_{k,i-j}\,\mathrm D^{k-1}a_j\ge0,
		\qquad j\in B_i^{(k)}.
	\end{equation}
	Consequently, with
	\begin{equation}\label{eq:fmt-positive-weight}
		w_{i,j}^{(k)}
		:=|\gamma_{k,i-j}|
		=\frac{1}{k\binom{k-1}{i-j}},
		\qquad j\in B_i^{(k)},
	\end{equation}
	the ENO dissipation has the representation
	\begin{equation}\label{eq:fmt-interval-source}
		Q_k(\bar u)
		=\sum_i|a_i|\sum_{j\in B_i^{(k)}}
		w_{i,j}^{(k)}|\mathrm D^{k-1}a_j|.
	\end{equation}
	In particular, if
	\begin{equation}\label{eq:interval-functional}
		S_k(a)
		:=\sum_i|a_i|\sum_{j\in B_i^{(k)}}
		|\mathrm D^{k-1}a_j|,
	\end{equation}
	then
	\begin{equation}\label{eq:interval-source-comparison}
		c_kS_k(a)\le Q_k(\bar u)\le C_kS_k(a)
	\end{equation}
	for some positive constants $c_k, C_k$ depending only on $k$.
\end{proposition}

\begin{proof}
	For each $j\in\mathbb Z$, let $\Pi_j\in\mathbb P_k$ interpolate $V$ on
	$\mathcal T_j^{(k)}$. Then
	\[
	P_i^{(k)}=\Pi_{r_i^{(k)}},
	\qquad
	P_{i+1}^{(k)}=\Pi_{r_{i+1}^{(k)}}.
	\]
	For distinct nodes $z_0,\ldots,z_r$, the ordinary nodal divided differences are
	$V[z_0]:=V(z_0)$ and
	\[
	V[z_0,\ldots,z_r]
	:=\frac{V[z_1,\ldots,z_r]-V[z_0,\ldots,z_{r-1}]}{z_r-z_0}
	\qquad(r\ge1),
	\]
	with
	$V(x_{j+1/2})=V_{j+1/2}$.
	For two consecutive interpolation stencils, the Newton formula gives
	\begin{equation}\label{eq:adjacent-interpolant-shift}
		\Pi_{j+1}(x)-\Pi_j(x)
		=
		(k+1)
		V[x_{j-1/2},\ldots,x_{j+k+1/2}]
		\prod_{q=0}^{k-1}
		\bigl(x-x_{j+q+1/2}\bigr).
	\end{equation}
	
	If $j\in B_i^{(k)}$ and $d=i-j$, then
	$0\le d\le k-1$ by
	\eqref{eq:index-interval-locality}. Differentiating
	\eqref{eq:adjacent-interpolant-shift} at
	$x_{i+1/2}=x_{j+d+1/2}$ and using
	\[
	V[x_{j-1/2},\ldots,x_{j+k+1/2}]
	=
	\frac{\mathrm D^{k-1}a_j}{(k+1)!}
	\]
	give
	\[
	\bigl(\Pi_{j+1}'-\Pi_j'\bigr)(x_{i+1/2})
	=
	\gamma_{k,d}\,\mathrm D^{k-1}a_j.
	\]
	Telescoping over
	$j=r_i^{(k)},\ldots,r_{i+1}^{(k)}-1$
	proves \eqref{eq:fmt-localized-expansion}.
	
	If $a_i\ne0$, \cite[Lemma~2.2]{FMT} shows that every summand in
	\eqref{eq:fmt-localized-expansion} has the sign of $a_i$. If $a_i=0$,
	the termwise inequality \eqref{eq:fmt-termwise-sign} is automatic.
	Consequently, in either case,
	\[
	a_i\Delta p_{i+1/2}^{(k)}
	=
	|a_i|
	\sum_{j\in B_i^{(k)}}
	|\gamma_{k,i-j}|\,|\mathrm D^{k-1}a_j|.
	\]
	Summing over $i$ proves \eqref{eq:fmt-interval-source}. Finally,
	\[
	|\gamma_{k,d}|
	=
	\frac{d!(k-1-d)!}{k!}
	=
	\frac1{k\binom{k-1}{d}},
	\qquad 0\le d\le k-1.
	\]
	The finitely many weights are bounded above and below by positive
	constants depending only on $k$, which proves
	\eqref{eq:interval-source-comparison}.
\end{proof}

Representation~\eqref{eq:fmt-interval-source} weights each selected difference
$\mathrm D^{k-1}a_j$ by the corresponding jump magnitude $|a_i|$.
To obtain a functional independent of the selected intervals,
we compare this weight with the amplitude of the length-$k$ block supporting
that difference.

\subsection{Stencil-independent localization of ENO dissipation}
\label{sec:local-amplitude-source}

For $\ell\ge1$, define
\begin{equation}\label{eq:local-block-amplitude}
	A_j^{(\ell)}:=\max_{0\le s<\ell}|a_{j+s}|.
\end{equation}
The locality property~\eqref{eq:index-interval-locality} places $a_i$ inside
this block whenever $j\in B_i^{(\ell)}$. The following lemma shows that
its magnitude also controls the entire block.

\begin{lemma}[Local amplitude bound]
	\label{lem:local-amplitude}
	For every $1\le\ell\le k$, there exists a constant $C_\ell$, depending only on $\ell$, such that
	\begin{equation}\label{eq:two-sided-local-amplitude}
		j\in B_i^{(\ell)}
		\quad\Longrightarrow\quad
		|a_i|\le A_j^{(\ell)}\le C_\ell|a_i|.
	\end{equation}
\end{lemma}

\begin{proof}
	The first inequality follows immediately from
	\eqref{eq:index-interval-locality}, since
	$j\le i\le j+\ell-1$. We prove the second by induction on $\ell$. It
	holds with equality when $\ell=1$. Retain the notation
	$q_j^{(\ell)}=\mathrm D^{\ell-1}a_j$ from
	\eqref{eq:q-level-differences}, and suppose that the assertion holds at some
	level $\ell$ with $1\le\ell<k$, with constant $C_\ell$. Setting
	\[
	L=r_i^{(\ell)},
	\qquad
	R=r_{i+1}^{(\ell)},
	\]
	the interval $B_i^{(\ell+1)}$ is empty if $L=R$. We may therefore assume $L<R$ and
	consider the three possible locations of $j\in B_i^{(\ell+1)}$.
	
	Suppose first that $j=L-1$. The left endpoint has moved, so the first branch
	of \eqref{eq:eno-endpoint-recursion} gives
	\[
	|q_{L-1}^{(\ell)}|<|q_L^{(\ell)}|.
	\]
	Since $L\in B_i^{(\ell)}$, the induction hypothesis controls
	$a_L,\ldots,a_{L+\ell-1}$. Hence
	\begin{equation}\label{eq:q-left-controlled}
		|q_L^{(\ell)}|
		\le \sum_{s=0}^{\ell-1}\binom{\ell-1}{s}|a_{L+s}|
		\le 2^{\ell-1}C_\ell|a_i|.
	\end{equation}
	In the expansion
	\[
	q_{L-1}^{(\ell)}
	=\sum_{s=0}^{\ell-1}
	(-1)^{\ell-1-s}\binom{\ell-1}{s}a_{L-1+s},
	\]
	only $a_{L-1}$ is not already controlled. Using the preceding comparison and
	\eqref{eq:q-left-controlled}, we obtain
	\begin{equation}\label{eq:new-left-amplitude}
		|a_{L-1}|
		\le (2^\ell-1)C_\ell|a_i|.
	\end{equation}
	Together with the induction hypothesis, this controls the length-$(\ell+1)$ block
	beginning at $L-1$.
	
	If $L\le j\le R-2$, then $j,j+1\in B_i^{(\ell)}$. The two overlapping
	blocks controlled by the induction hypothesis cover
	\[
	\{a_j,a_{j+1},\ldots,a_{j+\ell}\},
	\]
	so no new estimate is needed.
	
	It remains to consider $j=R-1$. This index belongs to the new interval only
	when the endpoint for cell $i+1$ remains at $R$. The second branch of
	\eqref{eq:eno-endpoint-recursion} then gives
	\[
	|q_R^{(\ell)}|\le|q_{R-1}^{(\ell)}|.
	\]
	The induction hypothesis for the block beginning at $R-1$ controls
	$a_{R-1},\ldots,a_{R+\ell-2}$ and therefore
	\[
	|q_{R-1}^{(\ell)}|
	\le 2^{\ell-1}C_\ell|a_i|.
	\]
	In the expansion of $q_R^{(\ell)}$, only $a_{R+\ell-1}$ remains uncontrolled. The same
	calculation as at the left endpoint gives
	\begin{equation}\label{eq:new-right-amplitude}
		|a_{R+\ell-1}|
		\le(2^\ell-1)C_\ell|a_i|.
	\end{equation}
	Thus the induction closes. For example, one may take
	$
	C_1=1
	$ and $
	C_{\ell+1}=(2^\ell-1)C_\ell.
	$ 
\end{proof}

We can now eliminate the dependence on selected intervals from the dissipation formula. Setting
\begin{equation}\label{eq:finite-difference-source-functional}
	\mathcal E_k(a)
	:=\sum_{j\in\mathbb Z}
	A_j^{(k)}|\mathrm D^{k-1}a_j|,
\end{equation}
we obtain the following selection-free representation.

\begin{theorem}[Stencil-independent localization of ENO dissipation]
	\label{cor:source-equivalence}
	For every fixed $k\ge2$,
	\begin{equation}\label{eq:source-functional-equivalence}
		Q_k(\bar u)
		\asymp_k
		\sum_i|a_i|\sum_{j\in B_i^{(k)}}
		|\mathrm D^{k-1}a_j|
		\asymp_k
		\sum_j A_j^{(k)}|\mathrm D^{k-1}a_j|
		=\mathcal E_k(a),
	\end{equation}
	where the comparison constants depend only on $k$.
\end{theorem}

\begin{proof}
	The first comparison is \eqref{eq:interval-source-comparison}. By
	Lemma~\ref{lem:endpoint-partition}, the intervals $B_i^{(k)}$ partition
	$\mathbb Z$. Thus each $j$ belongs to a unique nonempty interval, and
	Lemma~\ref{lem:local-amplitude} gives, for the corresponding $i$,
	\[
	|a_i|\le A_j^{(k)}\le C_k|a_i|.
	\]
	Multiplying by $|\mathrm D^{k-1}a_j|$ and summing over $j$, using the
	partition $\{B_i^{(k)}\}_i$, proves the second comparison.
\end{proof}

\subsection{ENO null modes}
\label{sec:polynomial-null-modes}

In the definition and local statements below, $\bar u$ is not assumed
to be compactly supported.

The functional~\eqref{eq:finite-difference-source-functional} identifies a selection-free vanishing condition: its $j$-th summand,
\[
A_j^{(k)}|\D^{k-1}a_j|
=A_j^{(k)}|\D^k\bar u_j|,
\]
depends only on the cell averages
$\bar u_j,\ldots,\bar u_{j+k}$ and vanishes whenever
$\D^k\bar u_j=0$. We now characterize the cell-average sequences for which this
condition holds on consecutive cell windows.

For an interface $i+1/2$, let
\begin{equation}\label{eq:complete-reconstruction-hull}
	\mathcal H_{i,k}:=\{i-k+1,\ldots,i+k\}.
\end{equation}
By the definition of an admissible stencil, $\mathcal H_{i,k}$ is the union
of all admissible order-$k$ ENO stencils for cells $i$ and $i+1$.

\begin{definition}[ENO null modes and blocks for reconstruction]
	\label{def:polynomial-null-mode}
	A bi-infinite sequence of cell averages $\bar u$ is an
	\emph{order-$k$ ENO null mode} if
	\[
	\D^k\bar u_j=0
	\qquad\text{for every }j\in\mathbb Z.
	\]
	Equivalently, its jump sequence $a=\D\bar u$ satisfies
	$\D^{k-1}a=0$ on $\mathbb Z$.
	
	Let $I=\{r,\ldots,s\}$ be a finite integer interval with
	$|I|=s-r+1\ge k+1$. We say that $\bar u$ contains an
	\emph{order-$k$ ENO null-mode block on $I$} if there exists a
	bi-infinite order-$k$ ENO null mode
	$\bar u^{\mathrm{null}}$ such that
	\[
	\bar u_j=\bar u_j^{\mathrm{null}},
	\qquad j\in I.
	\]
	We call the restriction $(\bar u_j)_{j\in I}$ the corresponding
	ENO null-mode block.
\end{definition}

The lower bound $|I|\ge k+1$ is minimal for the null-mode condition to
impose a nontrivial constraint: arbitrary data on at most $k$ cells admit a
degree-$(k-1)$ cell-average interpolant.

Although an ENO null mode is defined on the whole lattice, the
resulting vanishing properties are local. If $\bar u$ contains an ENO null-mode block on $I$, then
\[
A_j^{(k)}|\D^{k-1}a_j|=0
\qquad\text{whenever}\qquad
\{j,\ldots,j+k\}\subset I.
\]
Moreover, Proposition~\ref{prop:polynomial-null-mode} establishes that
\[
\Delta p_{i+1/2}^{(k)}=0
\qquad\text{whenever}\qquad
\mathcal H_{i,k}\subset I.
\]
Thus the data outside $I$ play no role in either interior conclusion.
Since $|\mathcal H_{i,k}|=2k$, an interface satisfying
$\mathcal H_{i,k}\subset I$ can exist only when $|I|\ge2k$. This stronger
condition is imposed only when an interior interface is under consideration,
not in the definition of a block.
The term ``ENO null mode'' refers to this exact vanishing of ENO
dissipation: a bi-infinite mode has zero reconstructed jump and zero
dissipation at every interface, while a block has these properties at
interior interfaces whose complete reconstruction hull lies in the block.
The polynomial representation of these data is a consequence of
Proposition~\ref{prop:polynomial-null-mode} below, rather than an additional
assumption in the definition.

\begin{lemma}[Kernel of a forward difference]
	\label{lem:discrete-polynomial-kernel}
	Let $r\ge0$, let
	$I=\{i_0,i_0+1,\ldots,i_1\}$ contain at least $r+2$ indices, and let
	$a=(a_i)_{i\in I}$ be a real sequence. The following are equivalent:
	\begin{enumerate}[label=\textup{(\roman*)}]
		\item $\D^{r+1}a_j=0$ whenever
		$\{j,\ldots,j+r+1\}\subset I$,
		\item the values $(a_i)_{i\in I}$ are sampled from a polynomial of
		degree at most $r$.
	\end{enumerate}
	When these conditions hold,
	\begin{equation}\label{eq:newton-forward-representation}
		a_i=\sum_{q=0}^{r}(\D^qa)_{i_0}\binom{i-i_0}{q},
		\qquad i\in I.
	\end{equation}
\end{lemma}

\begin{proof}
	For $q\ge1$, put
	\[
	\binom{x}{q}:=\frac{x(x-1)\cdots(x-q+1)}{q!},
	\qquad \binom{x}{0}:=1.
	\]
	Pascal's identity gives
	\begin{equation}\label{eq:newton-basis-difference}
		\D\binom{i-i_0}{q}=\binom{i-i_0}{q-1}.
	\end{equation}
	Hence $\D^{r+1}$ annihilates the sequence on the right of
	\eqref{eq:newton-forward-representation}. Conversely, that sequence has
	the same forward differences of orders $0,\ldots,r$ at $i_0$ as $a$ and
	therefore the same first $r+1$ values. Both sequences satisfy the same
	order-$(r+1)$ recurrence wherever its stencil lies in $I$, so induction
	from the left endpoint proves equality throughout $I$. The Newton
	polynomials in \eqref{eq:newton-forward-representation} have degrees
	$0,\ldots,r$ and form a basis of the polynomial space of degree at most
	$r$.
\end{proof}

\begin{proposition}[Polynomial characterization of ENO null modes]
	\label{prop:polynomial-null-mode}
	Let $I=\{r,\ldots,s\}$ contain at least $k+1$ cells.
The following conditions are equivalent to $(\bar u_j)_{j\in I}$
being an order-$k$ ENO null-mode block:
	\begin{enumerate}[label=\textup{(\roman*)}]
		\item $\D^k\bar u_j=0$ whenever
		$\{j,\ldots,j+k\}\subset I$,
		\item $\bar u_j=q(j)$ on $I$ for some $q\in\mathbb P_{k-1}$,
		\item $a_j=P(j)$ on $\{r,\ldots,s-1\}$ for some
		$P\in\mathbb P_{k-2}$,
		\item there is a unique $p\in\mathbb P_{k-1}$ such that
		\[
		\bar u_j=\int_{j-1/2}^{j+1/2}p(x)\,dx,
		\qquad j\in I.
		\]
	\end{enumerate}
	If $\mathcal H_{i,k}\subset I$, every admissible order-$k$ ENO stencil in
	cells $i$ and $i+1$ reconstructs this same polynomial $p$. Consequently,
	\begin{equation}\label{eq:null-mode-reconstruction-vanishing}
		p_i^{(k)}=p_{i+1}^{(k)}=p,
		\qquad \Delta p_{i+1/2}^{(k)}=0,
	\end{equation}
	independently of stencil comparisons and of the tie convention.
\end{proposition}

\begin{proof}
	Applying Lemma~\ref{lem:discrete-polynomial-kernel} to $\bar u$ with
	parameter $k-1$ gives the equivalence of
	\textup{(i)} and~\textup{(ii)}. Since
	$\D^{k-1}a=\D^k\bar u$, applying the same lemma to $a$ on
	$\{r,\ldots,s-1\}$ with parameter $k-2$ gives the
	equivalence of \textup{(i)} and~\textup{(iii)}.
	
	Define the unit-cell averaging operator
	\[
	\mathcal A:\mathbb P_{k-1}\longrightarrow\mathbb P_{k-1},
	\qquad
	(\mathcal A p)(x)
	:=\int_{x-1/2}^{x+1/2}p(y)\,dy
	=\int_{-1/2}^{1/2}p(x+z)\,dz.
	\]
	For every $0\le m\le k-1$,
	\[
	(\mathcal A x^m)(x)
	=
	x^m+
	\sum_{q=1}^{m}\binom{m}{q}x^{m-q}
	\int_{-1/2}^{1/2}z^q\,dz.
	\]
	Thus $\mathcal A x^m$ equals $x^m$ plus terms of lower degree.
	In the monomial basis $1,x,\ldots,x^{k-1}$, the matrix of
	$\mathcal A$ is therefore triangular with all diagonal entries equal
	to one. Hence $\mathcal A$ is an isomorphism of
	$\mathbb P_{k-1}$.
	
	Given the polynomial $q$ in~\textup{(ii)}, set
	$p:=\mathcal A^{-1}q$. Then
	\[
	\bar u_j=q(j)
	=\int_{j-1/2}^{j+1/2}p(x)\,dx,
	\qquad j\in I,
	\]
	which proves existence in~\textup{(iv)}. If
	$\widetilde p\in\mathbb P_{k-1}$ also satisfies the identities
	in~\textup{(iv)}, then
	\[
	(\mathcal A\widetilde p-q)(j)=0,
	\qquad j\in I.
	\]
	The polynomial $\mathcal A\widetilde p-q\in\mathbb P_{k-1}$ therefore
	vanishes identically, since $I$ contains at least $k+1$ distinct
	integers. Hence
	$\mathcal A\widetilde p=q=\mathcal A p$, and the invertibility of
	$\mathcal A$ gives $\widetilde p=p$.
	Conversely, any $p$ satisfying~\textup{(iv)} yields
	$q:=\mathcal A p\in\mathbb P_{k-1}$ with
	$\bar u_j=q(j)$ for every $j\in I$, which proves~\textup{(ii)}.
	This proves the equivalence of~\textup{(ii)} and~\textup{(iv)}, including
	the asserted uniqueness.
	
	Now let $\mathcal S$ be an admissible order-$k$ ENO stencil for either
	cell $i$ or cell $i+1$. Since
	$\mathcal S\subset\mathcal H_{i,k}\subset I$, the polynomial $p$
	satisfies all $k$ reconstruction conditions on $\mathcal S$.
	These conditions determine a polynomial in $\mathbb P_{k-1}$
	uniquely. Indeed, suppose that $w\in\mathbb P_{k-1}$ has zero
	integral on the $k$ consecutive cells of $\mathcal S$, and choose
	$W\in\mathbb P_k$ with $W'=w$. Then $W$ takes the same value at the
	$k+1$ interfaces bounding those cells. Subtracting this common value
	from $W$ gives a polynomial of degree at most $k$ with $k+1$ distinct
	roots. It therefore vanishes identically, and hence $w=0$.
	
	Therefore every admissible stencil under consideration reconstructs
	the same polynomial $p$. In particular,
	\[
	p_i^{(k)}=p_{i+1}^{(k)}=p,
	\]
	and consequently $\Delta p_{i+1/2}^{(k)}=0$.
\end{proof}

An ENO null-mode block can be nonzero on a finite interval, but the
only compactly supported bi-infinite ENO null mode is zero.
Indeed, Proposition~\ref{prop:polynomial-null-mode}
implies that such a mode has the form
$\bar u_j=q(j)$ on all of $\mathbb Z$ for one
$q\in\mathbb P_{k-1}$. Compact support would force $q$ to vanish at
infinitely many integers, and hence $q\equiv0$. The next lemma
elevates this fact to the strict positivity of $Q_k$ on every
nonzero compactly supported datum. It also ensures that the quotients in
Section~\ref{sec:even-obstruction} have positive denominators.

\begin{lemma}[Positivity on nonzero compactly supported data]
	\label{lem:source-positivity}
	Fix $k\ge2$. Let $a$ be a nonzero finitely supported sequence with
	$\sum_i a_i=0$, and define
	\[
	\bar u_i:=\sum_{q<i}a_q.
	\]
	Then $Q_k(\bar u)>0$.
\end{lemma}

\begin{proof}
	By assumption, $\bar u$ is finitely supported and $\mathrm D\bar u=a$.
	If $Q_k(\bar u)=0$, Theorem~\ref{cor:source-equivalence} gives
	$\mathcal E_k(a)=0$. In every term of
	\eqref{eq:finite-difference-source-functional}, either the full length-$k$
	block of $a$ is zero or $A_j^{(k)}>0$. Hence
	$\D^{k-1}a_j=0$ for every $j$. Choose a finite interval containing the
	support of $a$ whose first $k-1$ entries lie to the left of the support.
	The corresponding initial forward differences vanish, and
	\eqref{eq:newton-forward-representation}, with $r=k-2$, gives $a=0$, a
	contradiction.
\end{proof}

ENO null modes play two distinct roles in the remainder of the paper. 
In Sections~\ref{sec:odd-coercivity}--\ref{sec:point-value-enotv}, they guide the coercivity analysis: for second and odd orders, discrete interpolation overcomes their local degeneracy to yield global coercivity, whereas in the even-order counterexamples, concatenating long ENO null-mode blocks confines ENO dissipation strictly to boundary and junction regions. Subsequently, in Sections~\ref{sec:shift-cohomology} and~\ref{sec:entropy-transfer}, ENO null modes serve as the testing ground for multi-entropy compatibility in the TeCNO schemes.

\section{Uniform-grid coercivity at second and odd orders}
\label{sec:odd-coercivity}

Theorem~\ref{cor:source-equivalence} decouples the analytical estimate
from the data-dependent stencil selection. For an odd order $k=2m-1$, where
$m\ge2$ is an integer, it remains to prove the purely discrete estimate
\begin{equation}\label{eq:odd-analytic-target}
	\sum_i|\mathrm Dv_i|^{2m}
	\le C_m\|v\|_{\ell^\infty}^{2m-2}
	\mathcal E_{2m-1}(\mathrm Dv)
\end{equation}
for every finitely supported sequence $v$.

The proof has two steps. Set $a=\mathrm Dv$. We first
bound $\sum_i|a_i|^{2m}$ in terms of the squared $\ell^2$ norm of the
$(m-1)$st difference of $a$, and then control this squared norm
by the localized functional:
\begin{align}
	& \sum_i|a_i|^{2m}
	\le C_m\|v\|_{\ell^\infty}^{2m-2}
	\sum_i|\mathrm D^{m-1}a_i|^2
	&&\text{(Step 1: interpolation)},
	\label{eq:odd-interpolation-step}\\
	& \sum_i|\mathrm D^{m-1}a_i|^2
	\le \mathcal E_{2m-1}(a)
	&&\text{(Step 2: dissipation control)}.
	\label{eq:odd-source-control-step}
\end{align}
Together, these inequalities give \eqref{eq:odd-analytic-target}. Step~1
is a discrete Gagliardo--Nirenberg inequality. The restriction to odd orders
enters in Step~2: since $k-1=2(m-1)$, repeated summation by parts converts the
highest difference in the localized functional into the nonnegative quantity
on the left of \eqref{eq:odd-source-control-step}. 

We treat the case $k=2$ separately before turning to the general
odd-order argument. That argument begins with $m=2$, corresponding
to the first previously open order $k=3$.

\subsection{The second-order case}
\label{sec:second-order}

The second-order estimate was established in
\cite[(2.16)]{CFM2AN}. See also
\cite[Section~4.4]{FjordholmHandbook}. Here, the localization theorem provides
an immediate proof: since $a=\mathrm D\bar u$ is finitely supported, summation by parts
combined with the elementary scalar inequality
\[
\bigl|x|x|-y|y|\bigr|
\le (|x|+|y|)|x-y|
\]
yields
\begin{align*}
	\sum_i|a_i|^3
	&=\sum_i(\mathrm D\bar u_i)a_i|a_i|
	=-\sum_i\bar u_{i+1}\mathrm D(a_i|a_i|)\\
	&\le \|\bar u\|_{\ell^\infty}
	\sum_i(|a_i|+|a_{i+1}|)|\mathrm Da_i|
	\le2U\mathcal E_2(a)
	\lesssim UQ_2(\bar u).
\end{align*}

\subsection{The interpolation step: a discrete Gagliardo--Nirenberg inequality}
\label{sec:discrete-gagliardo}

The interpolation inequality needed for higher odd orders depends only
on sequences and their finite differences, independently of ENO
reconstruction.

\begin{theorem}[Discrete Gagliardo--Nirenberg inequality]
	\label{thm:discrete-GN}
	Let $m\ge2$. There exists a constant $C_m>0$, depending only on $m$, such that every finitely supported sequence $v=(v_i)_{i\in\mathbb Z}$
	satisfies
	\begin{equation}\label{eq:discrete-GN}
		\|\mathrm Dv\|_{\ell^{2m}}^{2m}
		\le C_m\|v\|_{\ell^\infty}^{2m-2}
		\|\mathrm D^mv\|_{\ell^2}^{2}.
	\end{equation}
\end{theorem}

To establish this estimate, we employ a cardinal B-spline lift. Define
\begin{equation}\label{eq:cardinal-bspline-definition}
	B_0:=\mathbf 1_{[0,1)},
	\qquad
	B_M:=B_{M-1}*B_0
	\quad(M\ge1).
\end{equation}
Then $B_M$ is a nonnegative, degree-$M$ piecewise polynomial supported in
$[0,M+1]$, whose integer translates form a partition of unity. Its weak
derivative satisfies
\begin{equation}\label{eq:bspline-properties}
	\begin{aligned}
		\sum_{j\in\mathbb Z}B_M(x-j)=1 \qquad 
		&\text{for a.e. }x\quad(M\ge0),\\
		B_M'(x)=B_{M-1}(x)-B_{M-1}(x-1) \qquad 
		&(M\ge1).
	\end{aligned}
\end{equation}
These standard properties follow directly from the convolution definition. See,
e.g., \cite{Schumaker}.

The proof relies on the following $\ell^p$-stability of these B-spline translates.

\begin{lemma}[Stability of cardinal B-spline translates]
	\label{lem:bspline-stability}
	Let $M\ge0$ and $1\le p<\infty$. There exist constants
	$0<c_{M,p}\le C_{M,p}<\infty$, depending only on $M$ and $p$, such that every finitely supported sequence
	$c=(c_j)$ satisfies
	\begin{equation}\label{eq:bspline-stability}
		c_{M,p}\|c\|_{\ell^p}
		\le
		\left\|\sum_jc_jB_M(\,\cdot-j)\right\|_{L^p(\mathbb R)}
		\le C_{M,p}\|c\|_{\ell^p}.
	\end{equation}
\end{lemma}

\begin{proof}
	Set
	\[
	F(x):=\sum_jc_jB_M(x-j).
	\]
	The upper bound follows from the uniformly bounded overlap of the compactly
	supported translates. For the lower bound, let $I_n=(n,n+1)$. Exactly
	the translates with indices $j=n-M,\ldots,n$ are nonzero on $I_n$, and
	their restrictions are linearly independent (see, e.g., \cite{Schumaker}).
	Finite-dimensional norm equivalence yields
	\[
	\sum_{j=n-M}^{n}|c_j|^p
	\le C_{M,p}\int_{I_n}|F(x)|^p\,dx.
	\]
	Summing over $n\in\mathbb Z$, each coefficient $|c_j|^p$ appears $M+1$ times,
	which completes the proof of the lower bound.
\end{proof}

\begin{proof}[Proof of Theorem~\ref{thm:discrete-GN}]
	For a finitely supported sequence $v$, define its spline lift by
	\begin{equation}\label{eq:spline-lift}
		S_v(x):=\sum_jv_jB_m(x-j).
	\end{equation}
	Positivity and the partition of unity give
	\begin{equation}\label{eq:spline-lift-linfty}
		\|S_v\|_{L^\infty(\mathbb R)}
		\le\|v\|_{\ell^\infty}.
	\end{equation}
	By repeated application of the derivative identity in~\eqref{eq:bspline-properties}, we have, for $0\le r\le m$,
	\begin{equation}\label{eq:spline-derivative-difference}
		S_v^{(r)}(x)
		=\sum_j\mathrm D^rv_j\,B_{m-r}(x-j-r).
	\end{equation}
	
	Applying the lower bound in Lemma~\ref{lem:bspline-stability} to
	$S_v'$ yields
	\begin{equation}\label{eq:first-difference-via-spline}
		\|\mathrm Dv\|_{\ell^{2m}}
		\le C_m\|S_v'\|_{L^{2m}(\mathbb R)}.
	\end{equation}
	For $r=m$, the integer translates of $B_0$ have disjoint unit
	supports up to null sets, and hence
	\begin{equation}\label{eq:top-difference-via-spline}
		\|S_v^{(m)}\|_{L^2(\mathbb R)}
		=\|\mathrm D^mv\|_{\ell^2}.
	\end{equation}
	The one-dimensional Gagliardo--Nirenberg inequality
	\cite{Nirenberg}, applied first to smooth compactly supported functions
	and then extended by density, gives
	\begin{equation}\label{eq:continuous-GN}
		\|S_v'\|_{L^{2m}(\mathbb R)}^{2m}
		\le
		C_m\|S_v\|_{L^\infty(\mathbb R)}^{2m-2}
		\|S_v^{(m)}\|_{L^2(\mathbb R)}^2.
	\end{equation}
	Since the spline $S_v$ is compactly supported and belongs to
	$W^{m,2}(\mathbb R)$, this continuous inequality is applicable.
	Combining it with
	\eqref{eq:spline-lift-linfty},
	\eqref{eq:first-difference-via-spline}, and
	\eqref{eq:top-difference-via-spline} proves
	\eqref{eq:discrete-GN}.
\end{proof}

It remains to establish the second estimate,
\eqref{eq:odd-source-control-step}.

\subsection{The dissipation bound and the general odd-order estimate}
\label{sec:odd-general}

We write
\begin{equation}\label{eq:odd-order-parameters}
	k=2m-1\ge3,
	\qquad m\ge2.
\end{equation}
The highest-order finite difference in $\mathcal E_k(a)$ is
$\mathrm D^{2m-2}a$. Since $2m-2=2(m-1)$, repeated summation by parts splits
this even-order difference symmetrically to produce
$\|\mathrm D^{m-1}a\|_{\ell^2}^2$.

\begin{lemma}[Even-difference energy identity]
	\label{lem:even-difference-energy}
	Let $s\ge1$ and let $a$ be finitely supported. Then
	\begin{equation}\label{eq:even-difference-energy}
		\sum_i|\mathrm D^sa_i|^2
		=(-1)^s\sum_i a_{i+s}\mathrm D^{2s}a_i.
	\end{equation}
	Consequently,
	\begin{equation}\label{eq:energy-controlled-by-source-functional}
		\sum_i|\mathrm D^sa_i|^2
		\le\sum_i A_i^{(2s+1)}|\mathrm D^{2s}a_i|
		=\mathcal E_{2s+1}(a).
	\end{equation}
\end{lemma}

\begin{proof}
	For finitely supported sequences, the discrete summation-by-parts formula reads
	\begin{equation}\label{eq:discrete-summation-by-parts}
		\sum_i(\mathrm Df_i)g_i
		=-\sum_i f_{i+1}\mathrm Dg_i.
	\end{equation}
	Iterating \eqref{eq:discrete-summation-by-parts} $s$ times gives
	\[
	\sum_i(\mathrm D^sa_i)h_i
	=(-1)^s\sum_i a_{i+s}\mathrm D^sh_i.
	\]
	Taking $h_i=\mathrm D^sa_i$ proves
	\eqref{eq:even-difference-energy}. The second assertion follows by taking
	absolute values and using the point estimate
	$|a_{i+s}|\le A_i^{(2s+1)}$.
\end{proof}

Combining these two ingredients yields the discrete estimate stated below.

\begin{proposition}[Odd-order discrete interpolation estimate]
	\label{prop:odd-discrete-interpolation}
	Let $m\ge2$, let $v$ be finitely supported, and set $a=\mathrm Dv$. There exists a constant $C_m>0$, depending only on $m$, such that
	\begin{equation}\label{eq:odd-discrete-interpolation}
		\sum_i|a_i|^{2m}
		\le C_m\|v\|_{\ell^\infty}^{2m-2}
		\mathcal E_{2m-1}(a).
	\end{equation}
\end{proposition}

\begin{proof}
	Applying Theorem~\ref{thm:discrete-GN} and using
	$\mathrm D^mv=\mathrm D^{m-1}a$, we obtain
	\begin{equation}\label{eq:odd-intermediate-energy}
		\sum_i|a_i|^{2m}
		\le C_m\|v\|_{\ell^\infty}^{2m-2}
		\sum_i|\mathrm D^{m-1}a_i|^2.
	\end{equation}
	Lemma~\ref{lem:even-difference-energy}, with $s=m-1$, gives
	\[
	\sum_i|\mathrm D^{m-1}a_i|^2
	\le\sum_iA_i^{(2m-1)}|\mathrm D^{2m-2}a_i|
	=\mathcal E_{2m-1}(a).
	\]
	Combining these two bounds proves
	\eqref{eq:odd-discrete-interpolation}.
\end{proof}

\begin{proof}[Proof of the odd-order assertion of
	Theorem~\ref{thm:A}]
	Let $k=2m-1\ge3$ and apply
	Proposition~\ref{prop:odd-discrete-interpolation} with $v=\bar u$.
	Together with Theorem~\ref{cor:source-equivalence}, this yields
	\[
	\sum_i|a_i|^{k+1}
	\le C_kU^{k-1}\mathcal E_k(a)
	\le C_kU^{k-1}Q_k(\bar u),
	\]
	establishing estimate~\eqref{eq:intro-odd-main} and completing the proof of Theorem~\ref{thm:A} for all odd orders.
\end{proof}

The algebraic relation $k-1=2(m-1)$ illuminates the decisive role of parity: it
allows the highest difference in the localized functional to be paired symmetrically into a
nonnegative sum of squares. For even $k$, this symmetric pairing is intrinsically unavailable.
The next section exploits this structural absence to construct explicit compact counterexamples showing that coercivity
fails at every even order $k\ge4$.

\section{Even orders: counterexamples from ENO null-mode blocks}
\label{sec:even-obstruction}

Theorem~\ref{cor:source-equivalence} reduces the even-order assertion of
Theorem~\ref{thm:A} to a discrete construction: we seek a sequence of compactly
supported jump data $a=\mathrm D\bar u$ for which the jump moment
\[
\sum_i|a_i|^{k+1}
\]
is large, while the localized functional
\[
\mathcal E_k(a)
=
\sum_j A_j^{(k)}|\mathrm D^{k-1}a_j|
\]
remains small relative to $\|\bar u\|_{\ell^\infty}^{k-1}$.

ENO null modes provide the starting point for this construction. If the
jump sequence coincides with a polynomial of degree at most $k-2$ on an
interval, then
\[
\mathrm D^{k-1}a_j=0
\]
whenever the forward-difference stencil
$\{j,\ldots,j+k-1\}$ lies entirely within that interval. Its interior may therefore
contribute substantially to the jump moment while contributing nothing to the
localized functional.

Although an infinite ENO null mode has zero ENO dissipation at every
interface, any nonzero polynomial sequence is not compactly supported and therefore
does not directly contradict \eqref{eq:full-enotv-intro}. The construction
must therefore truncate and concatenate long ENO null-mode blocks, controlling the
dissipation at their junctions and outer endpoints while ensuring
three properties:
\begin{enumerate}[label=\textup{(\roman*)},leftmargin=2.2em]
	\item the jump moment accumulates throughout the block interiors,
	\item the discrete primitive $\bar u$ remains compactly supported and
	its amplitude remains bounded independently of the number of blocks,
	\item the dissipation generated at the internal junctions and the two outer
	endpoints remains uniformly controlled.
\end{enumerate}
The discrete symmetry of Euler polynomials provides the exact control required at the
internal junctions.

\subsection{Alternating Euler blocks with controlled junctions}
\label{sec:euler-blocks}

Throughout this section, we construct the jump sequence
\[
a_i=(\mathrm D\bar u)_i=\bar u_{i+1}-\bar u_i.
\]
A polynomial block is therefore an interval on which the values $a_i$
are sampled from a polynomial in the integer index $i$. These block polynomials describe the jump data and are distinct from the
ENO reconstruction polynomials $p_i^{(k)}$.

Once $a$ has been constructed, the corresponding cell-average sequence is
its discrete primitive
\[
\bar u_i=\sum_{q<i}a_q.
\]
For a finitely supported sequence $a$, this primitive is compactly supported
if and only if 
$
\sum_i a_i=0.
$ 

Fix an even order $k\ge4$ and set
\begin{equation}\label{eq:even-order-parameter}
	d:=k-2.
\end{equation}
Then $d\ge2$ is even and $k-1=d+1$. A degree-$d$ block of length $O(L)$
and typical magnitude $O(L^d)$ contributes on the scale
\[
L(L^d)^{k+1}=L^{1+d(k+1)}
\]
to the jump moment. Its partial sums, however, may change $\bar u$ by
$O(L^{d+1})$. Repeating blocks with the same sign would therefore cause
$\|\bar u\|_{\ell^\infty}$ to grow with the number of blocks.

We therefore alternate a block and its negative. Their sums cancel in
pairs, so the cell-average sequence returns to its baseline after each
pair and its amplitude remains bounded independently of the number of pairs. It remains
to control $\mathrm D^{d+1}a$ where consecutive polynomial
blocks meet. While a generic junction could produce a crossing difference of size
$O(L^d)$, the Euler-polynomial identity below ensures that the second block differs
from the polynomial continuation of the first by the fixed monomial
$-2h^d$, where $h$ is the displacement from the junction. This makes
all $(d+1)$st differences crossing an internal junction bounded
independently of $L$. The two outer endpoints, where a block is joined to
the zero extension rather than to its negative, will be estimated separately.

On the space of polynomials of degree at most $d$, the map
\[
P(x)\longmapsto P(x+1)+P(x)
\]
is invertible: in the monomial basis its matrix is triangular with diagonal
entries equal to $2$. Let $E_d$ be the unique polynomial satisfying
\begin{equation}\label{eq:euler-polynomial-identity}
	E_d(x+1)+E_d(x)=2x^d.
\end{equation}
This is the $d$th Euler polynomial. Evaluating
\eqref{eq:euler-polynomial-identity} at $-x$ gives
\[
E_d(-x)+E_d(1-x)=2(-x)^d=2x^d.
\]
Thus the polynomial $E_d(1-x)$ satisfies the same identity and hence, by
uniqueness,
\begin{equation}\label{eq:euler-polynomial-symmetry}
	E_d(1-x)=E_d(x).
\end{equation}
Evaluating \eqref{eq:euler-polynomial-identity} at $x=0$ gives
$E_d(1)+E_d(0)=0$, while
\eqref{eq:euler-polynomial-symmetry} gives $E_d(1)=E_d(0)$. Hence
\begin{equation}\label{eq:euler-polynomial-endpoints}
	E_d(0)=E_d(1)=0.
\end{equation}

For a positive integer $L$, define the rescaled polynomial
\begin{equation}\label{eq:scaled-euler-polynomial}
	P_L(t):=(2L)^dE_d\!\left(\frac{t}{2L}\right).
\end{equation}
It has degree $d$, vanishes at $0$ and $2L$, and has magnitude comparable to
$L^d$ when $t/(2L)$ lies in a fixed compact subset of $(0,1)$ on which
$E_d$ does not vanish. The following identity compares the second,
oppositely signed block with the polynomial continuation of the first,
yielding an $L$-independent bound at their junction.

\begin{lemma}[Euler junction identity and uniform junction bound]
	\label{lem:euler-junction}
	For every positive integer $L$ and every $h\in\mathbb R$,
	\begin{equation}\label{eq:euler-junction}
		P_L(2L+h)+P_L(h)=2h^d.
	\end{equation}
	For $h\in\mathbb Z$, define
	\[
	b_h:=
	\begin{cases}
		P_L(2L+h),&h\le0,\\
		-P_L(h),&h>0.
	\end{cases}
	\]
	Then $b$ is sampled from a polynomial of degree $d$ on each side of $h=0$,
	and
	\begin{equation}\label{eq:euler-junction-difference-support}
		\operatorname{supp}(\mathrm D^{d+1}b)
		\subset\{-d,\ldots,-1\},
		\qquad
		\|\mathrm D^{d+1}b\|_{\ell^\infty}
		\le C_d^{\mathrm{jun}},
	\end{equation}
	where $C_d^{\mathrm{jun}}>0$ depends only on $d$. In particular, it
	is independent of $L$.
\end{lemma}

\begin{proof}
	Substitution in \eqref{eq:euler-polynomial-identity} gives
	\[
	P_L(2L+h)+P_L(h)
	=(2L)^d\left[
	E_d\!\left(1+\frac{h}{2L}\right)
	+E_d\!\left(\frac{h}{2L}\right)
	\right]
	=2h^d.
	\]
	To obtain the difference bound, continue the polynomial $P_L(2L+h)$ across
	$h=0$.
	For $h\ge0$, the piecewise profile differs from the polynomial continuation
	$P_L(2L+h)$ by $-2h^d\mathbf 1_{\{h\ge0\}}$. The operator
	$\mathrm D^{d+1}$ annihilates the continuation. Applied to the truncated
	monomial, it can be nonzero only at the starting indices $-d,\ldots,-1$. On this fixed set, its values are bounded by a constant depending only on $d$. This
	proves \eqref{eq:euler-junction-difference-support}.
\end{proof}

We now assemble the blocks. Let $N$ be a positive even integer. For
$0\le n<N$ and $0\le t<2L$, set
\begin{equation}\label{eq:euler-block-definition}
	a_{2Ln+t}^{L,N}:=(-1)^nP_L(t),
\end{equation}
and set $a_i^{L,N}=0$ outside $0\le i<2LN$. The half-open convention is
compatible with both adjacent polynomials because $P_L(0)=P_L(2L)=0$.
Set $B_n:=2Ln$ for $0\le n\le N$. The indices
$B_1,\ldots,B_{N-1}$ are the points where consecutive blocks meet, whereas
$B_0$ and $B_N$ are the two endpoints of the block interval.
The block sums alternate in sign and have equal absolute value.
Since $N$ is even,
\begin{equation}\label{eq:euler-block-zero-sum}
	\sum_i a_i^{L,N}=0.
\end{equation}
We may therefore define the associated compactly supported cell-average
sequence by
\begin{equation}\label{eq:euler-block-primitive}
	\bar u_i^{L,N}:=\sum_{\ell<i}a_\ell^{L,N}.
\end{equation}
Then $\mathrm D\bar u_i^{L,N}=a_i^{L,N}$.

Figure~\ref{fig:euler-blocks} illustrates the construction. The jump moment in
the ENO--TV estimate accumulates throughout every block. The cell-average
sequence, however, returns to the baseline after each pair of blocks, while
$\mathrm D^{k-1}a^{L,N}$ is confined to fixed-width neighborhoods of the
junctions and the two endpoints.

\begin{figure}[!htb]
	\centering
	\includegraphics[width=0.99\linewidth]{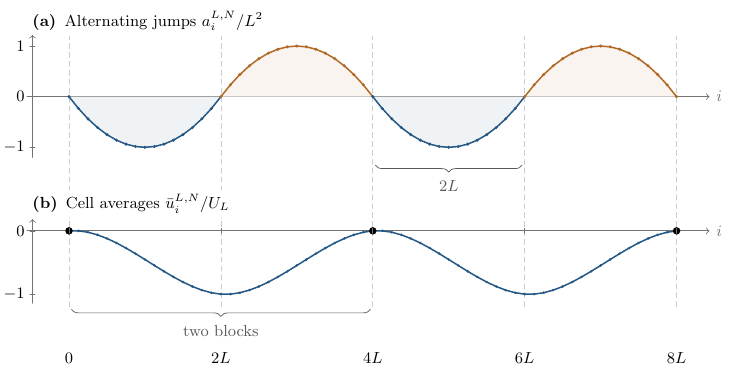}
	\caption{Alternating Euler-polynomial blocks for $k=4$,
		$L=8$, and $N=4$. (a) The jump sequence, scaled by $L^2$.
		(b) Its discrete primitive, scaled by
		$U_L=L(4L^2-1)/3$, returns to zero after each pair of blocks.
		Dots represent integer samples. Connecting lines guide the eye.}
	\label{fig:euler-blocks}
\end{figure}

\subsection{Bounds on jump moments, amplitude, and dissipation}
\label{sec:euler-block-estimates}

We first estimate the $\ell^{k+1}$ jump moment $\sum_i|a_i^{L,N}|^{k+1}$ and the $\ell^\infty$ amplitude of the discrete
primitive $\bar u^{L,N}$, and subsequently bound the contributions to the localized dissipation functional arising from the
internal junctions and the two outer endpoints.

\begin{lemma}[$\ell^{k+1}$ and $\ell^\infty$ bounds]
	\label{lem:euler-block-scales}
	There exist an integer threshold $L_0=L_0(k)$ and positive constants $c_k,C_k$ such that, for every
	$L\ge L_0$ and every positive even integer $N$,
	\begin{align}
		\sum_i|a_i^{L,N}|^{k+1}
		&\ge c_kNL^{1+d(k+1)},
		\label{eq:euler-bulk-scale}\\
		\|\bar u^{L,N}\|_{\ell^\infty}
		&\le C_kL^{d+1}.
		\label{eq:euler-primitive-scale}
	\end{align}
\end{lemma}

\begin{proof}
	Because the Euler polynomial $E_d$ does not vanish identically, it is bounded away from zero on
	some closed interval $J\subset(0,1)$. A fixed positive proportion of
	the indices $t\in[0,2L)$ in each block satisfy $t/(2L)\in J$, and at those indices
	$|P_L(t)|\ge c_k L^d$. Summing over all $N$ blocks proves
	\eqref{eq:euler-bulk-scale}.
	
	Moreover, $|P_L(t)|\le C_k L^d$ on $[0,2L]$, so the partial sums within any single block
	are $O_k(L^{d+1})$. The sums of consecutive, oppositely signed blocks
	cancel in pairs, so the primitive returns to its baseline after each
	pair. This proves \eqref{eq:euler-primitive-scale}.
\end{proof}

The two outer endpoints differ from the internal junctions because the
neighboring profile is the zero extension rather than an oppositely oriented Euler
block. The nearby jump values and the crossing $(d+1)$st differences are
both of order $O_k(L^{d-1})$. Since the number of affected endpoint stencils depends solely
on the order $k$, their total contribution to $\mathcal E_k$ is $O_k(L^{2d-2})$.

\begin{lemma}[Endpoint estimates]
	\label{lem:euler-endpoints}
	For every integer $R\ge1$, there exists a constant $C_{k,R}>0$ such
	that, for every integer $L\ge R$ and every integer $h$ with $0\le h\le R$,
	\begin{equation}\label{eq:euler-endpoint-values}
		|P_L(h)|+|P_L(2L-h)|\le C_{k,R}L^{d-1}.
	\end{equation}
	Moreover, the absolute value of every $(d+1)$st forward difference
	formed from values of $P_L$ and its zero extension across either endpoint is
	at most $C_{k,R}L^{d-1}$, provided its stencil lies in the $R$-neighborhood
	of that endpoint. Here a stencil lies in the $R$-neighborhood of an
	endpoint if every index in the stencil is at distance at most $R$ from that
	endpoint. The constant $C_{k,R}$ is uniform over all such $L$, $h$, and
	stencils.
\end{lemma}

\begin{proof}
	By \eqref{eq:euler-polynomial-endpoints} and the fact that $E_d$ is a fixed
	polynomial,
	\[
	|E_d(x)|\le C_k|x|,
	\qquad
	|E_d(1-x)|\le C_k|x|
	\]
	for $|x|\le1/2$. Inserting $x=h/(2L)$ into
	\eqref{eq:scaled-euler-polynomial} proves
	\eqref{eq:euler-endpoint-values}. A forward difference of fixed order is a
	fixed linear combination of the values in its stencil. The assertion for
	the zero extension follows from the same bound.
\end{proof}

Combining the junction and endpoint estimates yields the following dissipation bound.

\begin{proposition}[Dissipation bound for Euler-polynomial blocks]
	\label{prop:euler-source}
	There exists a constant $C_k>0$ such that, for every
	integer $L\ge d+1$ and every positive even integer $N$,
	\begin{equation}\label{eq:euler-source-functional-bound}
		\mathcal E_k(a^{L,N})
		\le C_k\bigl(NL^{d-1}+L^{2d-2}\bigr).
	\end{equation}
	Consequently,
	\begin{equation}\label{eq:euler-source-bound}
		Q_k(\bar u^{L,N})
		\le C_k\bigl(NL^{d-1}+L^{2d-2}\bigr).
	\end{equation}
\end{proposition}

\begin{proof}
	Set
	\[
	w_j:=\mathrm D^{d+1}a_j^{L,N}.
	\]
	If the forward-difference stencil $\{j,\ldots,j+d+1\}$ is contained in a
	single block, then $w_j=0$ as the block is sampled from a polynomial of
	degree $d$.
	
	Recall that the internal junctions are $B_n=2Ln$,
	$1\le n<N$. For every such $n$ and every integer $h$ with
	$|h|\le d+1$, the definition of the blocks gives the local jump sequence
	\begin{equation}\label{eq:euler-internal-junction-profile}
		a_{B_n+h}^{L,N}
		=(-1)^{n-1}
		\begin{cases}
			P_L(2L+h),&h\le0,\\
			-P_L(h),&h>0,
		\end{cases}
		=(-1)^{n-1}b_h.
	\end{equation}
	Consequently,
	\begin{equation}\label{eq:euler-internal-junction-differences}
		w_{B_n+s}=(-1)^{n-1}\mathrm D^{d+1}b_s,
		\qquad -d\le s\le-1,
	\end{equation}
	and Lemma~\ref{lem:euler-junction} gives
	$|w_{B_n+s}|\le C_d^{\mathrm{jun}}\le C_k$ at these indices. A
	stencil beginning at $B_n-d-1$ merely terminates at the common zero endpoint
	$P_L(2L)=0$ and remains within a single degree-$d$ polynomial piece, so it contributes
	zero. Thus contributions at $B_n$ can arise only from the indices
	$B_n-d,\ldots,B_n-1$.
	Every entry in the corresponding $k$-point window lies
	within $d+1$ sites of the junction, so
	Lemma~\ref{lem:euler-endpoints}, applied with $R=d+1$, gives
	\[
	A_j^{(k)}\le C_kL^{d-1}.
	\]
	There are only $d$ possible indices $j$ at each junction.
	Their total contribution to $\mathcal E_k$ is therefore at most
	$C_kNL^{d-1}$.
	
	At the two outer endpoints, a stencil beginning one site earlier
	than the ranges below merely ends at the common zero endpoint and has vanishing
	$(d+1)$st difference. Stencils farther away lie wholly in one polynomial
	piece or in the zero extension. Thus the only indices that can contribute are
	\begin{equation}\label{eq:euler-outer-endpoint-difference-indices}
		\{B_0-d,\ldots,B_0-1\}
		\quad\text{and}\quad
		\{B_N-d,\ldots,B_N-1\},
	\end{equation}
	respectively. Together with the internal-junction calculation, this gives
	\begin{equation}\label{eq:euler-difference-global-support}
		\operatorname{supp}w
		\subset
		\bigcup_{n=0}^{N}\{B_n-d,\ldots,B_n-1\}.
	\end{equation}
	On the endpoint stencils in
	\eqref{eq:euler-outer-endpoint-difference-indices},
	Lemma~\ref{lem:euler-endpoints}, again with $R=d+1$, gives both
	$|w_j|\le C_kL^{d-1}$ and $A_j^{(k)}\le C_kL^{d-1}$. The two
	endpoint neighborhoods consequently contribute at most
	$C_kL^{2d-2}$. This proves
	\eqref{eq:euler-source-functional-bound}. The upper comparison in
	Theorem~\ref{cor:source-equivalence} then yields
	\eqref{eq:euler-source-bound},
	independently of the data-dependent stencil selections.
\end{proof}

Thus
\[
\sum_i|a_i^{L,N}|^{k+1}
\gtrsim_k NL^{1+d(k+1)},
\]
whereas
\[
\|\bar u^{L,N}\|_{\ell^\infty}
\lesssim_k L^{d+1},\qquad Q_k(\bar u^{L,N})
\lesssim_k NL^{d-1}+L^{2d-2}.
\]
The sum $\sum_i|a_i^{L,N}|^{k+1}$ grows linearly with the number $N$ of blocks because every
block contributes on the bulk scale given in
\eqref{eq:euler-bulk-scale}. By contrast, the cell-average sequence returns
to its baseline after each pair of oppositely signed blocks, so its
amplitude is independent of $N$. The localized functional receives
contributions only from fixed-width neighborhoods of the internal junctions and the two outer
endpoints. We choose $N$ below so that the junction term dominates the endpoint
term in the upper bound.

\subsection{Failure of even-order coercivity}
\label{sec:even-failure}

Letting $L\to\infty$ through the even integers, we set
\begin{equation}\label{eq:euler-number-of-blocks}
	N:=L^d.
\end{equation}
Then $N$ is even, and the junction dissipation term
$NL^{d-1}=L^{2d-1}$ dominates the boundary endpoint term
$L^{2d-2}$. The denominator in~\eqref{eq:euler-divergent-quotient} below is strictly positive by
Lemma~\ref{lem:source-positivity}. Since $k=d+2$, combining estimates
\eqref{eq:euler-bulk-scale},
\eqref{eq:euler-primitive-scale}, and
\eqref{eq:euler-source-bound} yields
\begin{align}
	\frac{\displaystyle\sum_i|a_i^{L,N}|^{k+1}}
	{\displaystyle
		\|\bar u^{L,N}\|_{\ell^\infty}^{k-1}
		Q_k(\bar u^{L,N})}
	\ge
	c_k\frac{NL^{1+d(k+1)}}
	{L^{(d+1)(k-1)}\bigl(NL^{d-1}+L^{2d-2}\bigr)}
	\ge c_kL.
	\label{eq:euler-divergent-quotient}
\end{align}
Because $L$ can be chosen arbitrarily large, no constant depending solely on the reconstruction order $k$ can validate coercivity
estimate~\eqref{eq:intro-odd-main} for even $k\ge4$. This disproves the ENO--TV conjecture at all even orders above two and concludes the proof of the uniform-grid classification in Theorem~\ref{thm:A}.

\subsection{Numerical confirmation of the analytical counterexamples at fourth order}
\label{sec:k4-numerical-experiment}

The preceding analysis establishes the failure of coercivity at every even
order $k\ge4$. We now illustrate this failure at fourth order by prescribing
an explicit family of cell averages, computing the corresponding coercivity
quotients, and deriving exact formulas that explain their growth.

For an integer $C\ge1$, define 
\[
M_C:=200C^2,\quad U_C:=\frac{5}{3} C(100C^2-1), \quad 
\Phi_C(t):=\frac16 {t(t-1)(30C+1-2t)}.
\]
For $0\le i<M_C$, write $i=10Cn+t$, where
$n=\lfloor i/(10C)\rfloor$ and $0\le t<10C$. We prescribe the cell averages
\begin{equation}\label{eq:k4-explicit-cell-averages}
\bar u_i^{(C)}=
\begin{cases}
-\Phi_C(t),&0\le i<M_C,\quad n\ \text{even},\\
\Phi_C(t)-U_C,&0\le i<M_C,\quad n\ \text{odd},\\
0,&i<0\ \text{or}\ i\ge M_C.
\end{cases}
\end{equation}
This is the alternating block construction of
Subsection~\ref{sec:euler-blocks} with $d=2$, $L=5C$, and $N=20C$.
For these compactly supported data, we consider the coercivity quotient
\begin{equation}\label{eq:k4-explicit-ratio}
R_4(C):=
\frac{\displaystyle\sum_i
	|\bar u_{i+1}^{(C)}-\bar u_i^{(C)}|^5}
{\displaystyle\|\bar u^{(C)}\|_{\ell^\infty}^3
	Q_4(\bar u^{(C)})}.
\end{equation}

We compute $R_4(C)$ for $C\in\{1,10,100,1000,10000\}$ by direct
fourth-order ENO reconstruction. The computation
visits every interface within the support and the surrounding zero padding
required by the reconstruction stencils. 
Figure~\ref{fig:k4-numerical-verification} shows that all five computed
quotients exceed $C$ and grow approximately linearly. To examine the
departure from the leading linear behavior, Figure~\ref{fig:k4-numerical-verification} also  plots
$R_4(C)-(90/77)C$. The computed residuals decrease
from $0.240040790$ to $0.175330227$, approaching the reference level
$27/154\approx0.175324675$. Both reference constants are consistent with the
closed-form expression derived below.

\begin{figure}[!ht]
	\centering
	\begin{minipage}[t]{0.49\linewidth}
		\centering
		\includegraphics[width=\linewidth]{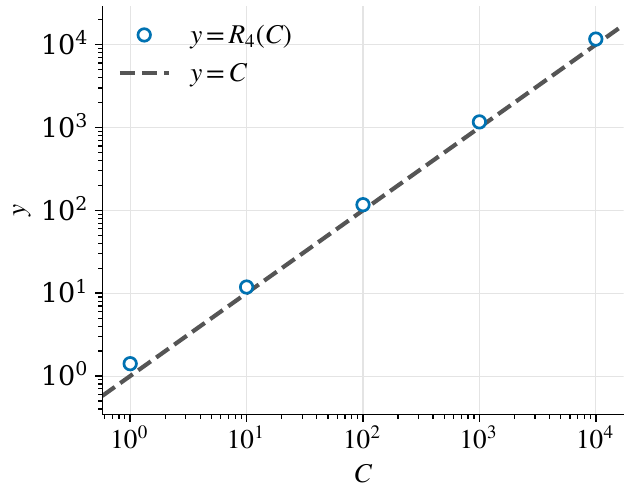}
	\end{minipage}\hfill
	\begin{minipage}[t]{0.49\linewidth}
		\centering
		\includegraphics[width=\linewidth]{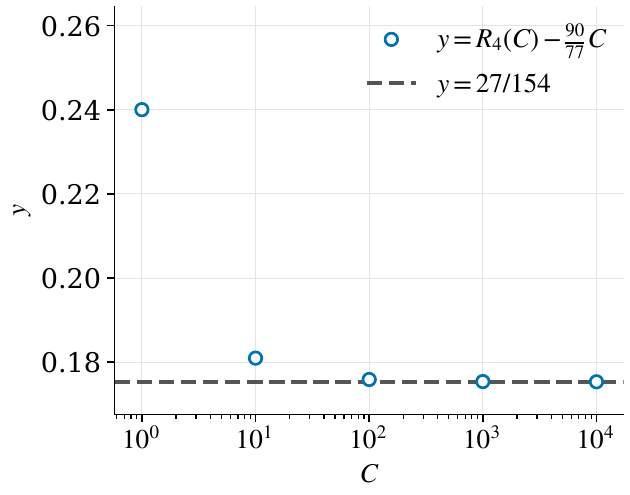}
	\end{minipage}
	\caption{Full-grid fourth-order ENO computations for the cell averages
		\eqref{eq:k4-explicit-cell-averages}, with $C\in\{1,10,100,1000,10000\}$. Circles denote the computed values. Dashed lines denote the analytical references.
		Left: the coercivity quotient $R_4(C)$ compared with $y=C$. Right: the residual $R_4(C)-(90/77)C$ compared with its limiting value
		$y=27/154$.}
	\label{fig:k4-numerical-verification}
\end{figure}

We next derive exact expressions for the quantities entering
\eqref{eq:k4-explicit-ratio}. The identities
\[
\Phi_C(0)=0,\qquad \Phi_C(10C)=U_C,\qquad
\Phi_C(t+1)-\Phi_C(t)=t(10C-t)
\]
imply $\|\bar u^{(C)}\|_{\ell^\infty}=U_C$ and
\[
a_{10Cn+t}^{(C)}
:=\bar u_{10Cn+t+1}^{(C)}-\bar u_{10Cn+t}^{(C)}
=(-1)^n t(t-10C),
\]
with $0\le n<20C$, $0\le t<10C$. The jumps vanish outside these blocks. Since all $20C$ blocks have the
same jump magnitudes, we obtain
\begin{equation}\label{eq:k4-exact-quantities}
J_5(C):=\sum_i|a_i^{(C)}|^5
=20C\sum_{t=0}^{10C-1}\bigl[t(10C-t)\bigr]^5.
\end{equation}

To compute the dissipation, let $B_n:=10Cn$. Applying the ENO recursion
\eqref{eq:eno-endpoint-recursion} with ties resolved to the right yields
\begin{equation}\label{eq:k4-interface-contributions}
a_i^{(C)}\Delta p_{i+1/2}^{(4)}=
\begin{cases}
\displaystyle\frac{(10C-1)(20C-1)}6,
&i\in\{1,M_C-1\},\\[2mm]
\displaystyle\frac{10C-1}6,
&\substack{i\in\{B_n-1,B_n+1\},\\1\le n<20C,}\\[2mm]
0,&\text{otherwise}.
\end{cases}
\end{equation}
The same recursion shows that the reconstructed jump itself vanishes at
every other interface, including the block boundaries $i=B_n$. Thus exactly
$40C$ interfaces make nonzero contributions to the dissipation: two near
the outer endpoints and two adjacent to each of the $20C-1$ internal
junctions. Summing \eqref{eq:k4-interface-contributions} gives
\begin{equation}\label{eq:k4-exact-dissipation}
Q_4(\bar u^{(C)})=\frac{2(10C-1)(20C-1)}3.
\end{equation}
These formulas make the mechanism behind the loss of coercivity explicit:
the jump moment $J_5(C)$ accumulates throughout the polynomial blocks,
whereas the ENO dissipation is confined to interfaces adjacent to the
junctions and outer endpoints.

Finally, set $L=5C$ and
\[
H_L:=128L^8+32L^6+8L^4-42L^2+105.
\]
Evaluating the power sum in \eqref{eq:k4-exact-quantities} gives
\[
\sum_{t=0}^{2L-1}\bigl[t(2L-t)\bigr]^5
=\frac{L(4L^2-1)H_L}{693}.
\]
Combining this identity with $\|\bar u^{(C)}\|_{\ell^\infty}=U_C$ and
\eqref{eq:k4-exact-dissipation}, we obtain
\begin{equation}\label{eq:k4-closed-ratio}
R_4(C)=\frac{18H_L}
{77L(4L^2-1)^2(2L-1)(4L-1)}.
\end{equation}
For each of the five tested values of $C$, this expression agrees exactly,
as a rational number, with the quotient obtained from the full-grid
ENO reconstruction. Since $L\ge5$, we have
\[
H_L>128L^8,\qquad
(4L^2-1)^2(2L-1)(4L-1)<128L^6,
\]
and hence
\begin{equation}\label{eq:k4-linear-growth}
R_4(C)>\frac{18L}{77}=\frac{90}{77}C>C
\qquad\text{for every integer }C\ge1,
\end{equation}
consistent with the numerical results in
Figure~\ref{fig:k4-numerical-verification}.

\section{Mesh geometry and ENO coercivity}
\label{sec:mesh-geometry}

Having established the uniform-grid parity classification in Sections~\ref{sec:odd-coercivity} and~\ref{sec:even-obstruction}, we now investigate the extent to which these coercivity properties persist on nonuniform meshes.

On a general mesh, polynomial exactness and quantitative coercivity decouple. Polynomial reproduction holds purely algebraically without any restriction on cell-width ratios and guarantees that the reconstructed jump vanishes at interior interfaces of ENO null-mode blocks. By contrast, uniform coercivity estimates depend crucially on the underlying metric geometry. We first prove that the parity dichotomy persists on globally quasi-uniform meshes. We then demonstrate that global comparability cannot be relaxed to purely local mesh-ratio bounds over fixed-radius stencils, even when adjacent cell widths vary by a factor arbitrarily close to one.

\subsection{General meshes and polynomial exactness}
\label{sec:nonuniform-reconstruction}

Let $(x_{i+1/2})_{i\in\mathbb Z}$ be a strictly increasing sequence such
that $x_{i+1/2}\to\pm\infty$ as $i\to\pm\infty$, and set
\begin{equation}\label{eq:nonuniform-cells}
	I_i=(x_{i-1/2},x_{i+1/2}),
	\qquad
	h_i=x_{i+1/2}-x_{i-1/2}>0.
\end{equation}
The mesh is $\Lambda$-quasi-uniform if, for some $h_*>0$,
\begin{equation}\label{eq:nonuniform-quasi-uniformity}
	h_*\le h_i\le\Lambda h_*
	\qquad(i\in\mathbb Z).
\end{equation}
We first consider compactly supported cell-average data
$(\bar u_i)_{i\in\mathbb Z}$ and retain the notation
\begin{equation}\label{eq:nonuniform-jumps-amplitude}
	a_i:=\bar u_{i+1}-\bar u_i,
	\qquad
	U:=\|\bar u\|_{\ell^\infty}.
\end{equation}

We introduce the cumulative discrete primitive $(V_{i+1/2})_{i\in\mathbb Z}$ satisfying
\begin{equation}\label{eq:nonuniform-primitive}
	V_{i+1/2}-V_{i-1/2}=h_i\bar u_i.
\end{equation}
These quantities represent the interface values of an antiderivative of the piecewise-constant
function equal to $\bar u_i$ on $I_i$. They are determined up to an
additive constant, which leaves all positive-order divided differences invariant. For $r\ge0$, define the window width
\begin{equation}\label{eq:nonuniform-window-width}
	H_i^{(r)}
	:=x_{i+r+1/2}-x_{i-1/2}
	=\sum_{q=0}^{r}h_{i+q}.
\end{equation}
The cell-average divided differences are given recursively by
\begin{equation}\label{eq:nonuniform-divided-differences}
	D_i^{(0)}:=\bar u_i,
	\qquad
	D_i^{(r)}
	:=\frac{D_{i+1}^{(r-1)}-D_i^{(r-1)}}{H_i^{(r)}},
	\qquad r\ge1.
\end{equation}
Throughout this section, $D_i^{(r)}$ denotes the divided differences
defined by \eqref{eq:nonuniform-divided-differences}, whereas $\mathrm D$
continues to denote the forward-difference operator on sequences.
Starting from
\[
V[x_{i-1/2},x_{i+1/2}]
=\frac{V_{i+1/2}-V_{i-1/2}}{h_i}
=\bar u_i,
\]
induction on $r$ gives
\begin{equation}\label{eq:nonuniform-primitive-divided-difference}
	D_i^{(r)}
	=
	V[x_{i-1/2},x_{i+1/2},\ldots,x_{i+r+1/2}].
\end{equation}
Thus $D_i^{(r)}$ is the divided difference of order $r+1$ of the cumulative
interface data. Here the superscript $r$ indexes the order of difference taken on cell averages,
rather than the number of interface values.

We use the ENO recursion from \cite{FMT}, as in
Subsection~\ref{sec:eno-reconstruction}, now with the divided differences in
\eqref{eq:nonuniform-divided-differences}. Let $r_i^{(\ell)}$ be the left
endpoint of the $\ell$-cell stencil for cell $i$, beginning with
$r_i^{(1)}=i$. If $r_i^{(\ell)}=r$, define
\[
L_{i,h}^{(\ell)}:=D_{r-1}^{(\ell)},
\qquad
R_{i,h}^{(\ell)}:=D_r^{(\ell)},
\]
and, for $1\le\ell<k$, set
\begin{equation}\label{eq:nonuniform-endpoint-recursion}
	r_i^{(\ell+1)}
	=
	\begin{cases}
		r_i^{(\ell)}-1,
		& |L_{i,h}^{(\ell)}|<|R_{i,h}^{(\ell)}|,\\[2mm]
		r_i^{(\ell)},
		& |L_{i,h}^{(\ell)}|\ge |R_{i,h}^{(\ell)}|.
	\end{cases}
\end{equation}
The final stencil is
\[
\mathcal S_{i,h}^{(k)}
:=
\{r_i^{(k)},r_i^{(k)}+1,\ldots,r_i^{(k)}+k-1\},
\]
and adjacent endpoints determine the half-open index intervals
\begin{equation}\label{eq:nonuniform-selected-index-interval}
	B_i^{(\ell)}
	:=
	[r_i^{(\ell)},r_{i+1}^{(\ell)})\cap\mathbb Z.
\end{equation}

Let $\mathbb P_d$ denote the real polynomials of degree at most $d$. The
reconstruction $p_i^{(k,h)}\in\mathbb P_{k-1}$ is the unique polynomial
satisfying
\begin{equation}\label{eq:nonuniform-reconstruction}
	\frac1{h_j}\int_{I_j}p_i^{(k,h)}(x)\,dx=\bar u_j,
	\qquad
	j\in\mathcal S_{i,h}^{(k)}.
\end{equation}
Uniqueness follows from the antiderivative argument of Section~\ref{sec:eno-source}:
if a polynomial in $\mathbb P_{k-1}$ has vanishing cell averages on $k$ consecutive
cells, its antiderivative assumes identical values at $k+1$ distinct
interfaces and is therefore identically constant.

At the interface $x_{i+1/2}$, define the reconstructed point-value traces by
\[
p_{i+1/2}^{-,(k,h)}
:=p_i^{(k,h)}(x_{i+1/2}),
\qquad
p_{i+1/2}^{+,(k,h)}
:=p_{i+1}^{(k,h)}(x_{i+1/2}).
\]
Their difference is the reconstructed interface jump
\begin{equation}\label{eq:nonuniform-reconstructed-jump}
	\Delta p_{i+1/2}^{(k,h)}
	:=
	p_{i+1/2}^{+,(k,h)}
	-
	p_{i+1/2}^{-,(k,h)}.
\end{equation}
The admissible stencil positions depend only on the cell indices.
Consequently, the complete reconstruction hull
\[
\mathcal H_{i,k}
=
\{i-k+1,\ldots,i+k\},
\]
defined in \eqref{eq:complete-reconstruction-hull}, contains every
order-$k$ stencil used for the two cells adjacent to
$x_{i+1/2}$ on a nonuniform mesh as well. We define the corresponding
ENO dissipation functional by
\begin{equation}\label{eq:nonuniform-source}
	Q_{k,h}(\bar u)
	:=
	\sum_{i\in\mathbb Z}
	a_i\Delta p_{i+1/2}^{(k,h)}.
\end{equation}

ENO null-mode blocks extend to general meshes through an equivalent
polynomial characterization. Let $\{r,\ldots,s\}$ contain at least $k+1$
cell indices. We say that the cell averages on this interval form an
order-$k$ ENO null-mode block if they are generated by a polynomial
$p\in\mathbb P_{k-1}$ in the physical coordinate $x$ through
\begin{equation}\label{eq:general-mesh-polynomial-block}
	\bar u_\ell
	=
	\frac1{h_\ell}\int_{I_\ell}p(x)\,dx,
	\qquad r\le\ell\le s.
\end{equation}
Equivalently,
\begin{equation}\label{eq:general-mesh-null-differences}
	D_j^{(k)}=0
	\qquad\text{whenever}\qquad
	r\le j
	\quad\text{and}\quad
	j+k\le s.
\end{equation}
Each divided difference in this condition involves the
$k+1$ cell averages
\[
\bar u_j,\ldots,\bar u_{j+k}.
\]

To verify the equivalence, observe that~\eqref{eq:nonuniform-primitive-divided-difference}
identifies condition~\eqref{eq:general-mesh-null-differences} with the vanishing of all
consecutive divided differences of order $k+1$ of the cumulative interface
values on the block. Interpolating the first $k+2$
interface values by a polynomial $F\in\mathbb P_k$, each one-cell stencil shift
yields an interpolant sharing $k+1$ distinct interface values with $F$. Since both
polynomials have degree at most $k$, they coincide. Iterating this
argument across the entire block produces a single polynomial $F$, and setting $p=F'\in\mathbb P_{k-1}$
yields \eqref{eq:general-mesh-polynomial-block}. The converse follows
immediately by integrating $p$.

If
\[
\mathcal H_{i,k}\subset\{r,\ldots,s\},
\]
every admissible order-$k$ stencil for the two cells adjacent to
$x_{i+1/2}$ reconstructs this same polynomial $p$. Consequently,
\begin{equation}\label{eq:general-mesh-polynomial-exactness}
	\Delta p_{i+1/2}^{(k,h)}=0.
\end{equation}
Thus polynomial exactness and the vanishing of interior interface
dissipation hold unconditionally on arbitrary meshes, independent of any
mesh-ratio assumptions. Such geometric assumptions enter only
in the quantitative estimates below.

We normalize the highest divided difference in a scale-invariant form:
\begin{equation}\label{eq:nonuniform-normalized-difference}
	\widehat D_i^{(k)}
	:=
	\bigl(H_i^{(k)}\bigr)^k|D_i^{(k)}|.
\end{equation}
We shall also use the local jump amplitude
\begin{equation}\label{eq:nonuniform-local-amplitude}
	A_i^{(\ell)}
	:=
	\max_{0\le q\le\ell-1}|a_{i+q}|,
	\qquad \ell\ge1.
\end{equation}

\begin{lemma}[Affine scaling]
	\label{lem:nonuniform-affine-scaling}
	Under the coordinate dilation $y=x/h_*$, the cell averages, the jumps
	$a_i$, the amplitude $U$, the ENO comparisons and selected stencils, the
	reconstructed interface values, the dissipation functional $Q_{k,h}$, and
	the quantities $\widehat D_i^{(k)}$ are unchanged. Consequently, every estimate on a
	$\Lambda$-quasi-uniform mesh may be proved under the normalization
	\begin{equation}\label{eq:normalized-quasi-uniformity}
		1\le h_i\le\Lambda.
	\end{equation}
\end{lemma}

\begin{proof}
	Let the subscripts $x$ and $y$ indicate the coordinate in which a quantity
	is computed. The recursion \eqref{eq:nonuniform-divided-differences} gives
	\[
	D_{i,y}^{(r)}=h_*^rD_{i,x}^{(r)},
	\qquad
	H_{i,y}^{(r)}=\frac{H_{i,x}^{(r)}}{h_*}.
	\]
	At ENO level $r$, both candidates are multiplied by the same positive
	factor $h_*^r$, so every comparison and every selected endpoint is
	unchanged. The polynomial $p_{i,x}^{(k,h)}(h_*y)$ has the prescribed
	averages on the rescaled cells. Uniqueness in
	\eqref{eq:nonuniform-reconstruction} identifies it with the reconstruction
	in the $y$-coordinate. Thus the corresponding interface values and
	$Q_{k,h}$ are unchanged. The remaining normalized quantity satisfies
	\[
	\bigl(H_{i,y}^{(k)}\bigr)^k|D_{i,y}^{(k)}|
	=
	\bigl(H_{i,x}^{(k)}\bigr)^k|D_{i,x}^{(k)}|.
	\]
\end{proof}

The ordering of the selected endpoints and the exact dissipation identity
below hold without a mesh-ratio bound. Quasi-uniformity enters only when
these identities are converted into uniform estimates: first in the geometric
coefficient comparison and local amplitude bound below, and then in the
interpolation construction of
Subsection~\ref{sec:quasi-uniform-interpolation}.

\subsection{Localized ENO dissipation with geometric coefficients}
\label{sec:nonuniform-localized-source}

We begin by establishing the monotonicity and partition properties of the selected stencil endpoints on arbitrary meshes, requiring no bounds on cell-width ratios.

\begin{lemma}[Endpoint monotonicity and the induced index partition]
	\label{lem:nonuniform-endpoint-partition}
	For every strictly increasing mesh $(x_{i+1/2})_{i\in\mathbb Z}$ and each $1\le\ell\le k$, the mapping $i\mapsto r_i^{(\ell)}$ is nondecreasing, and the half-open intervals $B_i^{(\ell)}$ form a partition of $\mathbb Z$, where empty intervals are permitted. Moreover,
	\begin{equation}\label{eq:nonuniform-endpoint-bounds}
		i-\ell+1\le r_i^{(\ell)}\le i,
	\end{equation}
	and
	\begin{equation}\label{eq:nonuniform-interval-locality}
		j\in B_i^{(\ell)}
		\quad\Longrightarrow\quad
		0\le i-j\le\ell-1.
	\end{equation}
	For every $1\le\ell<k$,
	\begin{equation}\label{eq:nonuniform-interval-nesting}
		B_i^{(\ell+1)}
		\subset
		B_i^{(\ell)}\cup\{r_i^{(\ell)}-1\}.
	\end{equation}
	For the same range $1\le\ell<k$, if
	$B_i^{(\ell)}=\varnothing$, then
	$B_i^{(\ell+1)}=\varnothing$.
\end{lemma}

\begin{proof}
	As in the uniform-grid case (Lemma~\ref{lem:endpoint-partition}), the argument relies solely on the fact that each stencil endpoint either remains fixed or moves one cell to the left. When two adjacent cells share an endpoint, they evaluate identical candidate divided differences in~\eqref{eq:nonuniform-endpoint-recursion}, hence undergo the identical update. The proofs of monotonicity, endpoint bounds, partitioning, and nesting proceed without change, independently of mesh spacing.
\end{proof}

The following proposition provides the general-mesh counterpart of
Proposition~\ref{prop:localized-fmt-source}, reformulating the arbitrary-mesh identities of \cite[Lemmas~2.1 and~2.2]{FMT} in terms of cell-based stencils.

\begin{proposition}[\texorpdfstring{Localized FMT identity on a general mesh}{Localized FMT identity on a general mesh}]
	\label{prop:nonuniform-localized-fmt}
	For every $k\ge2$ and every interface $x_{i+1/2}$,
	\begin{equation}\label{eq:nonuniform-fmt-expansion}
		\Delta p_{i+1/2}^{(k,h)}
		=
		\sum_{j\in B_i^{(k)}}
		D_j^{(k)}X_{i,j}^{(k,h)},
	\end{equation}
	where
	\begin{equation}\label{eq:nonuniform-fmt-coefficient}
		X_{i,j}^{(k,h)}
		:=
		H_j^{(k)}
		\prod_{\substack{0\le q\le k-1\\j+q\ne i}}
		\bigl(x_{i+1/2}-x_{j+q+1/2}\bigr).
	\end{equation}
	The selected terms satisfy the termwise sign inequalities
	\begin{equation}\label{eq:nonuniform-fmt-termwise-sign}
		a_iD_j^{(k)}X_{i,j}^{(k,h)}\ge0,
		\qquad j\in B_i^{(k)}.
	\end{equation}
	Consequently,
	\begin{equation}\label{eq:nonuniform-fmt-positive-identity}
		a_i\Delta p_{i+1/2}^{(k,h)}
		=
		|a_i|
		\sum_{j\in B_i^{(k)}}
		\bigl|D_j^{(k)}X_{i,j}^{(k,h)}\bigr|.
	\end{equation}
\end{proposition}

\begin{proof}
	Formula~\eqref{eq:nonuniform-fmt-expansion} is the
	arbitrary-node identity of \cite[Lemma~2.1, (2.5)]{FMT}.
	Moreover,
	\[
	V[x_{i-1/2},x_{i+1/2},x_{i+3/2}]
	=\frac{a_i}{h_i+h_{i+1}},
	\]
	and hence has the sign of $a_i$. The geometric sign formula
	\cite[(2.6)]{FMT} and the divided-difference sign induction
	\cite[Lemma~2.2, (2.7a)--(2.7b)]{FMT} therefore give
	\[
	a_iD_j^{(k)}X_{i,j}^{(k,h)}\ge0,
	\qquad j\in B_i^{(k)}.
	\]
\end{proof}

To extract quantitative coercivity from identity~\eqref{eq:nonuniform-fmt-positive-identity}, we first compare each geometric product $|D_j^{(k)}X_{i,j}^{(k,h)}|$ with the scale-invariant divided difference $\widehat D_j^{(k)}$, and subsequently replace the interface jump $|a_i|$ with the block amplitude $A_j^{(k)}$. Before establishing these comparisons, we record the exact equality between the signed and absolute dissipation functionals.

\begin{corollary}[Signed and absolute dissipation]
	\label{cor:signed-equals-absolute-source}
	For every strictly increasing mesh,
	\[
	Q_{k,h}(\bar u)
	=
	Q_{k,h}^{\mathrm{abs}}(\bar u)
	:=
	\sum_i|a_i|
	\bigl|\Delta p_{i+1/2}^{(k,h)}\bigr|.
	\]
\end{corollary}

\begin{proof}
	Summing the termwise sign identity~\eqref{eq:nonuniform-fmt-positive-identity} over all interfaces yields the claim immediately.
\end{proof}

Quasi-uniformity is first used to compare the geometric factors in
Proposition~\ref{prop:nonuniform-localized-fmt} with the normalized divided
differences $\widehat D_j^{(k)}$.

\begin{corollary}[\texorpdfstring{Localized dissipation comparison on a quasi-uniform mesh}{Localized dissipation comparison on a quasi-uniform mesh}]
	\label{cor:nonuniform-interval-source}
	On a $\Lambda$-quasi-uniform mesh,
	\begin{equation}\label{eq:nonuniform-interval-source-comparison}
		Q_{k,h}(\bar u)
		\asymp_{k,\Lambda}
		\sum_i|a_i|
		\sum_{j\in B_i^{(k)}}\widehat D_j^{(k)}.
	\end{equation}
\end{corollary}

\begin{proof}
	By Lemma~\ref{lem:nonuniform-affine-scaling}, assume
	$1\le h_i\le\Lambda$. If $j\in B_i^{(k)}$, then
	$0\le i-j\le k-1$. The factor omitted from
	\eqref{eq:nonuniform-fmt-coefficient} is therefore the unique zero factor.
	Every remaining distance is a nonempty sum of consecutive cell widths
	contained in the interval from $x_{j-1/2}$ to $x_{j+k+1/2}$. Hence
	\[
	1\le
	\bigl|x_{i+1/2}-x_{j+q+1/2}\bigr|
	\le H_j^{(k)}
	\quad(j+q\ne i),
	\]
	while
	$k+1\le H_j^{(k)}\le(k+1)\Lambda$. Since the product contains $k-1$
	factors,
	\begin{equation}\label{eq:nonuniform-X-comparison}
		\frac{\bigl(H_j^{(k)}\bigr)^k}
		{((k+1)\Lambda)^{k-1}}
		\le
		|X_{i,j}^{(k,h)}|
		\le
		\bigl(H_j^{(k)}\bigr)^k.
	\end{equation}
	Inserting this comparison into
	\eqref{eq:nonuniform-fmt-positive-identity}, summing over $i$, and using
	\eqref{eq:nonuniform-normalized-difference} completes the proof.
\end{proof}

To replace the interface factor $|a_i|$ by the amplitude of the jump block
beginning at $j$, we examine the coefficients in the expansion of
$D_j^{(\ell)}$. We need lower bounds for the coefficients of the first
and last jumps and uniform upper bounds for all coefficients.

\begin{lemma}[Endpoint coefficients in the divided-difference expansion]
	\label{lem:nonuniform-endpoint-coefficients}
	For every $\ell\ge1$ and every $j\in\mathbb Z$,
	\begin{equation}\label{eq:nonuniform-endpoint-expansion}
		D_j^{(\ell)}
		=
		\sum_{r=0}^{\ell-1}
		\beta_{j,r}^{(\ell)}a_{j+r},
	\end{equation}
	where the coefficients depend only on the local mesh widths. Under the quasi-uniform normalization
	$1\le h_i\le\Lambda$, there exist positive constants
	$c_{\ell,\Lambda}$ and $C_{\ell,\Lambda}$, depending only on $\ell$ and $\Lambda$, such that for all $j\in\mathbb Z$,
	\begin{equation}\label{eq:nonuniform-endpoint-coefficient-bounds}
		\min\left\{
		|\beta_{j,0}^{(\ell)}|,
		|\beta_{j,\ell-1}^{(\ell)}|
		\right\}
		\ge c_{\ell,\Lambda},
		\qquad
		\max_{0\le r<\ell}|\beta_{j,r}^{(\ell)}|
		\le C_{\ell,\Lambda}.
	\end{equation}
\end{lemma}

\begin{proof}
	Starting from
	\[
	D_j^{(1)}=\frac{a_j}{H_j^{(1)}},
	\]
	the divided-difference recursion gives
	\[
	\beta_{j,r}^{(\ell+1)}
	=
	\frac{
		\beta_{j+1,r-1}^{(\ell)}
		-
		\beta_{j,r}^{(\ell)}
	}{
		H_j^{(\ell+1)}
	},
	\qquad
	0\le r\le\ell,
	\]
	where coefficients outside their index range are understood to be zero.
	This proves \eqref{eq:nonuniform-endpoint-expansion} by induction.
	For the two endpoint coefficients, the recursion gives
	\[
	\beta_{j,0}^{(\ell)}
	=
	\frac{(-1)^{\ell-1}}
	{\displaystyle\prod_{s=1}^{\ell}H_j^{(s)}},
	\qquad
	\beta_{j,\ell-1}^{(\ell)}
	=
	\frac{1}{
		\displaystyle\prod_{s=1}^{\ell}
		H_{j+\ell-s}^{(s)}}.
	\]
	If $1\le h_i\le\Lambda$, then
	\[
	s+1\le H_q^{(s)}\le(s+1)\Lambda,
	\]
	and hence
	\[
	|\beta_{j,0}^{(\ell)}|,
	|\beta_{j,\ell-1}^{(\ell)}|
	\ge
	\frac{1}{(\ell+1)!\Lambda^\ell}.
	\]
	The same recursion, together with
	$H_j^{(\ell+1)}\ge\ell+2$, gives a uniform upper bound for all
	coefficients by induction on $\ell$. This proves
	\eqref{eq:nonuniform-endpoint-coefficient-bounds}.
\end{proof}

\begin{lemma}[Local amplitude bound on a quasi-uniform mesh]
	\label{lem:nonuniform-local-amplitude-bound}
	For every $1\le\ell\le k$,
	\begin{equation}\label{eq:nonuniform-two-sided-amplitude}
		j\in B_i^{(\ell)}
		\quad\Longrightarrow\quad
		|a_i|\le A_j^{(\ell)}
		\le C_{\ell,\Lambda}|a_i|.
	\end{equation}
\end{lemma}

\begin{proof}
	After affine scaling, assume $1\le h_i\le\Lambda$. If
	$j\in B_i^{(\ell)}$, then
	$j\le i\le j+\ell-1$, which gives
	\[
	|a_i|\le A_j^{(\ell)}.
	\]
	
	We prove the reverse inequality by induction on $\ell$. The case
	$\ell=1$ is immediate. Assume that the assertion holds at level
	$\ell$, where $1\le\ell<k$, and write
	\[
	L:=r_i^{(\ell)},
	\qquad
	R:=r_{i+1}^{(\ell)}.
	\]
	If $L=R$, then $B_i^{(\ell+1)}=\varnothing$ by
	Lemma~\ref{lem:nonuniform-endpoint-partition}, and the assertion is vacuous. Hence assume $L<R$.
	
	Let $j\in B_i^{(\ell+1)}$. By
	\eqref{eq:nonuniform-interval-nesting}, there are three cases.
	
	First, suppose that $L\le j\le R-2$. Then both $j$ and $j+1$
	belong to $B_i^{(\ell)}$. By the induction hypothesis, the two
	overlapping blocks
	\[
	\{a_j,\ldots,a_{j+\ell-1}\},
	\qquad
	\{a_{j+1},\ldots,a_{j+\ell}\}
	\]
	are each bounded by $C_{\ell,\Lambda}|a_i|$. Their union is the
	length-$(\ell+1)$ block beginning at $j$, so
	\[
	A_j^{(\ell+1)}
	\le C_{\ell,\Lambda}|a_i|.
	\]
	
	Next, suppose that $j=L-1$. In this case the left endpoint has shifted,
	and the ENO recursion~\eqref{eq:nonuniform-endpoint-recursion} requires
	\[
	|D_{L-1}^{(\ell)}|<|D_L^{(\ell)}|.
	\]
	Since $L\in B_i^{(\ell)}$, the induction hypothesis and the coefficient
	upper bound in
	Lemma~\ref{lem:nonuniform-endpoint-coefficients} imply
	\[
	|D_L^{(\ell)}|
	\le C_{\ell,\Lambda}|a_i|.
	\]
	In the expansion of $D_{L-1}^{(\ell)}$, the only jump not already
	controlled by the block beginning at $L$ is $a_{L-1}$. Using the
	lower bound for the first endpoint coefficient, we obtain
	\[
	c_{\ell,\Lambda}|a_{L-1}|
	\le
	|D_{L-1}^{(\ell)}|
	+C_{\ell,\Lambda}|a_i|
	\le
	C_{\ell,\Lambda}'|a_i|.
	\]
	Thus the entire length-$(\ell+1)$ block beginning at $L-1$ is
	controlled by $C_{\ell,\Lambda}'|a_i|$.
	
	It remains to consider $j=R-1$. This index belongs to
	$B_i^{(\ell+1)}$ only when the endpoint for cell $i+1$ is retained,
	in which case
	\[
	|D_R^{(\ell)}|
	\le
	|D_{R-1}^{(\ell)}|.
	\]
	The induction hypothesis controls the block beginning at $R-1$, and
	hence
	\[
	|D_{R-1}^{(\ell)}|
	\le C_{\ell,\Lambda}|a_i|.
	\]
	In the expansion of $D_R^{(\ell)}$, the only new jump is
	$a_{R+\ell-1}$. The lower bound for the last endpoint coefficient
	therefore gives
	\[
	c_{\ell,\Lambda}|a_{R+\ell-1}|
	\le
	|D_R^{(\ell)}|
	+C_{\ell,\Lambda}|a_i|
	\le
	C_{\ell,\Lambda}'|a_i|.
	\]
	This closes the induction.
\end{proof}

Combining the preceding results yields a two-sided estimate for the
dissipation in terms of local jump amplitudes and normalized divided
differences, eliminating the adaptive stencil selection.

\begin{corollary}[Quasi-uniform dissipation equivalence]
	\label{cor:nonuniform-source-equivalence}
	For every fixed $k\ge2$, on a $\Lambda$-quasi-uniform mesh,
	\begin{equation}\label{eq:nonuniform-source-equivalence}
		Q_{k,h}(\bar u)
		\asymp_{k,\Lambda}
		\sum_i|a_i|
		\sum_{j\in B_i^{(k)}}\widehat D_j^{(k)}
		\asymp_{k,\Lambda}
		\sum_jA_j^{(k)}\widehat D_j^{(k)}.
	\end{equation}
\end{corollary}

\begin{proof}
	The first comparison is
	Corollary~\ref{cor:nonuniform-interval-source}. By
	Lemma~\ref{lem:nonuniform-endpoint-partition}, the intervals
	$B_i^{(k)}$ partition $\mathbb Z$. For the unique $i$ such that
	$j\in B_i^{(k)}$, Lemma~\ref{lem:nonuniform-local-amplitude-bound} gives
	\[
	|a_i|\le A_j^{(k)}
	\le C_{k,\Lambda}|a_i|.
	\]
	Multiplying by $\widehat D_j^{(k)}$ and summing over the partition gives
	the second comparison.
\end{proof}

With Corollary~\ref{cor:nonuniform-source-equivalence} establishing equivalence between the ENO dissipation $Q_{k,h}(\bar u)$ and the canonical divided-difference functional $\sum_j A_j^{(k)}\widehat D_j^{(k)}$, the coercivity problem is decoupled from adaptive stencil selection. It remains to establish discrete interpolation inequalities analogous to Theorem~\ref{thm:discrete-GN}. On uniform grids, this discrete-to-continuous transfer was mediated by cardinal B-splines. On nonuniform meshes, translation invariance is lost. In the next subsection, we overcome this difficulty by constructing a smooth, compactly supported lift via local moving polynomial interpolants.

\subsection{A smooth local-polynomial lift on a quasi-uniform mesh}
\label{sec:quasi-uniform-interpolation}

Throughout this subsection, fix
\[
k\ge2,
\qquad
m:=\left\lceil\frac{k}{2}\right\rceil,
\]
and adopt the normalized coordinates $1\le h_i\le\Lambda$ supplied by
Lemma~\ref{lem:nonuniform-affine-scaling}.

For each $j\in\mathbb Z$, let
\[
\mathcal W_j:=(x_{j-1/2},x_{j+k-1/2})
\]
be the open interval spanned by the $k$ consecutive cells
$I_j,\ldots,I_{j+k-1}$, and let $P_j\in\mathbb P_{k-1}$ denote the unique
polynomial satisfying
\[
\frac1{h_\ell}\int_{I_\ell}P_j(x)\,dx=\bar u_\ell,
\qquad
\ell=j,\ldots,j+k-1.
\]
We refer to $P_j$ as the local cell-average interpolant on $\mathcal W_j$.

The function $S_h$ constructed below blends these local
interpolants via a subordinate smooth partition of unity: it coincides with a
single interpolant near each designated interface and transitions from
$P_j$ to $P_{j+1}$ strictly within the intervening transition cell.  Its purpose
is to transfer discrete jump moments and highest divided differences to
derivative norms of a smooth function on $\mathbb R$.  We emphasize that
$S_h$ is neither an ENO reconstruction polynomial nor a global
cell-average interpolant. The designation ``lift'' refers solely to this passage
from the discrete sequence $\bar u$ to an auxiliary smooth function.

\begin{theorem}[Smooth local-polynomial lift on a quasi-uniform mesh]
	\label{thm:quasi-local-patching}
	Assume that $1\le h_i\le\Lambda$ for all $i\in\mathbb Z$.
	For every compactly supported sequence of cell averages
	$\bar u=(\bar u_i)_{i\in\mathbb Z}$, there exists
	$S_h\in C_c^\infty(\mathbb R)$ such that
	\begin{equation}\label{eq:patching-linfty}
		\|S_h\|_{L^\infty(\mathbb R)}
		\le C_{k,\Lambda}\|\bar u\|_{\ell^\infty},
	\end{equation}
	\begin{equation}\label{eq:patching-first-derivative}
		\sum_i|a_i|^{k+1}
		\le C_{k,\Lambda}
		\|S_h'\|_{L^{k+1}(\mathbb R)}^{k+1},
	\end{equation}
	and
	\begin{equation}\label{eq:patching-high-derivative}
		\int_{\mathbb R}|S_h'(x)|\,|S_h^{(k)}(x)|\,dx
		\le C_{k,\Lambda}
		\sum_jA_j^{(k)}\widehat D_j^{(k)}.
	\end{equation}
\end{theorem}

The construction of $S_h$ and the proof of these estimates are deferred to
Appendix~\ref{app:quasi-patching}.

With this lift in place, we establish the nonuniform analogues of the
interpolation inequalities from Section~\ref{sec:odd-coercivity}.

\begin{corollary}[Second-order interpolation on a quasi-uniform mesh]
	\label{cor:quasi-second-order-interpolation}
	Suppose that $k=2$ and $1\le h_i\le\Lambda$ for all $i\in\mathbb Z$.  Then every compactly
	supported sequence of cell averages satisfies
	\[
	\sum_i|a_i|^3
	\le
	C_\Lambda U
	\sum_jA_j^{(2)}\widehat D_j^{(2)}.
	\]
\end{corollary}

\begin{proof}
	For any $f\in C_c^\infty(\mathbb R)$, decompose $f'=g_+ - g_-$ with
	\[
	g_+:=\max\{f',0\},
	\qquad
	g_-:=\max\{-f',0\}.
	\]
	Integrating by parts separately on the regions where $f'$ has constant sign yields
	\[
	\|f'\|_{L^3}^3
	\le
	2\|f\|_{L^\infty}
	\int_{\mathbb R}|f'(x)|\,|f''(x)|\,dx.
	\]
	Applying this bound to the lift $S_h$ supplied by
	Theorem~\ref{thm:quasi-local-patching} for $k=2$, and invoking
	\eqref{eq:patching-first-derivative} followed by
	\eqref{eq:patching-linfty} and
	\eqref{eq:patching-high-derivative}, we obtain
	\[
	\sum_i|a_i|^3
	\le
	C_\Lambda U
	\sum_jA_j^{(2)}\widehat D_j^{(2)},
	\]
	recalling $U=\|\bar u\|_{\ell^\infty}$.
\end{proof}

\begin{corollary}[Odd-order interpolation on a quasi-uniform mesh]
	\label{cor:quasi-odd-interpolation}
	Let $k=2m-1\ge3$ be odd, and suppose that
	$1\le h_i\le\Lambda$ for all $i\in\mathbb Z$.  Every
	compactly supported sequence of cell averages satisfies
	\begin{equation}\label{eq:quasi-odd-interpolation}
		\sum_i|a_i|^{2m}
		\le C_{k,\Lambda}
		\|\bar u\|_{\ell^\infty}^{2m-2}
		\sum_j A_j^{(k)}\widehat D_j^{(k)}.
	\end{equation}
\end{corollary}

\begin{proof}
	Applying the one-dimensional Gagliardo--Nirenberg inequality
	\eqref{eq:continuous-GN} to the lift $S_h$, we have
	\begin{equation}\label{eq:quasi-continuous-GN}
		\|S_h'\|_{L^{2m}}^{2m}
		\le C_m\|S_h\|_{L^\infty}^{2m-2}
		\|S_h^{(m)}\|_{L^2}^{2}.
	\end{equation}
	Since $S_h\in C_c^\infty(\mathbb R)$, integrating by parts $m-1$ times
	yields
	\begin{equation}\label{eq:quasi-integration-by-parts}
		\|S_h^{(m)}\|_{L^2}^{2}
		=
		(-1)^{m-1}
		\int_{\mathbb R}S_h'(x)S_h^{(2m-1)}(x)\,dx.
	\end{equation}
	Taking absolute values in \eqref{eq:quasi-integration-by-parts}, and using $2m-1=k$, gives
	\[
	\|S_h^{(m)}\|_{L^2}^{2}
	\le
	\int_{\mathbb R}|S_h'(x)|\,|S_h^{(k)}(x)|\,dx.
	\]
	Combining this estimate with
	\eqref{eq:patching-linfty},
	\eqref{eq:patching-first-derivative},
	\eqref{eq:patching-high-derivative}, and
	\eqref{eq:quasi-continuous-GN} completes the proof of
	\eqref{eq:quasi-odd-interpolation}.
\end{proof}

\subsection{Persistence of the uniform-grid classification}
\label{sec:quasi-uniform-coercivity}

\begin{proof}[Proof of Theorem~\ref{thm:mesh-main}\textup{(i)}]
	By the affine scaling invariance established in
	Lemma~\ref{lem:nonuniform-affine-scaling}, it suffices to work under the
	normalized coordinate condition $1\le h_i\le\Lambda$.
	
	If $k=2$, combining Corollary~\ref{cor:quasi-second-order-interpolation}
	with the dissipation equivalence \eqref{eq:nonuniform-source-equivalence}
	yields
	\[
	\sum_i|a_i|^3
	\le
	C_\Lambda U
	\sum_jA_j^{(2)}\widehat D_j^{(2)}
	\le
	C_\Lambda UQ_{2,h}(\bar u).
	\]
	
	Now suppose that $k=2m-1\ge3$ is odd.  Combining
	Corollary~\ref{cor:quasi-odd-interpolation} with the dissipation equivalence
	\eqref{eq:nonuniform-source-equivalence} gives
	\begin{align*}
		\sum_i|a_i|^{k+1}
		&\le
		C_{k,\Lambda}U^{k-1}
		\sum_jA_j^{(k)}\widehat D_j^{(k)}\\
		&\le
		C_{k,\Lambda}U^{k-1}Q_{k,h}(\bar u).
	\end{align*}
	
	In both cases, all discrete norms and functionals appearing in these estimates
	are invariant under the affine rescaling in
	Lemma~\ref{lem:nonuniform-affine-scaling}.  The coercivity bound therefore
	holds on every mesh satisfying
	\[
	0<h_*\le h_i\le\Lambda h_*,
	\]
	with the same constant $C_{k,\Lambda}$.
	
	Finally, if $k\ge4$ is even, the uniform unit grid belongs to every
	$\Lambda$-quasi-uniform mesh class (with $\Lambda\ge1$).  The unit-grid failure
	established in Theorem~\ref{thm:A} therefore precludes the existence of any
	uniform constant $C_{k,\Lambda}$.
\end{proof}

Consequently, the ENO--TV parity classification of Theorem~\ref{thm:A} remains
stable under bounded global mesh distortion.

\subsection{Slowly graded meshes and failure of global coercivity}
\label{sec:arbitrary-mesh-counterexample}

The positive coercivity result on quasi-uniform meshes prompts the question of
whether global comparability of cell sizes can be relaxed to local mesh-ratio
bounds imposed on sliding index windows of fixed radius.  We now show that the
answer is negative: coercivity fails completely even when consecutive cell
sizes vary by a factor arbitrarily close to unity.

The counterexample is based on a piecewise affine tent function on a slowly
graded mesh.  By polynomial exactness, the reconstructed jumps vanish identically
throughout the affine flanks of the tent.  Meanwhile, the geometric grading places
microscopic cells in narrow buffer zones around the three corners, while retaining
macroscopic cells in the affine interiors.  Consequently, for every $2\le p<\infty$,
the discrete jump moments $\sum_i|a_i|^p$ remain bounded away from zero, whereas
the total ENO dissipation vanishes in the limit.

By Corollary~\ref{cor:signed-equals-absolute-source},
\[
Q_{k,h}(\bar u)=Q_{k,h}^{\mathrm{abs}}(\bar u),
\]
ensuring that this decay does not rely on cancellation between interface
contributions.  We first construct a graded transition mesh connecting
the microscopic and macroscopic cell scales.

\begin{lemma}[A graded partition with prescribed length]
	\label{lem:graded-mesh-bridge}
	Let $r>1$, $L>0$, and let $M$ be a positive integer.  There exist
	$\varepsilon_0=\varepsilon_0(r,L,M)>0$ and $c_r>0$ such that,
	whenever $0<\varepsilon<\varepsilon_0$, the interval $[0,L]$
	admits a finite partition whose first $M$ and last $M$ cells have length
	$\varepsilon$, whose consecutive cell lengths $\ell_j$ satisfy
	\begin{equation}\label{eq:graded-bridge-ratio}
		r^{-1}
		\le
		\frac{\ell_{j+1}}{\ell_j}
		\le r,
	\end{equation}
	and which contains an interior cell of length at least $c_rL$.
\end{lemma}

\begin{proof}
	For $N\ge1$ and $1\le\rho\le r$, consider the symmetric list of cell lengths
	\[
	\underbrace{\varepsilon,\ldots,\varepsilon}_{M\ \mathrm{times}},
	\varepsilon\rho,\ldots,\varepsilon\rho^N,
	\varepsilon\rho^{N+1},
	\varepsilon\rho^N,\ldots,\varepsilon\rho,
	\underbrace{\varepsilon,\ldots,\varepsilon}_{M\ \mathrm{times}}.
	\]
	Its sum is
	\begin{equation}\label{eq:graded-bridge-total-length}
		F_N(\rho)
		:=
		2M\varepsilon
		+2\varepsilon\sum_{j=1}^{N}\rho^j
		+\varepsilon\rho^{N+1}.
	\end{equation}
	Choose the minimal integer $N$ such that $F_N(r)\ge L$.  For each fixed
	$N$, one has $F_N(r)\to0$ as $\varepsilon\downarrow0$. Hence this
	minimal integer tends to infinity as $\varepsilon\downarrow0$.  After
	decreasing $\varepsilon_0$ if necessary, we may therefore assume that $N\ge2$.
	To obtain an upper bound, define
	\[
	\widetilde N_\varepsilon
	:=
	\max\left\{
	1,
	\left\lceil
	\log_r\frac{L}{\varepsilon}
	\right\rceil
	\right\}.
	\]
	Then
	\[
	\varepsilon r^{\widetilde N_\varepsilon+1}\ge L,
	\]
	and hence $F_{\widetilde N_\varepsilon}(r)\ge L$.  Therefore,
	\[
	N\le\widetilde N_\varepsilon
	=
	O_{r,L}\bigl(|\log\varepsilon|\bigr).
	\]
	It follows that
	\[
	F_N(1)
	=
	(2M+2N+1)\varepsilon
	<L
	\]
	for every sufficiently small $\varepsilon$.  By continuity and
	monotonicity of $F_N$, there exists $\rho_\varepsilon\in[1,r]$ such that
	\[
	F_N(\rho_\varepsilon)=L.
	\]
	The displayed list with $\rho=\rho_\varepsilon$ therefore defines a
	partition of $[0,L]$ satisfying \eqref{eq:graded-bridge-ratio}.
	
	We claim that
	\begin{equation}\label{eq:graded-rho-lower-bound}
		\rho_\varepsilon\ge\sqrt r
	\end{equation}
	for all sufficiently small $\varepsilon$.  By the minimality of $N$,
	we have $F_{N-1}(r)<L$, and hence
	\[
	\varepsilon r^N<L.
	\]
	If $\rho_\varepsilon\le\sqrt r$, then
	\[
	\begin{aligned}
		F_N(\rho_\varepsilon)
		&\le
		2M\varepsilon
		+(2N+1)\varepsilon r^{(N+1)/2}\\
		&\le
		2M\varepsilon
		+C_r(N+1)\sqrt{\varepsilon L},
	\end{aligned}
	\]
	which tends to zero because $N=O_{r,L}(|\log\varepsilon|)$.  This
	contradicts $F_N(\rho_\varepsilon)=L$ and proves
	\eqref{eq:graded-rho-lower-bound}.
	
	Let
	\[
	\ell_*:=\varepsilon\rho_\varepsilon^{N+1}
	\]
	denote the length of the central cell.  Using
	\eqref{eq:graded-rho-lower-bound}, we obtain
	\[
	\begin{aligned}
		L-2M\varepsilon
		\le
		\ell_*
		\left(
		1+2\sum_{q=1}^{\infty}\rho_\varepsilon^{-q}
		\right)
		\le
		\ell_*
		\left(1+\frac{2}{\sqrt r-1}\right).
	\end{aligned}
	\]
	After decreasing $\varepsilon_0$, if necessary, so that
	$2M\varepsilon\le L/2$, this gives
	\[
	\ell_*
	\ge
	\frac{L}{2}
	\left(1+\frac{2}{\sqrt r-1}\right)^{-1}
	=:c_rL.
	\]
	Because $N\ge2$, this central cell is strictly in the interior of the
	partition, which completes the proof.
\end{proof}

\begin{theorem}[Fixed-radius mesh-ratio bounds do not imply coercivity]
	\label{thm:finite-window-mesh-failure}
	Fix $k\ge2$, an integer $R\ge1$, and $\Lambda>1$.
	There exist a fixed mesh satisfying the general conditions of
	Subsection~\ref{sec:nonuniform-reconstruction} and compactly supported
	cell-average sequences $\bar u^{(n)}$ on that mesh such that
	\begin{equation}\label{eq:sliding-window-mesh-ratio}
		\sup_{i\in\mathbb Z}
		\frac{\max_{|q-i|\le R}h_q}
		{\min_{|q-i|\le R}h_q}
		\le\Lambda.
	\end{equation}
	More strongly, the mesh may be chosen so that
	\begin{equation}\label{eq:adjacent-mesh-ratio}
		\Lambda^{-1/(2R)}
		\le
		\frac{h_{i+1}}{h_i}
		\le
		\Lambda^{1/(2R)}
		\qquad(i\in\mathbb Z).
	\end{equation}
	The data satisfy
	\[
	\sup_n\|\bar u^{(n)}\|_{\ell^\infty}\le\frac12,
	\qquad
	Q_{k,h}^{\mathrm{abs}}(\bar u^{(n)})\longrightarrow0,
	\]
	whereas, for every $2\le p<\infty$, there is a constant
	$c_{p,R,\Lambda}>0$ such that
	\begin{equation}\label{eq:graded-mesh-moment}
		\inf_n\sum_i|a_i^{(n)}|^p
		\ge c_{p,R,\Lambda}>0.
	\end{equation}
	Consequently, no constant depending only on $k,p,R,\Lambda$ can make
	\[
	\sum_i|a_i|^p
	\le
	C_{k,p,R,\Lambda}
	\|\bar u\|_{\ell^\infty}^{p-2}
	Q_{k,h}(\bar u)
	\]
	valid on the mesh class \eqref{eq:sliding-window-mesh-ratio}.  The same
	conclusion holds with $Q_{k,h}$ replaced by
	$Q_{k,h}^{\mathrm{abs}}$.  In fact, for each fixed finite $p\ge2$,
	no finite coercivity constant exists even on this single mesh.
\end{theorem}

The proof is given in
Appendix~\ref{app:graded-mesh-construction}.
Figure~\ref{fig:nonuniform-tent-mesh} illustrates the local
mesh pattern for the tent function $\phi$ defined
in~\eqref{eq:graded-mesh-tent}.
Here $r=\Lambda^{1/(2R)}$, and $\varepsilon_n$ denotes
the cell width near each of the three corners.
The marked interfaces have cell-average jumps bounded
away from zero but no reconstructed jump, whereas
dissipation can occur only near the corners.

\begin{figure}[!htb]
	\centering
	\includegraphics[width=0.99\linewidth]{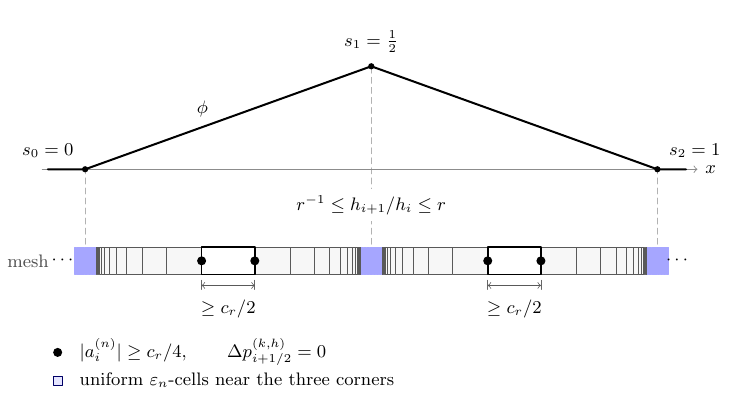}
\caption{Schematic local pattern in the fixed-mesh counterexample. Each corner is flanked on either side by at least $10k$ cells of uniform width $\varepsilon_n$. Between these uniform buffers, cell widths increase and decrease gradually under the adjacent-ratio bound~\eqref{eq:adjacent-mesh-ratio}. At the marked interfaces, all admissible stencils lie within a single affine piece, so the reconstructed jumps vanish by polynomial exactness. Only $O(k)$ interfaces near the corners can contribute to entropy dissipation.}
\label{fig:nonuniform-tent-mesh}
\end{figure}

The fixed mesh in
Theorem~\ref{thm:finite-window-mesh-failure} is not globally
quasi-uniform.  Indeed, it contains cells of length
$\varepsilon_n\to0$, while each embedded tent pattern contains a cell
of length at least $c_r/2$.  Hence there is no contradiction with the
positive quasi-uniform result.

This completes the ENO--TV coercivity analysis for ENO reconstruction, whereas the resolution of the conjecture for ENO interpolation is given in the next section.

\section{Resolution of the conjecture for ENO interpolation}
\label{sec:point-value-enotv}

This section extends the resolution of ENO--TV conjecture  to 
 ENO interpolation
by comparing its dissipation with the cell-average functional already
analyzed.  Let $v=(v_i)_{i\in\mathbb Z}$ be finitely supported real nodal data on
$x_i=i$, and set $a_i:=v_{i+1}-v_i$.  We use the standard order-$k$ ENO
interpolation procedure \cite[Section~1.2]{FMT}: starting from node $i$,
select $k$ consecutive nodes by the recursion
\eqref{eq:eno-endpoint-recursion}, using ordinary nodal divided differences
and resolving ties to the right.  Let $\pi_i^{(k)}\in\mathbb P_{k-1}$
interpolate $v$ on the selected stencil, and define
\begin{equation}\label{eq:point-value-dissipation}
 \begin{gathered}
 \Delta\pi_{i+1/2}^{(k)}
 :=\pi_{i+1}^{(k)}(i+\tfrac12)-\pi_i^{(k)}(i+\tfrac12),\\
 Q_k^{\mathrm{pt}}(v):=\sum_i a_i\Delta\pi_{i+1/2}^{(k)}.
 \end{gathered}
\end{equation}
Here $Q_k(v)$ denotes the cell-average reconstruction functional
\eqref{eq:eno-source-section2} applied to the \emph{same numerical sequence}
$v$. No averaging or change of data is performed.

\begin{proposition}[Comparison of the two ENO dissipations]
\label{prop:point-value-comparison}
 For every $k\ge2$, the ENO reconstruction and interpolation applied to
 the same numerical sequence select identical stencils.  Moreover,
 $a_i\Delta\pi_{i+1/2}^{(k)}\ge0$ at every interface, and
 \begin{equation}\label{eq:point-value-comparison}
 c_k^{\mathrm{pt}}Q_k(v)
 \le Q_k^{\mathrm{pt}}(v)
 \le C_k^{\mathrm{pt}}Q_k(v),
 \end{equation}
 where
 \begin{equation}\label{eq:point-value-comparison-constants}
 \begin{gathered}
 \rho_{k,d}:=\frac{k}{4^{k-1}}
 \binom{2d}{d}\binom{2(k-1-d)}{k-1-d},\qquad d=0,\ldots,k-1,\\
 c_k^{\mathrm{pt}}:=\min_{0\le d<k}\rho_{k,d}>0,
 \qquad C_k^{\mathrm{pt}}:=\max_{0\le d<k}\rho_{k,d}.
 \end{gathered}
 \end{equation}
 In particular, $Q_k^{\mathrm{pt}}(v)\asymp_k\mathcal E_k(\D v)$.
\end{proposition}

\begin{proof}
 On the unit grid, the ordinary nodal divided difference is
 \[
 v[x_j,\ldots,x_{j+\ell}]
 =\frac{\D^\ell v_j}{\ell!},
 \]
 whereas the cell-average divided difference in
 \eqref{eq:divided-difference-normalization} is
 $\D^\ell v_j/(\ell+1)!$.  Thus the two candidates at each level
 differ between the algorithms by the same positive factor $\ell+1$.
 Induction gives identical endpoints $r_i^{(\ell)}$ and hence identical
 intervals $B_i^{(\ell)}$ from \eqref{eq:selected-index-interval}.

 Let $L_j\in\mathbb P_{k-1}$ interpolate $v$ at $j,\ldots,j+k-1$.
 The adjacent-stencil Newton identity gives
 \[
 L_{j+1}(x)-L_j(x)
 =\frac{\D^kv_j}{(k-1)!}
 \prod_{q=1}^{k-1}(x-j-q).
 \]
 Telescoping over $j\in B_i^{(k)}$ and evaluating at $i+1/2$ yield
 \begin{equation}\label{eq:point-value-localized-jump}
 \begin{gathered}
 \Delta\pi_{i+1/2}^{(k)}
 =\sum_{j\in B_i^{(k)}}\beta_{k,i-j}\,\D^{k-1}a_j,\\
 \beta_{k,d}:=\frac1{(k-1)!}
 \prod_{q=1}^{k-1}(d+\tfrac12-q).
 \end{gathered}
 \end{equation}
 Since $0\le i-j<k$, comparison with
 \eqref{eq:localized-fmt-coefficient} gives
 \[
 \beta_{k,d}=\rho_{k,d}\gamma_{k,d},\qquad \rho_{k,d}>0.
 \]
 Each term $a_i\gamma_{k,i-j}\D^{k-1}a_j$ is nonnegative by
 \eqref{eq:fmt-termwise-sign}.  Multiplying these terms by the positive
 factors $\rho_{k,i-j}$ proves the point-value sign property and, after
 summation, \eqref{eq:point-value-comparison}.  The final equivalence
 follows from Theorem~\ref{cor:source-equivalence}.
\end{proof}

\begin{corollary}[Point-value ENO--TV classification]
\label{cor:point-value-parity}
 For each integer $k\ge2$, there exists a constant $K_k>0$ such that
 \begin{equation}\label{eq:point-value-enotv}
 \sum_i|\D v_i|^{k+1}
 \le K_k\|v\|_{\ell^\infty}^{k-1}Q_k^{\mathrm{pt}}(v)
 \end{equation}
 for every finitely supported nodal sequence $v$ if and only if $k=2$ or
 $k$ is odd.  For every even $k\ge4$, the same sequences constructed in
 Section~\ref{sec:even-obstruction}, interpreted as nodal data, have
 divergent point-value coercivity quotients.
\end{corollary}

\begin{proof}
 For $k=2$ and odd $k$, combine Theorem~\ref{thm:A} with
 $Q_k(v)\le(c_k^{\mathrm{pt}})^{-1}Q_k^{\mathrm{pt}}(v)$.
 For even $k\ge4$, take $v^{(L)}:=\bar u^{(L)}$ from
 Theorem~\ref{thm:A}.  Both dissipations are positive by
 Lemma~\ref{lem:source-positivity} and
 Proposition~\ref{prop:point-value-comparison}, and
 \[
 \frac{\sum_i|\D v_i^{(L)}|^{k+1}}
 {\|v^{(L)}\|_{\ell^\infty}^{k-1}Q_k^{\mathrm{pt}}(v^{(L)})}
 \ge\frac1{C_k^{\mathrm{pt}}}
 \frac{\sum_i|\D v_i^{(L)}|^{k+1}}
 {\|v^{(L)}\|_{\ell^\infty}^{k-1}Q_k(v^{(L)})}
 \longrightarrow\infty.
 \]
\end{proof}

An affine rescaling gives the same statements on every uniform grid
$x_i=x_0+ih$, with comparison constants independent of $h>0$.
The stencil identification above relies on uniform spacing.

For the local entropy analysis below, we now define ENO null modes for
nodal interpolation. Fix $k\ge2$ and a uniform grid $x_j=x_0+jh$, with
$h>0$. In the following definition and proposition, the data need not be
compactly supported or bounded. We retain the unscaled forward difference
$\D v_j=v_{j+1}-v_j$.

\begin{definition}[ENO null modes and blocks for interpolation]
\label{def:point-value-null-mode}
A bi-infinite sequence of nodal values $v=(v_j)_{j\in\mathbb Z}$ is an
\emph{order-$k$ ENO null mode for interpolation} if
\[
\D^k v_j=0\qquad\text{for every }j\in\mathbb Z.
\]
Equivalently, its jump sequence $a_j=v_{j+1}-v_j$ satisfies
$\D^{k-1}a=0$ on $\mathbb Z$.

Let $I=\{r,\ldots,s\}$ contain at least $k+1$ indices. We say that $v$
contains an \emph{order-$k$ ENO null-mode block for interpolation on $I$}
if there exists a bi-infinite order-$k$ ENO null mode for interpolation
$v^{\mathrm{null}}$ such that $v_j=v_j^{\mathrm{null}}$ for every $j\in I$.
The restriction $(v_j)_{j\in I}$ is called the corresponding block.
\end{definition}

\begin{proposition}[Polynomial characterization and vanishing interpolation dissipation]
\label{prop:point-value-null-mode}
A bi-infinite sequence $v$ is an order-$k$ ENO null mode for interpolation
if and only if there is a unique $\pi\in\mathbb P_{k-1}$ such that
$v_j=\pi(x_j)$ for every $j\in\mathbb Z$.
For a finite interval $I=\{r,\ldots,s\}$ with $|I|\ge k+1$, the following
conditions are equivalent:
\begin{enumerate}[label=\textup{(\roman*)}]
\item $(v_j)_{j\in I}$ is an order-$k$ ENO null-mode block for interpolation.
\item $\D^k v_j=0$ whenever $\{j,\ldots,j+k\}\subset I$.
\item There is a unique $q\in\mathbb P_{k-1}$ such that $v_j=q(j)$ for
every $j\in I$.
\end{enumerate}
Equivalently, the data on $I$ are the nodal values of the unique polynomial
\[
\pi(x)=q\!\left(\frac{x-x_0}{h}\right)\in\mathbb P_{k-1}.
\]
An admissible order-$k$ nodal stencil for node $i$ is any set of $k$
consecutive nodes containing $i$. If
$\mathcal H_{i,k}=\{i-k+1,\ldots,i+k\}\subset I$, every such stencil for
nodes $i$ and $i+1$ interpolates the same polynomial $\pi$. Consequently,
at $x_{i+1/2}=x_0+(i+\tfrac12)h$,
\begin{equation}\label{eq:point-value-null-mode-vanishing}
\begin{gathered}
\pi_i^{(k)}=\pi_{i+1}^{(k)}=\pi,\\
\Delta\pi_{i+1/2}^{(k)}=0,
\qquad a_i\Delta\pi_{i+1/2}^{(k)}=0,
\end{gathered}
\end{equation}
independently of stencil comparisons and of the tie convention.
For a bi-infinite null mode, these conclusions hold at every interface.
For a block, they require no condition on the data outside $I$.
\end{proposition}

\begin{proof}
For a finite interval, Lemma~\ref{lem:discrete-polynomial-kernel} with
parameter $k-1$ gives the equivalence of \textup{(ii)} and~\textup{(iii)}.
Condition~\textup{(i)} implies~\textup{(ii)} by restriction.
Conversely, extending $q(j)$ from~\textup{(iii)} to all $j\in\mathbb Z$
gives a bi-infinite sequence annihilated by $\D^k$, proving~\textup{(i)}.
The affine change of variable $x=x_0+jh$ gives the stated physical-coordinate
representation.

For a bi-infinite sequence satisfying $\D^kv=0$, choose the polynomial
$q\in\mathbb P_{k-1}$ matching any $k$ consecutive values. Both $v_j$ and
$q(j)$ satisfy the same order-$k$ difference recurrence. Since its first
and last coefficients are nonzero, this agreement propagates in both
index directions. Thus $v_j=q(j)$ on all of $\mathbb Z$, and the converse
follows because $\D^k$ annihilates every polynomial of degree at most
$k-1$. Uniqueness in both the finite and bi-infinite cases follows from
uniqueness of polynomial interpolation at $k$ distinct nodes.

Finally, every admissible stencil under consideration lies in
$\mathcal H_{i,k}\subset I$ and consists of $k$ distinct nodes. The
polynomial $\pi$ matches all their nodal values, so uniqueness of
interpolation makes every stencil return $\pi$. Its two interface traces
coincide, proving~\eqref{eq:point-value-null-mode-vanishing}.
\end{proof}

Definitions~\ref{def:polynomial-null-mode} and~\ref{def:point-value-null-mode}
therefore specify the same numerical sequences and blocks on a uniform
grid through the condition $\D^k v=0$. Their polynomial interpretations
depend on the data: the reconstruction polynomial matches cell averages,
whereas the interpolation polynomial matches nodal values. In the
following entropy analysis, ENO null modes always refer to the nodal
modes of Definition~\ref{def:point-value-null-mode}.

\section{Local entropy balances on ENO null modes}
\label{sec:shift-cohomology}

The ENO null modes developed in the coercivity analysis also provide
a necessary test for entropy-pair transfer.  On these modes the ENO
numerical diffusion vanishes identically, leaving the local residual for any
auxiliary entropy pair to be governed entirely by the underlying
entropy-conservative core.  Any discrete entropy inequality valid for general
solutions must therefore already hold on these null modes.  We are thus led
to ask which residual terms can be eliminated by a consistent numerical
entropy flux on a fixed finite stencil, while keeping the conservation-law
flux unchanged.

Throughout Sections~\ref{sec:shift-cohomology} and
\ref{sec:entropy-transfer}, we use the point-value finite-difference formulation
on $x_i=ih$, $h>0$, with $u_i(t)\approx u(x_i,t)$,
$a_i:=u_{i+1}-u_i$, and $x_{i+1/2}=(i+1/2)h$.
The physical flux $f$ and entropy functions are smooth on an open state
interval $J$, and all expressions are evaluated only where their finitely
many input arguments remain in $J$.  
We use the nodal ENO null modes and blocks of
Definition~\ref{def:point-value-null-mode}, with the polynomial
characterization and vanishing interpolation dissipation established in
Proposition~\ref{prop:point-value-null-mode}.

\subsection{The fixed flux and its null-mode residual}
\label{sec:tecno-architecture}
\label{sec:null-mode-restriction}

For the prescribed quadratic design entropy $\eta_2(u)=u^2/2$, the canonical
two-point entropy-conservative flux is
\begin{equation}\label{eq:quadratic-ec-core}
	F^*(A,B):=\int_0^1 f((1-\theta)A+\theta B)\,d\theta.
\end{equation}
Indeed, with $q_2'=uf'$ and entropy potential $\psi_2=uf-q_2$, the relation
$\psi_2'=f$ gives Tadmor's condition $(B-A)F^*(A,B)=\psi_2(B)-\psi_2(A)$.
For an ENO interpolation order $k\ge2$, set
\begin{equation}\label{eq:FMT-coefficients}
	\mu_k:=\left\lceil\frac{k}{2}\right\rceil,
	\qquad\alpha_\rho^{[\mu]}:=(-1)^{\rho+1}
	\frac{2(\mu!)^2}{\rho(\mu-\rho)!(\mu+\rho)!}.
\end{equation}
The high-order entropy-conservative FMT core \cite{FMTtecno} is
\begin{equation}\label{eq:FMT-core}
	\widetilde F^k_{i+1/2}
	:=\sum_{\rho=1}^{\mu_k}\alpha_\rho^{[\mu_k]}
	\sum_{\ell=0}^{\rho-1}
	F^*(u_{i-\ell},u_{i+\rho-\ell}),
\end{equation}
whose coefficients satisfy the discrete moment conditions
\begin{equation}\label{eq:FMT-moments}
	\sum_{\rho=1}^{\mu}\alpha_\rho^{[\mu]}\rho^{2j-1}
	=\begin{cases}1,&j=1,\\0,&2\le j\le\mu.\end{cases}
\end{equation}
For smooth nodal data $u_i=u(x_i)$, these identities give
$h^{-1}(\widetilde F^k_{i+1/2}-\widetilde F^k_{i-1/2})
=\partial_x f(u)(x_i)+O(h^{2\mu_k})$
\cite[Theorem~2.2]{FMTtecno}.  Furthermore, the FMT construction supplies
a consistent, translation-invariant finite-stencil numerical entropy flux
$\widetilde{\mathcal Q}^{k}_2$ such that the semi-discrete scheme with flux
$\widetilde F^k$ satisfies the exact local balance
\begin{equation}\label{eq:FMT-quadratic-balance}
	\frac{d}{dt}\frac{u_i^2}{2}
	+\frac{\widetilde{\mathcal Q}^{k}_{2,i+1/2}
		-\widetilde{\mathcal Q}^{k}_{2,i-1/2}}{h}=0.
\end{equation}
This exact conservation identity serves as the base point for the
entropy-transfer perturbation in Theorem~\ref{thm:no-two-sided-transfer}.

For dissipation, use the ENO interpolation jump $\Delta\pi_{i+1/2}^{(k)}$
from \eqref{eq:point-value-dissipation}, rescaled to spacing $h$. 
The full scalar TeCNO flux associated with the design entropy $\eta_2$ is
defined by
\begin{equation}\label{eq:full-scalar-tecno-flux}
	\widehat F^k_{i+1/2}
	:=\widetilde F^k_{i+1/2}
	-\frac12\lambda_{i+1/2}\Delta\pi_{i+1/2}^{(k)},
	\qquad\lambda_{i+1/2}\ge0.
\end{equation}
The quadratic entropy variable is $v_2(u)=u$, and
$a_i\Delta\pi_{i+1/2}^{(k)}\ge0$ by
Proposition~\ref{prop:point-value-comparison}.  When testing an auxiliary entropy, we keep the
interpolation in $u$ unchanged.  

\begin{remark*}[Dissipation weights]
When boundary fluxes cancel, the quadratic entropy dissipation is
$\frac12\sum_i\lambda_{i+1/2}a_i\Delta\pi_{i+1/2}^{(k)}$.
A lower bound $\lambda_{i+1/2}\ge\lambda_0>0$ transfers control to
$Q_k^{\mathrm{pt}}((u_i)_i)$.  
Proposition~\ref{prop:point-value-comparison} then also gives control of
$Q_k((u_i)_i)$.
\end{remark*}

For an arbitrary smooth target entropy pair $(\eta,q)$, we set
\begin{equation}\label{eq:additional-entropy-notation}
	v:=\eta',\qquad q'=\eta'f',\qquad\psi:=vf-q,
\end{equation}
and write $[w]_{i+1/2}:=w(u_{i+1})-w(u_i)$.  The interface entropy
mismatches for the core and full fluxes are defined respectively by
\begin{align}
	g_{\eta,i+1/2}
	&:=[v]_{i+1/2}\widetilde F^k_{i+1/2}-[\psi]_{i+1/2},
	\label{eq:mismatch}\\
	\widehat g_{\eta,i+1/2}
	&:=[v]_{i+1/2}\widehat F^k_{i+1/2}-[\psi]_{i+1/2}
	\label{eq:full-mismatch}\\
	&=g_{\eta,i+1/2}
	-\frac12\lambda_{i+1/2}[v]_{i+1/2}\Delta\pi_{i+1/2}^{(k)}.
	\label{eq:full-mismatch-decomposition}
\end{align}

\begin{proposition}[Scalar diffusion sign]
	\label{prop:universal-scalar-diffusion-sign}
	Let $\eta\in C^2(J)$ be convex.  For $A,B\in J$, define the mean
	second derivative
	\begin{equation}\label{eq:relative-entropy-curvature}
		\Theta_\eta(A,B):=\int_0^1\eta''((1-s)A+sB)\,ds\ge0.
	\end{equation}
	At every interface of the scalar scheme associated with $\eta_2$,
	\begin{equation}\label{eq:target-diffusion-sign}
		[v]_{i+1/2}\Delta\pi_{i+1/2}^{(k)}
		=\Theta_\eta(u_i,u_{i+1})
		a_i\Delta\pi_{i+1/2}^{(k)}\ge0.
	\end{equation}
	Consequently, the numerical diffusion contribution in
	\eqref{eq:full-mismatch-decomposition} is nonpositive.
\end{proposition}
\begin{proof}
	By the fundamental theorem of calculus,
	$[v]_{i+1/2}=\Theta_\eta(u_i,u_{i+1})a_i$, including when
	$a_i=0$.  The claim then follows from Proposition~\ref{prop:point-value-comparison}
	and $\lambda_{i+1/2}\ge0$.
\end{proof}

While the numerical diffusion contribution is always dissipative, the core
mismatch $g_{\eta,i+1/2}$ carries no definite sign.  To quantify the resulting
cell entropy residual, we introduce the centered reference numerical entropy flux
\begin{equation}\label{eq:canonical-full-entropy-flux}
	\begin{split}
		\widehat{\mathcal Q}^{0}_{\eta,i+1/2}
		:={}\frac{v(u_i)+v(u_{i+1})}{2}\widehat F^k_{i+1/2}
		-\frac{\psi(u_i)+\psi(u_{i+1})}{2}.
	\end{split}
\end{equation}
On constant data $u_j=V\in J$, consistency of the core and
polynomial exactness give $\widehat F^k=f(V)$, so that $\widehat{\mathcal Q}^0_\eta$
reduces to $v(V)f(V)-\psi(V)=q(V)$.  Substituting the semi-discrete
conservative update $\dot{u}_i=-(\widehat F^k_{i+1/2}-\widehat F^k_{i-1/2})/h$
yields the exact identity
\begin{equation}\label{eq:full-mismatch-balance}
	\frac{d}{dt}\eta(u_i)
	+\frac{\widehat{\mathcal Q}^{0}_{\eta,i+1/2}
		-\widehat{\mathcal Q}^{0}_{\eta,i-1/2}}{h}
	=\frac{\widehat g_{\eta,i+1/2}+\widehat g_{\eta,i-1/2}}{2h}.
\end{equation}
Even for the design entropy $\eta=\eta_2$, this centered reference flux
generally differs from the exact quadratic entropy flux $\widetilde{\mathcal Q}^k_2$
in \eqref{eq:FMT-quadratic-balance}.  The exact quadratic balance
\eqref{eq:FMT-quadratic-balance} implies that the centered core residual
$(g_{\eta_2,i+1/2}+g_{\eta_2,i-1/2})/(2h)$ can be represented as a discrete
flux divergence. It does not require the individual interface mismatches
$g_{\eta_2,i\pm1/2}$ to vanish.

\begin{corollary}[Null-mode restriction]
	\label{cor:null-mode-restriction}
	Suppose that the nodal data $(u_j)_{j\in I}$ form an order-$k$
	ENO null-mode block for interpolation in the sense of
	Definition~\ref{def:point-value-null-mode}. At every interface with
	$\mathcal H_{i,k}\subset I$, one has
	\begin{equation}\label{eq:full-core-restriction}
		\Delta\pi_{i+1/2}^{(k)}=0,
		\qquad\widehat g_{\eta,i+1/2}=g_{\eta,i+1/2}.
	\end{equation}
	Consequently, whenever
	$\mathcal H_{i-1,k}\cup\mathcal H_{i,k}\subset I$, the full scheme
	satisfies
	\begin{equation}\label{eq:mismatch-balance}
		\frac{d}{dt}\eta(u_i)
		+\frac{\widehat{\mathcal Q}^{0}_{\eta,i+1/2}
			-\widehat{\mathcal Q}^{0}_{\eta,i-1/2}}{h}
		=\frac{g_{\eta,i+1/2}+g_{\eta,i-1/2}}{2h}.
	\end{equation}
	We emphasize that no convexity assumption on $\eta$ is required for these
	algebraic identities.
\end{corollary}
\begin{proof}
	Proposition~\ref{prop:point-value-null-mode} gives
	$\Delta\pi_{i+1/2}^{(k)}=0$ whenever $\mathcal H_{i,k}\subset I$. Substituting this
	into \eqref{eq:full-mismatch-decomposition} reduces $\widehat g_\eta$ to
	$g_\eta$, and applying \eqref{eq:full-mismatch-balance} completes the proof.
\end{proof}

Any local entropy inequality valid for general solutions must therefore hold
for the core balance on these null modes.  Cancellation of this restricted
residual alone, however, gives no guarantee for data outside the space of
ENO null modes.  In the local tests developed below, the ENO null-mode block
is chosen sufficiently wide to encompass all interpolation, core-flux, and
entropy-flux stencils entering the fixed cell balance.

\subsection{ENO null modes and local flux differences}
\label{sec:profiles}
\label{sec:local-coboundaries}

Fix $u_*\in J$ and $m\ge0$.  By Lemma~\ref{lem:discrete-polynomial-kernel},
any jump sequence satisfying $\D^{m+1}a=0$ admits the Newton representation
\begin{equation}\label{eq:newton}
	a_{i+\ell}=\sum_{s=0}^{m}\binom{\ell}{s}c_s,
	\qquad c_s=(\D^sa)_i.
\end{equation}
Writing $u:=u_i-u_*$, the corresponding nodal sequence is given by
\begin{equation}\label{eq:profile-U}
	\begin{gathered}
		U(t)=u+\sum_{s=0}^{m}\binom{t}{s+1}c_s,\\
		u_{i+\ell}=u_*+U(\ell),\qquad \ell\in\mathbb Z.
	\end{gathered}
\end{equation}
For order-$k$ ENO null modes of Definition~\ref{def:point-value-null-mode}, we have $m=k-2$, so that the polynomial $U$ has
degree at most $n:=m+1=k-1$. 
Shifting the base node from $i$ to $i+1$ induces
the unipotent linear transformation
\begin{equation}\label{eq:shift}
	T_m(u,c_0,\ldots,c_m)
	=(u+c_0,c_0+c_1,\ldots,c_{m-1}+c_m,c_m),
\end{equation}
with $T_0(u,c_0)=(u+c_0,c_0)$ when $m=0$.  All Taylor expansions are centered
at the origin $z=(u,c_0,\ldots,c_m)=0$, where each coordinate variable is assigned
degree one.  We emphasize that the homogeneous Taylor degree $D$ is distinct
from both the interpolation order $k$ and the polynomial degree bound $m+1$ of $U$.

Because any local multi-point expression depends on only finitely many nodal
values, choosing a sufficiently small neighborhood of $z=0$ ensures that all
relevant states remain in $J$, so that all local operators and their spatial
shifts are well defined.  The formal bi-infinite polynomial extension therefore
imposes no global boundedness restriction.

Let $\mathscr P_{m,D}$ denote the real vector space of homogeneous polynomials of
degree $D$ in the variables $z$.  The discrete shift difference operator and the
associated first cohomology space are defined by
\begin{equation}\label{eq:coboundary}
	\delta_{m,D}H:=H\circ T_m-H,
	\qquad
	H^1_{m,D}:=\mathscr P_{m,D}/\operatorname{Range}\delta_{m,D}.
\end{equation}
Polynomials in $\operatorname{Range}\delta_{m,D}$ are \emph{coboundaries}, whose
cohomology class in $H^1_{m,D}$ vanishes (for example, $\delta_{m,1}u=c_0$).
The quotient space $H^1_{m,D}$ thus captures precisely the obstruction to
expressing a residual term as a local flux difference.

\begin{definition}[Local primitive]\label{def:local-primitive}
	Let $m,D\ge0$ and $G\in\mathscr P_{m,D}$.  A translation-invariant
	finite-stencil local primitive of $G$ at $u_*$ consists of an integer
	radius $R\ge0$ and a $C^D$ function $\mathcal H$ defined on a neighborhood
	of $(u_*,\ldots,u_*)\in J^{2R+1}$, evaluated at each cell $i$ by
	\[
	\mathcal H_i[(u_j)_j]
	=\mathcal H(u_{i-R},\ldots,u_{i+R}).
	\]
	Upon substituting the polynomial sequence \eqref{eq:profile-U}, the
	homogeneous degree-$D$ Taylor term of $\mathcal H_{i+1}-\mathcal H_i$ at
	$z=0$ must coincide with $G(z)$.  The radius $R$ and the function
	$\mathcal H$ are independent of $i$ and of the length of any ambient
	polynomial block.
\end{definition}

\begin{theorem}[Local-primitive criterion]\label{thm:cohomology-dim}
	Fix $u_*\in J$ and integers $m,D\ge0$.  For
	$G\in\mathscr P_{m,D}$, the following statements are equivalent:
	\begin{enumerate}[label=\textup{(\roman*)},leftmargin=2.2em]
		\item $G$ admits a local primitive in the sense of
		Definition~\ref{def:local-primitive}.
		\item $[G]=0$ in $H^1_{m,D}$, or equivalently
		$G=\delta_{m,D}H$ for some $H\in\mathscr P_{m,D}$.
	\end{enumerate}
	When these equivalent conditions hold, the primitive can be chosen as a
	polynomial in $m+2$ consecutive nodal values, with forward difference
	matching $G$ identically on the polynomial sequences \eqref{eq:profile-U}.
\end{theorem}
\begin{proof}
	Restricting a local primitive $\mathcal H$ to the sequences
	\eqref{eq:profile-U}, we obtain a $C^D$ function $h(z)$ near the origin.
	By translation invariance, its evaluation at cell $i+1$ is given by $h(T_mz)$.
	If $H$ denotes the degree-$D$ Taylor polynomial of $h$, the linearity of
	$T_m$ implies that the corresponding degree-$D$ term of the forward difference
	$\mathcal H_{i+1}-\mathcal H_i$ is $H(T_mz)-H(z)=\delta_{m,D}H(z)$, proving
	necessity.
	
	Conversely, to establish sufficiency, we recover the coordinates $z=(u,c_0,\ldots,c_m)$
	from $m+2$ consecutive nodal values $w_r:=u_{i+r}$ ($0\le r\le m+1$)
	via the explicit formulas
	\[
	\begin{gathered}
		u=w_0-u_*,\\
		c_s=\sum_{r=0}^{s+1}(-1)^{s+1-r}\binom{s+1}{r}w_r,
		\qquad 0\le s\le m.
	\end{gathered}
	\]
	These express the base value and the forward differences of orders
	$1,\ldots,m+1$.  Newton interpolation provides the linear isomorphism
	between these coordinate systems. Hence substitution into any $H\in\mathscr P_{m,D}$
	defines a translation-invariant polynomial function of these $m+2$ nodal values.
	On the sequences \eqref{eq:profile-U}, its values at consecutive indices are
	precisely $H(z)$ and $H(T_mz)$.  Thus any solution $H$ to $G=\delta_{m,D}H$
	yields the required local primitive.  A symmetric stencil as in
	Definition~\ref{def:local-primitive} is obtained by expanding the radius $R$
	and ignoring unused arguments.
\end{proof}

We note that the necessity argument demands $C^D$ regularity solely on the
parameter space of these polynomial sequences. It remains fully valid even if the
candidate primitive $\mathcal H$ fails to be $C^D$ on general non-polynomial data.
We use this restricted regularity for the entropy-flux families in
Theorem~\ref{thm:no-two-sided-transfer}.

The discrete cell balance \eqref{eq:full-mismatch-balance} contains the average
of two interface mismatches, whereas an entropy-flux correction enters through a
difference operator.  The following lemma shows that the exact same cohomological
criterion governs solvability of this cell-balance correction problem.  Let
$\mathcal Q_\eta^{\mathrm{cen}}$ denote the centered core reference numerical entropy
flux obtained from \eqref{eq:canonical-full-entropy-flux} upon replacing the full
flux $\widehat F^k$ with the unperturbed core $\widetilde F^k$.

\begin{lemma}[Cell-balance correction]\label{lem:solvability}
	Fix $u_*\in J$ and $m,D\ge0$, and let $G\in\mathscr P_{m,D}$ be the
	degree-$D$ Taylor polynomial of the core interface mismatch
	\eqref{eq:mismatch} evaluated on the sequences \eqref{eq:profile-U}.
	There exists a translation-invariant finite-stencil interface correction
	$\mathcal K$, whose restriction to these polynomial sequences is $C^D$ in $z$,
	such that the modified entropy flux $\mathcal Q_\eta=\mathcal Q_\eta^{\mathrm{cen}}-\mathcal K$
	cancels the degree-$D$ term of the centered core cell residual for all
	$z$ near zero if and only if $[G]=0$ in $H^1_{m,D}$.
	When such a correction exists, $\mathcal K$ can be chosen as a polynomial
	in the nodal values on the stencil.  (Consistency is not imposed in this lemma.)
	
	When $m=k-2$, the core reference $\mathcal Q_\eta^{\mathrm{cen}}$ coincides
	with $\widehat{\mathcal Q}^{0}_\eta$ on ENO null modes, so this is the
	same degree-$D$ correction problem for the full ENO scheme restricted to
	\eqref{eq:mismatch-balance}.  Cancellation here concerns only the specified
	Taylor degree.
\end{lemma}
\begin{proof}
	Let $\mathsf U:=T_m^*$ denote the pull-back operator $\mathsf UH=H\circ T_m$,
	and let $K\in\mathscr P_{m,D}$ be the degree-$D$ Taylor term of
	$\mathcal K_{i+1/2}$ in coordinates based at cell $i$.  At the left interface
	$i-1/2$, the corresponding terms are $\mathsf U^{-1}K$ and $\mathsf U^{-1}G$.
	Multiplying the corrected cell balance by $h$, cancellation of the degree-$D$
	residual is equivalent to
	\[
	K-\mathsf U^{-1}K=\frac{G+\mathsf U^{-1}G}{2}.
	\]
	Applying $\mathsf U$ to both sides and recalling $\delta_{m,D}=\mathsf U - I$,
	where $I$ is the identity operator, we obtain
	\begin{equation}\label{eq:averaged-coboundary}
		\delta_{m,D}K=\frac{I+\mathsf U}{2}G
		=G+\frac12\delta_{m,D}G.
	\end{equation}
	If an admissible correction exists, identity \eqref{eq:averaged-coboundary}
	directly yields $G=\delta_{m,D}(K-G/2)$, hence $[G]=0$ in $H^1_{m,D}$.
	Conversely, if $[G]=0$ with $G=\delta_{m,D}H$ for some $H\in\mathscr P_{m,D}$,
	then
	\[
	K:=H+\frac12G=\frac{I+\mathsf U}{2}H
	\]
	is an explicit solution to \eqref{eq:averaged-coboundary}.  The polynomial
	interpolation procedure of Theorem~\ref{thm:cohomology-dim} realizes this
	$K$ as a local polynomial correction $\mathcal K$.  The assertion for the
	full ENO scheme follows from Corollary~\ref{cor:null-mode-restriction}.
\end{proof}

\subsection{Consistency and odd-degree residuals}
\label{sec:mismatch}
\label{sec:stability-cone}

The reference numerical entropy flux already satisfies
$\widehat{\mathcal Q}^{0}_\eta(V,\ldots,V)=q(V)$ on constant data.
Thus preserving numerical consistency requires that any local correction
satisfy $\mathcal K(V,\ldots,V)=0$ for every constant state $V$ near $u_*$.
In the coordinates $z=(u,c_0,\ldots,c_m)$, constant data correspond to
$c_0=\cdots=c_m=0$. Hence a homogeneous polynomial correction preserves
consistency if and only if it belongs to the ideal
$\mathfrak I_m:=(c_0,\ldots,c_m)$. Here the ideal is taken in $\mathbb R[u,c_0,\ldots,c_m]$, and
$\mathfrak I_m^2$ consists of polynomials whose monomials each contain at
least two jump-coordinate factors.  The following lemma shows that this
consistency normalization introduces no additional cohomological obstruction
for the FMT mismatch.  We write $G_{\eta,k}^{(D)}$ for the degree-$D$ Taylor
polynomial of $g_\eta$ restricted to order-$k$ ENO null modes.

\begin{lemma}[Consistency normalization]\label{lem:FMT-consistency-normalization}
	Let $m\ge0$, $D\ge1$, and $G\in\operatorname{Range}\delta_{m,D}$.
	There exists $H\in\mathfrak I_m\cap\mathscr P_{m,D}$ satisfying
	$\delta_{m,D}H=G$ if and only if the coefficient of $u^{D-1}c_0$ in $G$
	vanishes:
	\[
	[u^{D-1}c_0]G=0.
	\]
	Moreover, for every smooth flux $f$ and target entropy $\eta$ on $J$,
	every $k\ge2$, and every $D\ge1$, the FMT mismatch satisfies the
	quadratic vanishing property
	\[
	G_{\eta,k}^{(D)}\in\mathfrak I_{k-2}^2\cap\mathscr P_{k-2,D}.
	\]
	Consequently, $[G_{\eta,k}^{(D)}]=0$ in $H^1_{k-2,D}$ is equivalent to
	the cancellation of that degree in the null-mode cell balance by a local
	polynomial entropy-flux correction that preserves consistency on all nearby
	constant states.
\end{lemma}
\begin{proof}
	A homogeneous polynomial vanishes on constant states (where
	$c_0=\cdots=c_m=0$) precisely when it belongs to the ideal $\mathfrak I_m$.
	For any $K\in\mathfrak I_m\cap\mathscr P_{m,D}$, its linear terms in the
	jump coordinates take the form
	\[
	K\equiv u^{D-1}\sum_{s=0}^{m}\beta_sc_s
	\pmod{\mathfrak I_m^2}.
	\]
	The shift transformation $T_m$ preserves both $\mathfrak I_m$ and its
	square $\mathfrak I_m^2$.  Substituting \eqref{eq:shift} and neglecting
	terms in $\mathfrak I_m^2$, we obtain
	\[
	\delta_{m,D}K\equiv
	u^{D-1}\sum_{s=0}^{m-1}\beta_sc_{s+1},
	\qquad
	\delta_{m,D}u^D\equiv Du^{D-1}c_0
	\pmod{\mathfrak I_m^2},
	\]
	where the first summation is omitted when $m=0$.  Every polynomial
	$H\in\mathscr P_{m,D}$ decomposes uniquely as $H=K+\alpha u^D$ with
	$K\in\mathfrak I_m$ and $\alpha\in\mathbb R$.  If $G=\delta_{m,D}H$, the
	displayed congruences yield $[u^{D-1}c_0]G=D\alpha$.  Hence the
	coefficient $[u^{D-1}c_0]G$ vanishes if and only if $\alpha=0$, which is
	equivalent to $H\in\mathfrak I_m$.
	
	To verify that $G_{\eta,k}^{(D)}\in\mathfrak I_{k-2}^2$, fix any constant
	state $V$ near $u_*$.  On constant data $u_j\equiv V$, we have
	$[v]=0$, $\widetilde F^k=f(V)$, and $g_\eta=0$.  Taking the first
	variation along an arbitrary perturbation direction in $z$,
	\[
	\begin{aligned}
		\mathrm d g_\eta
		&=\bigl(\eta''(V)f(V)-\psi'(V)\bigr)
		\mathrm d(u_{i+1}-u_i)
		+[v]\,\mathrm d\widetilde F^k\\
		&=0,
	\end{aligned}
	\]
	where we have used the entropy potential relation $\psi'=\eta''f$ and
	$[v]=0$.  Because this holds for all constant states $V$ in a neighborhood
	of $u_*$, the smooth function $g_\eta(u,c)$ and all its first partial
	derivatives with respect to the jump coordinates $c=(c_0,\ldots,c_{k-2})$
	vanish identically along the subspace $c=0$.  By Taylor's formula with
	integral remainder, we obtain
	$g_\eta(u,c)=\sum_{r,s=0}^{k-2}c_rc_s b_{rs}(u,c)$ for smooth functions
	$b_{rs}$ near the origin.  Every homogeneous Taylor component
	$G_{\eta,k}^{(D)}$ of degree $D\ge1$ therefore belongs to
	$\mathfrak I_{k-2}^2$, while the degree-zero component is identically zero.
	
	Finally, if $[G_{\eta,k}^{(D)}]=0$, the membership
	$G_{\eta,k}^{(D)}\in\mathfrak I_{k-2}^2$ implies in particular that
	$[u^{D-1}c_0]G_{\eta,k}^{(D)}=0$.  By the first part of the lemma, there
	exists $H\in\mathfrak I_{k-2}$ satisfying $\delta_{k-2,D}H=G_{\eta,k}^{(D)}$.
	The corresponding interface correction
	$K=H+G_{\eta,k}^{(D)}/2=(I+T_{k-2}^*)H/2$ also belongs to
	$\mathfrak I_{k-2}$.  In its local polynomial realization, evaluated on
	constant states where $c_s=0$, the correction $\mathcal K$ vanishes
	identically, thereby preserving consistency for all constant states near
	$u_*$.  Necessity was established in Lemma~\ref{lem:solvability}.
\end{proof}

\begin{remark}[Odd-degree and full entropy inequalities]
	\label{rmk:odd-degree-residual}
	\label{rmk:stability-cone-scope}
	Under the invertible change of base cell in \eqref{eq:averaged-coboundary},
	the corrected degree-$D$ cell residual (scaled by $h>0$) is given by
	\begin{equation}\label{eq:corrected-homogeneous-residual}
		R_{G,K}=\frac{I+T_m^*}{2}G-\delta_{m,D}K,
		\qquad [R_{G,K}]=[G]\in H^1_{m,D}.
	\end{equation}
	When $D$ is odd, the polynomial $R_{G,K}$ is odd: $R_{G,K}(-z)=-R_{G,K}(z)$.
	Consequently, the sign condition $R_{G,K}(z)\le0$ throughout a neighborhood
	of the origin forces $R_{G,K}\equiv0$.  A non-vanishing cohomology class at
	any odd degree $D$ therefore presents an insurmountable obstruction to
	achieving a discrete entropy inequality at that homogeneous Taylor degree.
	
	The same conclusion holds for the full $C^D$ cell residual whenever all
	lower-degree Taylor terms have been canceled: a non-zero leading term of odd
	degree $D$ necessarily takes strictly positive values along either the ray
	$tz$ or $-tz$, and the higher-order remainder $o(t^D)$ cannot alter this sign
	for all small $t>0$.  If, however, lower-degree terms remain non-zero, a
	strictly negative lower even-degree term could in principle dominate locally.
	For this reason, the non-vanishing of the degree-five class alone does not
	immediately rule out an entropy inequality for an isolated target entropy.
	The complete two-sided transfer obstruction established in
	Section~\ref{sec:entropy-transfer} will circumvent this limitation by
	differentiating through the exact quadratic balance to obtain an exact
	linearized identity.
\end{remark}

\section{Entropy defects of the FMT flux on ENO null modes}
\label{sec:entropy-transfer}

Section~\ref{sec:shift-cohomology} reduced the solvability of the local
entropy-flux correction problem to the vanishing of the cohomology class
$[G_{\eta,k}^{(D)}]$ in $H^1_{k-2,D}$.  We now compute these obstruction
classes for the fixed quadratic-design FMT flux.  Our analysis relies on
apolar duality, representing linear functionals that annihilate all discrete
flux divergences as explicit polynomial differential operators.  Evaluating
such an apolar test on the entropy mismatch factors the pairing into a scalar
stencil moment and a bilinear pairing between flux and entropy derivatives.
The high-order FMT moment conditions eliminate low-index contributions,
while an explicit family of classical binary covariants separates all
surviving directions.  The resulting rank formula delivers a degreewise
correction criterion and reveals an intrinsic obstruction to smooth
two-sided entropy transfer: in particular, whenever $k\ge4$ and
$f''(u_*)\eta^{(4)}(u_*)\ne0$, the degree-five cohomology class is strictly
non-zero.

Throughout this section, derivatives of the smooth functions $f$ and $\eta$
are evaluated at $u_*$.  Each statement at degree $D$ uses only derivatives
of order at most $D$.  Here $D$ is the homogeneous Taylor degree in the
coefficients of $U$, distinct from both the polynomial degree of $U$ and
the ENO interpolation order $k$.

\subsection{Apolar tests and their index}
\label{sec:apolar}
\label{sec:binary-covariants}

A polynomial residual is a local flux difference if and only if it is
annihilated by every linear functional on $\mathscr P_{m,D}$ that vanishes on
$\operatorname{Range}\delta_{m,D}$.  Taylor coordinates allow us to realize these dual
functionals via apolar differential operators.  Fix integers $m,D\ge0$,
set $n:=m+1$, and express the polynomial $U$ in its Taylor basis at the
origin:
\begin{equation}\label{eq:taylor-jet}
	U(t)=\sum_{j=0}^n x_j\frac{t^j}{j!},\qquad x_j=U^{(j)}(0).
\end{equation}
Let $y_0,\ldots,y_n$ be independent formal variables representing
$\partial_{x_0},\ldots,\partial_{x_n}$ in the apolar pairing.
The linear change of coordinates between $(u,c_0,\ldots,c_m)$ and
$x=(x_0,\ldots,x_n)$ is invertible.  Under the spatial shift $t\mapsto t+1$,
Taylor's theorem maps $x_j$ to $\sum_{q=0}^{n-j}x_{j+q}/q!$, so that the 
pull-back operator on polynomials is given by $T_m^*=e^{\mathcal N}$, where 
\begin{equation}\label{eq:ND}
	\mathcal N:=\sum_{j=0}^{n-1}x_{j+1}\partial_{x_j},\qquad
	\mathcal D:=\sum_{j=0}^{n-1}y_j\partial_{y_{j+1}}.
\end{equation}
On homogeneous polynomials of degree $D$, the derivation $\mathcal N$ is
nilpotent.  Factoring the difference operator yields
\begin{equation}\label{eq:delta-factorization}
	\delta_{m,D}=\mathcal N\Phi(\mathcal N),\qquad
	\Phi(z):=\frac{e^z-1}{z},\qquad
	\Range\delta_{m,D}=\Range\mathcal N,
\end{equation}
where $\Phi(0):=1$, and the equality of ranges follows because
$\Phi(\mathcal N)$ is an invertible polynomial in $\mathcal N$ with
constant term $1$.  The standard
Fischer (apolar) inner product $\langle A,P\rangle:=A(\partial_x)P(0)$
satisfies the formal adjoint relation
\begin{equation}\label{eq:fischer-adjoint}
	\langle A,\mathcal NP\rangle=\langle\mathcal DA,P\rangle.
\end{equation}

For $\mathbb K=\mathbb R$ or $\mathbb C$, the notation
$\mathbb K[y_0,\ldots,y_n]_D$ denotes the homogeneous degree-$D$
polynomials, and $V^\vee$ denotes the linear dual of a vector space $V$.
For a multi-index $\alpha=(\alpha_0,\ldots,\alpha_n)$ of nonnegative
integers, write $y^\alpha:=\prod_r y_r^{\alpha_r}$ and
$\alpha!:=\prod_r\alpha_r!$.

\begin{lemma}[Apolar duality]\label{lem:apolar-principle}
	For $A\in\mathbb R[y_0,\ldots,y_n]_D$, define the linear functional
	$\Lambda_A(P):=A(\partial_x)P(0)$.  Then
	\[
	\begin{gathered}
		\Lambda_A\circ\delta_{m,D}=0\quad\Longleftrightarrow\quad
		\mathcal DA=0,\\
		(H^1_{m,D})^\vee\simeq\ker\mathcal D\big|_{\mathbb R[y_0,\ldots,y_n]_D}.
	\end{gathered}
	\]
	In particular, a cohomology class $[G]\in H^1_{m,D}$ vanishes if and only
	if $\Lambda_A(G)=0$ for all $A\in\ker\mathcal D\cap\mathbb R[y_0,\ldots,y_n]_D$.
\end{lemma}
\begin{proof}
	By \eqref{eq:delta-factorization}, annihilating $\Range\delta_{m,D}$ is
	equivalent to annihilating $\Range\mathcal N$.  Fischer duality
	\eqref{eq:fischer-adjoint} shows that $\langle A,\mathcal NP\rangle=0$
	for all $P\in\mathscr P_{m,D}$ if and only if $\mathcal DA=0$.  Since the
	Fischer pairing is positive definite (with
	$\langle y^\alpha,x^\beta\rangle=\alpha!\mathbf1_{\alpha=\beta}$), every
	linear functional on the quotient space $H^1_{m,D}$ is uniquely represented
	by an element of $\ker\mathcal D$, and the family of functionals $\Lambda_A$
	separates all non-zero cohomology classes.
\end{proof}

The FMT numerical flux involves pairs of nodal values at various lattice
separations $\rho$.  To evaluate an apolar test on such sums, we decompose
the dual space by the homogeneous \emph{index} $j$, which records the total
power of the separation parameter appearing in the evaluation.  Complexifying
the polynomial spaces, let
\[
W^y_{n,D}:=\mathbb C[y_0,\ldots,y_n]_D,\qquad
\mathsf h_y:=\sum_{r=0}^n(n-2r)y_r\partial_{y_r},
\]
and write $W^y_{n,D}[q]:=\{A\in W^y_{n,D}:\mathsf h_yA=qA\}$ for the weight-$q$
eigenspace.  Together with
\[
\mathcal F:=\sum_{r=0}^{n-1}(r+1)(n-r)y_{r+1}\partial_{y_r},
\]
these operators satisfy
\[
[\mathsf h_y,\mathcal D]=2\mathcal D,\qquad
[\mathsf h_y,\mathcal F]=-2\mathcal F,\qquad
[\mathcal D,\mathcal F]=\mathsf h_y,
\]
the defining relations of the Lie algebra $\mathfrak{sl}_2$ of traceless
$2\times2$ matrices.  A highest-weight vector of weight $q$ is a nonzero
$A$ satisfying $\mathcal DA=0$ and $\mathsf h_yA=qA$.
A degree-$D$ monomial $y^\alpha$ of index $\sum_{r=0}^n r\alpha_r=j$
has weight $nD-2j$.  We define the graded apolar test spaces by
\begin{equation}\label{eq:source-space-index}
	\mathscr S_{n,D,j}:=\ker\mathcal D\cap W^y_{n,D}[nD-2j],
	\qquad 0\le j\le\lfloor nD/2\rfloor,
\end{equation}
and set $\mathscr S_{n,D,j}:=0$ outside this range.  Because
$[\mathsf h_y,\mathcal D]=2\mathcal D$, the kernel of $\mathcal D$ is
weight-homogeneous.  By the representation theory of $\mathfrak{sl}_2$
\cite{FultonHarris}, highest-weight vectors
have non-negative weights $nD-2j\ge0$, yielding the canonical direct-sum
decomposition
\begin{equation}\label{eq:apolar-test-decomposition}
	\ker\mathcal D\big|_{W^y_{n,D}}
	=\bigoplus_{j=0}^{\lfloor nD/2\rfloor}\mathscr S_{n,D,j}.
\end{equation}
In addition, index one never produces non-trivial tests:
\begin{equation}\label{eq:index-one-test-space}
	\mathscr S_{n,D,1}=0.
\end{equation}
For $D\ge1$, the only degree-$D$ monomial of index one is $y_0^{D-1}y_1$,
whose image $\mathcal D(y_0^{D-1}y_1)=y_0^D\ne0$ lies outside the kernel. For $D=0$, the assertion follows from the index-range convention.

Explicit test polynomials in $\mathscr S_{n,D,j}$ are generated by classical
binary covariants.  Let $V_r:=\operatorname{Sym}^r\mathbb C^2$ be the space
of homogeneous polynomials of degree $r$ in $Z_0,Z_1$.
The group $SL_2=SL_2(\mathbb C)$ of $2\times2$ matrices with determinant one
acts by $(g\cdot F)(Z):=F(g^{\mathsf T}Z)$, with $Z$ a column vector.
In divided coordinates, write
\begin{equation}\label{eq:binary-form-divided-coordinates}
	F_y(Z_0,Z_1):=\sum_{j=0}^n y_j\frac{Z_0^{n-j}Z_1^j}{(n-j)!}.
\end{equation}
For the evaluation vector $\zeta_n(t):=(1,t,\ldots,t^n/n!)$, this convention
yields the pure power form
\begin{equation}\label{eq:zeta-pure-power}
	F_{\zeta_n(t)}(Z)=\frac{(Z_0+tZ_1)^n}{n!}.
\end{equation}
Recall that a binary covariant of degree $D$ and order $q$ is an
$SL_2$-equivariant homogeneous polynomial mapping $\mathbf B:V_n\to V_q$. Equivariance means $\mathbf B(g\cdot F)=g\cdot\mathbf B(F)$ for every
$g\in SL_2$.
The leading coefficient of any such covariant defines an apolar test whose
index is determined by the weight:

\begin{theorem}[Covariant--test principle]
	\label{thm:cohomology-covariants}
	Let $\mathbf B:V_n\to V_q$ be a non-zero $SL_2$-equivariant homogeneous
	polynomial map of degree $D$, and write
	$\mathbf B(F;Z)=\sum_{r=0}^q B_r(F)Z_0^{q-r}Z_1^r$.
	Then $B_0\ne0$, $\mathcal DB_0=0$, and $\mathsf h_yB_0=qB_0$.
	Consequently, $B_0\in\mathscr S_{n,D,j}$ with index $j=(nD-q)/2$.
\end{theorem}
\begin{proof}
	Under the divided-coordinate convention \eqref{eq:binary-form-divided-coordinates},
	the generator $Z_1\partial_{Z_0}$ of the unipotent subgroup in $\mathfrak{sl}_2$
	acts on coefficient space as $\mathcal D$.  Differentiating the equivariance
	identity along this one-parameter subgroup and equating coefficients of
	$Z_0^q$ yields $\mathcal DB_0=0$.  
Similarly, differentiating equivariance for
	$g=\operatorname{diag}(e^s,e^{-s})$ at $s=0$ gives $\mathsf h_yB_0=qB_0$.  Since the irreducible $SL_2$-representation
	$V_q$ contains no non-trivial invariant subspaces, the image of the non-zero
	covariant $\mathbf B$ must span $V_q$. In particular, the leading coefficient
	$B_0$ cannot vanish identically.  Its weight $q=nD-2j$ then gives
	$j=(nD-q)/2$, completing the proof.
\end{proof}

\subsection{Evaluating the mismatch by secant covariance}
\label{sec:complete-mismatch-spectrum}

Because the apolar functional $\Lambda_A$ is invariant under simultaneous
shifts of the evaluation points of $U$, the inner sum defining the high-order
FMT flux telescopes directly into a two-point entropy mismatch.  Evaluating
$\Lambda_A$ on the mismatch therefore reduces to understanding the restriction
of $A$ to pairs of evaluation vectors $\zeta_n(t)$.

For ENO interpolation of order $k$, we set
\[
m:=k-2,\qquad n:=m+1=k-1,\qquad
\mu:=\left\lceil\frac{k}{2}\right\rceil,
\qquad j_0:=2\mu+1.
\]
All vector spaces and pairings below are complexified.  For any apolar test
$A\in\mathscr S_{n,D,j}$, we define its \emph{secant restriction}
$R_A\in\mathbb C[c,d]_D$ by
\begin{equation}\label{eq:source-secant-restriction}
	R_A(c,d):=A\bigl(c\zeta_n(0)+d\zeta_n(1)\bigr).
\end{equation}
Since $U(t)=x\cdot\zeta_n(t)$, this restriction governs the action of
$\Lambda_A$ on polynomials depending only on $U(0)$ and $U(1)$.  We equip
$\mathbb C[c,d]$ and $\mathbb C[X,Y]$ with the two-variable Fischer pairing
normalized by
\[
\langle c^ad^b,X^rY^s\rangle_F
:=a!b!\,\mathbf1_{(a,b)=(r,s)}.
\]

For any function $g$, let $\langle g\rangle_{X,Y}:=\int_0^1 g(X+t(Y-X))\,dt$
denote its secant average.  Recalling $v'=\eta''$, $\psi'=\eta''f$, and the
two-point flux formula \eqref{eq:quadratic-ec-core}, the two-point entropy
mismatch assumes the compact form
\begin{equation}\label{eq:covariance-polynomial}
	\begin{aligned}
		\mathscr C_{u_*}^{f,\eta}(X,Y)
		&:=(Y-X)\Bigl(
		\langle\eta''(u_*+\cdot)\rangle_{X,Y}
		\langle f(u_*+\cdot)\rangle_{X,Y}\\
		&\hspace{4em}-\langle(\eta''f)(u_*+\cdot)\rangle_{X,Y}\Bigr).
	\end{aligned}
\end{equation}
The term in parentheses represents the negative covariance of $\eta''$ and $f$
with respect to uniform Lebesgue measure along the line segment from $u_*+X$
to $u_*+Y$. It quantifies the precise failure of the two-point flux $F^*$ to
balance the target entropy pair $(\eta,q)$.  Let $\mathscr C_D^{f,\eta}(X,Y)$
denote the homogeneous degree-$D$ Taylor polynomial of the function
$\mathscr C_{u_*}^{f,\eta}$ at the origin:
\[
\mathscr C_D^{f,\eta}(X,Y)
:=
\frac{1}{D!}
\left.
\frac{d^D}{d\varepsilon^D}
\mathscr C_{u_*}^{f,\eta}(\varepsilon X,\varepsilon Y)
\right|_{\varepsilon=0}.
\]
As before, $G_{\eta,k}^{(D)}\in\mathscr P_{k-2,D}$ denotes the homogeneous
degree-$D$ Taylor component of the core interface mismatch \eqref{eq:mismatch}
restricted to order-$k$ ENO null modes.

The following fundamental formula cleanly decouples the stencil geometry
from the PDE data: the stencil enters exclusively through a scalar moment
$M_{\mu,j}$, while the flux--entropy pair enters through the two-variable
covariance polynomial $\mathscr C_D^{f,\eta}$.

\begin{lemma}[Secant--covariance formula]
	\label{lem:secant-covariance-formula}
	For every apolar test $A\in\mathscr S_{n,D,j}$,
	\begin{equation}\label{eq:secant-covariance-formula}
			\Lambda_A(G_{\eta,k}^{(D)})
			=
			M_{\mu,j}\,
			\langle R_A,\mathscr C_D^{f,\eta}\rangle_F,\qquad 
			M_{\mu,j}:=\sum_{\rho=1}^{\mu}
			\alpha_\rho^{[\mu]}\rho^j.
	\end{equation}
	Moreover, the secant restriction satisfies the parity relation
	\begin{equation}\label{eq:secant-parity}
		R_A(d,c)=(-1)^jR_A(c,d).
	\end{equation}
	Consequently, every even-index apolar test annihilates the mismatch
	identically: $\Lambda_A(G_{\eta,k}^{(D)})=0$ whenever $j$ is even.
\end{lemma}

\begin{proof}
	Throughout the proof, we work modulo terms of degree higher than $D$.
	Set $u_t:=u_*+U(t)$ and $v_t:=v(u_t)$.  The identity $\psi'=\eta''f$
	implies that the two-point mismatch
	\[
	[v(u_*+Y)-v(u_*+X)]F^*(u_*+X,u_*+Y)
	-[\psi(u_*+Y)-\psi(u_*+X)]
	\]
	coincides with $\mathscr C_{u_*}^{f,\eta}(X,Y)$ in \eqref{eq:covariance-polynomial}.
	For any $s\in\mathbb R$, Fischer adjointness gives
	$\Lambda_A(e^{s\mathcal N}P)=\Lambda_{e^{s\mathcal D}A}(P)=\Lambda_A(P)$,
	since $\mathcal DA=0$.  The functional $\Lambda_A$ is therefore invariant
	under the global translation $U(t)\mapsto U(t+s)$.  In the $\ell$-th summand
	of the inner FMT sum in \eqref{eq:FMT-core}, shifting the evaluation parameter
	by $\ell$ replaces the flux arguments with $(u_0,u_\rho)$ and the entropy
	jump with $v_{\ell+1}-v_\ell$.  Summing over $\ell$ telescopes the jump:
	\[
	\begin{aligned}
		\sum_{\ell=0}^{\rho-1}
		\Lambda_A\bigl((v_1-v_0)F^*(u_{-\ell},u_{\rho-\ell})\bigr)
		&=\Lambda_A\!\left(
		\sum_{\ell=0}^{\rho-1}(v_{\ell+1}-v_\ell)F^*(u_0,u_\rho)\right)\\
		&=\Lambda_A\bigl((v_\rho-v_0)F^*(u_0,u_\rho)\bigr).
	\end{aligned}
	\]
	Since any potential difference $[\psi]_{i+1/2}$ is a discrete shift coboundary,
	it is annihilated by $\Lambda_A$.  Subtracting the corresponding potential
	difference at separation $\rho$, the operand reduces to
	$\mathscr C_{u_*}^{f,\eta}(U(0),U(\rho))$.
	
	For $a+b=D$, polarization yields
	\[
	\Lambda_A(U(0)^aU(\rho)^b)
	=a!b![c^ad^b]
	A(c\zeta_n(0)+d\zeta_n(\rho)).
	\]
	Under the substitution $y=c\zeta_n(0)+d\zeta_n(\rho)$, we have
	$y_0=c+d$ and $y_r=d\rho^r/r!$ for $1\le r\le n$.  Since every monomial in
	$A\in\mathscr S_{n,D,j}$ has total index $\sum r\alpha_r=j$, it acquires
	precisely the factor $\rho^j$, giving
	\[
	A(c\zeta_n(0)+d\zeta_n(\rho))=\rho^jR_A(c,d).
	\]
	Summing over $\rho=1,\ldots,\mu$ with the FMT weights $\alpha_\rho^{[\mu]}$
	proves \eqref{eq:secant-covariance-formula}.
	
	To establish the parity relation \eqref{eq:secant-parity}, exchanging
	$c$ and $d$ and shifting both evaluation points in $(0,1)$ by $-1$ transforms
	the pair into $(-1,0)$.  Because a test of index $j$ is homogeneous of
	degree $j$ in the separation vector, replacing $1$ by $-1$ introduces the
	factor $(-1)^j$, proving \eqref{eq:secant-parity}.  Finally, the covariance
	polynomial $\mathscr C_D^{f,\eta}(X,Y)$ is antisymmetric under $X\leftrightarrow Y$. Its Fischer pairing with any symmetric polynomial $R_A$ (which occurs when
	$j$ is even) vanishes identically, completing the proof.
\end{proof}

\subsection{The surviving moments and a separating family of tests}
\label{sec:FMT-apolar-tests}

Symmetry eliminates all even indices.  Among the odd indices, the
accuracy conditions \eqref{eq:FMT-moments} annihilate every moment between
$3$ and $2\mu-1$.  The following identity shows that all subsequent
odd moments are strictly nonzero, ensuring that the stencil coefficients
introduce no further cancellation.

\begin{lemma}[Odd FMT moments beyond the accuracy range]
	\label{lem:all-free-FMT-moments}
	The generating function of the odd moments is
	\begin{equation}\label{eq:FMT-moment-generating-function}
		\sum_{r=0}^{\infty}M_{\mu,2r+1}t^r
		=1+(-1)^{\mu+1}
		\frac{(\mu!)^2t^\mu}
		{\prod_{\rho=1}^{\mu}(1-\rho^2t)}.
	\end{equation}
	Consequently,
	\begin{align}
		M_{\mu,1}&=1,\notag\\
		M_{\mu,2r+1}&=0 &&(1\le r\le\mu-1),\notag\\
		M_{\mu,2r+1}
		&=(-1)^{\mu+1}(\mu!)^2
		h_{r-\mu}(1^2,2^2,\ldots,\mu^2)\ne0
		&&(r\ge\mu),
		\label{eq:all-free-FMT-moments}
	\end{align}
	where
	\[
	h_q(z_1,\ldots,z_\mu)
	:=\sum_{\substack{\nu_1+\cdots+\nu_\mu=q\\\nu_r\ge0}}
	z_1^{\nu_1}\cdots z_\mu^{\nu_\mu},\qquad h_0=1,
	\]
	is the complete homogeneous symmetric polynomial.  Thus
	$j_0:=2\mu+1$ is the first odd index capable of contributing to the
	mismatch: all smaller odd moments vanish except $M_{\mu,1}=1$, and
	there is no nonzero apolar test of index one.
\end{lemma}

\begin{proof}
	By definition of the stencil weights $\alpha_\rho^{[\mu]}$,
	\[
	\sum_{r\ge0}M_{\mu,2r+1}t^r
	=\sum_{\rho=1}^{\mu}
	\frac{\alpha_\rho^{[\mu]}\rho}{1-\rho^2t}.
	\]
	Matching the partial-fraction residues at the simple poles
	$t=\rho^{-2}$ gives \eqref{eq:FMT-moment-generating-function}, with the
	value at $t=0$ fixing the constant term.  Expanding
	$\prod_{\rho=1}^{\mu}(1-\rho^2t)^{-1}$ in power series then yields
	\eqref{eq:all-free-FMT-moments}.  Because each $h_q$ is strictly positive
	when evaluated at $1^2,\ldots,\mu^2$, these moments never vanish.
	
	The index-one assertion follows from
	\eqref{eq:index-one-test-space}.
\end{proof}

To detect a surviving class, we need an apolar test whose index is at
least $j_0$ and whose secant restriction pairs nontrivially with the
covariance polynomial.  Such a test exists at Taylor degree five for every
$n\ge3$.  To construct tests at higher degrees, we introduce the
quadratic invariant of the unipotent action generated by $\mathcal D$
\begin{equation}\label{eq:quadratic-source-I2}
	\mathcal I_2:=y_1^2-2y_0y_2.
\end{equation}
Multiplication by $y_0$ and $\mathcal I_2$ preserves the apolar
condition and generates tests at every higher degree.  Indeed, since
\[
\mathcal Dy_0=0,\qquad
\mathcal Dy_1=y_0,\qquad
\mathcal Dy_2=y_1,
\]
we have
\[
\mathcal D\mathcal I_2
=
2y_1y_0-2y_0y_1
=0.
\]
Furthermore, for every $n\ge2$,
\begin{equation}\label{eq:I2-secant-restriction}
	\mathcal I_2
	\bigl(c\zeta_n(t)+d\zeta_n(s)\bigr)
	=
	-cd(s-t)^2.
\end{equation}

\begin{lemma}[A degree-five apolar functional]
	\label{lem:universal-degree-five-source}
	Assume $n\ge3$, and set
	\[
	\epsilon:=j_0-n=
	\begin{cases}2,&n\text{ odd},\\3,&n\text{ even}.
	\end{cases}
	\]
	There is a nonzero apolar test
	\[
	\mathfrak B_n\in\mathscr S_{n,5,j_0}
	\]
	and a constant $\kappa_n\ne0$ such that
	\begin{equation}\label{eq:universal-degree-five-secant}
		\mathfrak B_n(c\zeta_n(t)+d\zeta_n(s))
		=\kappa_n(s-t)^{j_0}(cd)^2(c-d).
	\end{equation}
\end{lemma}

Appendix~\ref{app:invariant-details},
Subsection~\ref{app:proof-degree-five-source}, carries out the explicit
construction of $\mathfrak B_n$ from binary covariants.  For the mismatch
calculation, only its secant restriction is required.  For $n\ge3$ and
$D\ge5$, let $L:=D-5$ and define
\begin{equation}\label{eq:apolar-test-family}
	\mathfrak B_{n,D,a}
	:=y_0^{L-2a}\mathcal I_2^a\mathfrak B_n,
	\qquad
	0\le a\le\left\lfloor\frac L2\right\rfloor.
\end{equation}
Because the three factors lie in $\ker\mathcal D$, with degrees and
indices adding additively, we have
\begin{equation}\label{eq:apolar-test-source-space}
	\mathfrak B_{n,D,a}\in\mathscr S_{n,D,j_0+2a}.
\end{equation}
Furthermore,
\begin{equation}\label{eq:apolar-test-secant}
	\begin{aligned}
		&\mathfrak B_{n,D,a}
		\bigl(c\zeta_n(0)+d\zeta_n(\rho)\bigr)
		\\
		&\qquad=
		\kappa_{n,a}\rho^{j_0+2a}
		(c+d)^{L-2a}(cd)^{a+2}(c-d),
	\end{aligned}
\end{equation}
where $\kappa_{n,a}\ne0$.  All corresponding FMT moments $M_{\mu,j_0+2a}$
are nonzero by Lemma~\ref{lem:all-free-FMT-moments}.

These tests suffice because the covariance polynomials span a
low-dimensional subspace of bivariate polynomials.  The next lemma
gives a basis of this subspace, isolates the direction eliminated by the
full FMT moment sum, and establishes that the tests above separate all
remaining directions.

\begin{lemma}[Nondegeneracy on the covariance space]
	\label{lem:covariance-nondegeneracy}
	Fix integers $k\ge2$ and $D\ge3$, and set
	\[
	n:=k-1,\qquad
	\mu:=\left\lceil\frac{k}{2}\right\rceil,
	\qquad
	j_0:=2\mu+1.
	\]
	For $r\ge0$, define
	\[
	\begin{gathered}
		P_r(X,Y):=\frac1{r+1}\sum_{\ell=0}^{r}X^{r-\ell}Y^\ell,
		\\
		C_{r,s}(X,Y):=(Y-X)(P_rP_s-P_{r+s}).
	\end{gathered}
	\]
	Then the polynomials $C_{r,s}$, with
	$1\le r\le s$ and $r+s=D-1$, form a basis of
	\[
	\mathscr V_D:=\{H\in\mathbb C[X,Y]_D:
	H(Y,X)=-H(X,Y),\ (Y-X)^3\mid H\}.
	\]
	
	For every index $j$ and every $A\in\mathscr S_{n,D,j}$,
	\begin{equation}\label{eq:endpoint-all-tests}
		M_{\mu,j}\langle R_A,C_{1,D-2}\rangle_F=0.
	\end{equation}
	Thus the endpoint direction is annihilated by \emph{all} apolar
	functionals after the full FMT moment sum.
	
	If, in addition, $k\ge4$ and $D\ge5$, put
	\[
	\mathscr V_D^{\mathrm{int}}
	:=\operatorname{span}\{C_{r,s}:2\le r\le s,\ r+s=D-1\}.
	\]
	For $H\in\mathscr V_D^{\mathrm{int}}$,
	\[
	\begin{gathered}
		H=0\quad\Longleftrightarrow\\
		\langle R_{\mathfrak B_{n,D,a}},H\rangle_F=0
		\quad\text{for every }0\le a\le\lfloor(D-5)/2\rfloor.
	\end{gathered}
	\]
	This equivalence remains valid after multiplying the $a$-th pairing
	by $M_{\mu,j_0+2a}$, since all these moments are nonzero.
\end{lemma}

The endpoint assertion gives an upper bound on the dimension of the
mismatch image, while the separating family provides the matching lower
bound.  We outline the argument below, deferring the explicit coefficient
evaluations to Appendix~\ref{app:proof-covariance-nondegeneracy}.

Regarding the endpoint assertion, observe that $C_{1,D-2}$ contains only
pure powers and monomials in which one variable has exponent one.
Even-index tests vanish identically by parity.  For odd $j$, the pure-power
coefficients of $R_A$ vanish. A coefficient of $c^{D-1}d$ can occur only from
$y_0^{D-1}y_j$, requiring $j\le n$.  Index one admits no apolar test, and
every remaining odd $j\le n$ lies strictly below $j_0$, where
$M_{\mu,j}=0$.  This establishes \eqref{eq:endpoint-all-tests} for
every test, independently of the degree-five construction.

For separation, multiplication of the secant symbol by
$(cd)^2(c-d)$ is Fischer-adjoint to the differential operator
\[
\mathscr L_D:=\partial_X^2\partial_Y^2(\partial_X-\partial_Y).
\]
The polynomials $(c+d)^{L-2a}(cd)^a$, $0\le a\le\lfloor L/2\rfloor$,
form a basis of the symmetric homogeneous polynomials of degree $L=D-5$.
Write $\mathbb C[X,Y]_L^{\mathfrak S_2}$ for this space in $X,Y$, where
$\mathfrak S_2$ acts by interchanging the two variables.
As shown in Appendix~\ref{app:proof-covariance-nondegeneracy}, the map
\[
\mathscr L_D:\mathscr V_D\longrightarrow
\mathbb C[X,Y]_{D-5}^{\mathfrak S_2}
\]
is surjective with kernel precisely $\operatorname{span}\{C_{1,D-2}\}$.
Its restriction to $\mathscr V_D^{\mathrm{int}}$ is therefore an
isomorphism.  Nondegeneracy of the Fischer pairing then yields the stated
separating property.  Thus the mismatch image is determined without
constructing a basis of the entire cohomology group.

\subsection{Taylor coefficients of the mismatch class}
\label{sec:mismatch-rank}

Expanding the covariance converts the preceding pairing results into an
explicit criterion in terms of the derivatives of $f$ and $\eta$ at $u_*$.
The surviving terms involve derivatives of $f$ of order at least two
and derivatives of $\eta$ of order at least four.  For $a\ge0$, we introduce
the Taylor coefficients of $f$ and $\eta''$
\begin{equation}\label{eq:divided-flux-entropy-jets}
	F_a:=\frac{f^{(a+2)}(u_*)}{(a+2)!},
	\qquad
	E_a:=\frac{\eta^{(a+4)}(u_*)}{(a+2)!}.
\end{equation}

For $D\ge5$, put
\[
L:=D-5.
\]
Let $\mathcal J_L$ denote the complex vector space spanned by the basis
\[
e_{\{a,b\}},
\qquad
0\le a\le b,
\qquad
a+b=L,
\]
with the unordered convention $e_{\{a,b\}}=e_{\{b,a\}}$.

To each flux--entropy pair, we associate the symmetric coefficient vector
\begin{equation}\label{eq:symmetric-jet-vector}
	\begin{split}
		\mathbf j_L(f,\eta)
		:={}
		\sum_{\substack{a<b\\a+b=L}}
		\bigl(F_aE_b+F_bE_a\bigr)e_{\{a,b\}}
		+
		\sum_{\substack{a\ge0\\2a=L}}
		F_aE_a\,e_{\{a,a\}}.
	\end{split}
\end{equation}

For every $D\ge0$, expanding the centered covariance function at the origin gives
\begin{equation}\label{eq:covariance-jet-expansion}
	\mathscr C_D^{f,\eta}
	=
	\sum_{\substack{r,s\ge1\\r+s=D-1}}
	\frac{\eta^{(r+2)}(u_*)}{r!}
	\frac{f^{(s)}(u_*)}{s!}\,C_{r,s}.
\end{equation}
When $D\ge5$, the terms with $r=1$ or $s=1$ lie in the single endpoint
direction $C_{1,D-2}$.  By Lemma~\ref{lem:covariance-nondegeneracy}, their
contribution to the FMT mismatch pairs to zero with every apolar test and
therefore represents the trivial cohomology class.  Modulo this direction,
the bivariate expansion reduces to
\begin{align}
	\mathscr C_D^{f,\eta}
	\equiv{}&
	\sum_{\substack{a<b\\a+b=L}}
	\bigl(F_aE_b+F_bE_a\bigr)C_{a+2,b+2}
	\notag\\
	&+
	\sum_{\substack{a\ge0\\2a=L}}
	F_aE_a\,C_{a+2,a+2}
	\pmod{\operatorname{span}\{C_{1,D-2}\}}.
	\label{eq:covariance-interior-expansion}
\end{align}

The symmetric coefficient vector determines the cohomology class uniquely.  Indeed,
any two flux--entropy pairs sharing the same $\mathbf j_L$ have covariance
polynomials differing only by a scalar multiple of the endpoint direction
$C_{1,D-2}$.  In light of \eqref{eq:endpoint-all-tests} and the
secant--covariance formula \eqref{eq:secant-covariance-formula}, their
mismatch classes pair identically with every apolar test, and thus coincide.

Every basis vector $e_{\{a,b\}}$ can be realized by real polynomials $f,\eta$:
it suffices to set $F_a=E_b=1$ with all other relevant coefficients vanishing
(taking $F_a=E_a=1$ on the diagonal).  Adding a positive quadratic term to the
entropy leaves these higher Taylor coefficients unaffected while ensuring local
strict convexity.  We then define $\mathfrak M_{k,D}(e_{\{a,b\}})$ to be the
resulting class and extend by linearity.  By the preceding observation, this
definition is independent of the particular realizing polynomials chosen.  For any
coefficient vector generated by a flux--entropy pair, the corresponding linear
combination of basis classes pairs with every apolar test exactly as
$[G_{\eta,k}^{(D)}]$ does.  Apolar duality therefore yields
\begin{equation}\label{eq:linearized-mismatch-map}
	\begin{gathered}
		\mathfrak M_{k,D}:\mathcal J_{D-5}\longrightarrow
		H^1_{k-2,D}\otimes_{\mathbb R}\mathbb C,\\
		[G_{\eta,k}^{(D)}]\otimes1
		=\mathfrak M_{k,D}\bigl(\mathbf j_{D-5}(f,\eta)\bigr).
	\end{gathered}
\end{equation}
The realizability of each basis vector also demonstrates that
\[
\operatorname{Range}\mathfrak M_{k,D}=\mathcal M_{k,D},
\]
where $\mathcal M_{k,D}$ denotes the complex span of the mismatch classes
obtained by varying the smooth functions $f,\eta$ near $u_*$. A fixed pair yields the
single class in \eqref{eq:linearized-mismatch-map}.

The same expansion clarifies why all mismatch classes vanish identically
below degree five.
For $D\le2$, the summation in \eqref{eq:covariance-jet-expansion} is empty, so
\[
\mathscr C_D^{f,\eta}=0.
\]
For $D=3,4$, every summand in \eqref{eq:covariance-jet-expansion} belongs to
\[
\operatorname{span}\{C_{1,D-2}\}.
\]
Lemma~\ref{lem:covariance-nondegeneracy} establishes that this endpoint
contribution produces a vanishing mismatch class upon summing against the
FMT moments.  The first potential interior term, namely $C_{2,2}$, appears
only at $D=5$.  Consequently,
\begin{equation}\label{eq:low-degree-mismatch-exact}
	[G_{\eta,k}^{(D)}]=0
	\qquad (D<5).
\end{equation}
For $D<5$, we adopt the convention that $\mathfrak M_{k,D}$ is the zero map
defined on the zero vector space $\{0\}$, consistent with
\eqref{eq:low-degree-mismatch-exact}.

\begin{theorem}[Degreewise rank of the mismatch map]
	\label{thm:complete-mismatch-rank}
	For every integer $k\ge2$ and homogeneous Taylor degree $D\ge0$,
	let $\mathcal M_{k,D}$ be the complex span of the FMT mismatch
	classes obtained from smooth real fluxes and locally strictly convex
	entropies defined near $u_*$, evaluated on order-$k$ ENO null modes, whose jump sequences
have polynomial degree at most $k-2$.  For the map $\mathfrak M_{k,D}$ defined in
	\eqref{eq:linearized-mismatch-map}, with the zero-space convention
	when $D<5$,
	\begin{equation}\label{eq:complete-mismatch-rank}
		\operatorname{rank}\mathfrak M_{k,D}
		=
		\begin{cases}
			0,&k=2,3,\\[1mm]
			0,&k\ge4,\ D\le4,\\[1mm]
			\left\lfloor\dfrac{D-3}{2}\right\rfloor,
			&k\ge4,\ D\ge5.
		\end{cases}
	\end{equation}
	By construction,
	\[
	\mathcal M_{k,D}
	=
	\operatorname{Range}\mathfrak M_{k,D}
	\subset
	H^1_{k-2,D}\otimes_{\mathbb R}\mathbb C.
	\]
	Moreover, for $k\ge4$ and $D\ge5$, the map is injective and
	\begin{equation}\label{eq:mismatch-image-symmetric-polynomials}
		\mathcal M_{k,D}
		\simeq
		\mathcal J_{D-5}
		\simeq
		\mathbb C[X,Y]_{D-5}^{\mathfrak S_2},
	\end{equation}
	where the second isomorphism is determined by
	\[
	e_{\{a,b\}}
	\longleftrightarrow
	\begin{cases}
		X^aY^b+X^bY^a,&a<b,\\[1mm]
		X^aY^a,&a=b.
	\end{cases}
	\]
	Under the first isomorphism, $e_{\{a,b\}}$ corresponds to
	$\mathfrak M_{k,D}(e_{\{a,b\}})$.
\end{theorem}

\begin{proof}
	When $k\ge4$, the apolar tests \eqref{eq:apolar-test-family}, the
	nonvanishing FMT moments from Lemma~\ref{lem:all-free-FMT-moments}, and
	the separating property in Lemma~\ref{lem:covariance-nondegeneracy}
	guarantee that the image of every nonzero vector in $\mathcal J_{D-5}$
	pairs nontrivially with at least one of these functionals.
	Consequently, $\mathfrak M_{k,D}$ is injective for all $D\ge5$, yielding
	\[
	\operatorname{rank}\mathfrak M_{k,D}
	=
	\dim\mathcal J_{D-5}
	=
	\left\lfloor\frac{D-5}{2}\right\rfloor+1
	=
	\left\lfloor\frac{D-3}{2}\right\rfloor.
	\]
	For $D<5$, the low-degree calculation
	\eqref{eq:low-degree-mismatch-exact} ensures that the mismatch class
	vanishes for every admissible flux--entropy pair. Hence
	$\operatorname{rank}\mathfrak M_{k,D}=0$.
	
	For $k=2$, the binary degree is $n=1$, and only secant index
	$j=0$ can arise because $\ker(y_0\partial_{y_1})=\mathbb C[y_0]$.
	For $k=3$, the binary-quadratic decomposition \cite{FultonHarris} gives
	\[
	\operatorname{Sym}^D(V_2)
	\simeq
	\bigoplus_{a=0}^{\lfloor D/2\rfloor}V_{2D-4a}.
	\]
	The highest weight of the summand indexed by $a$ is $2D-4a$,
	so its apolar-test index is $(2D-(2D-4a))/2=2a$.
	Hence every admissible test index is even.  In both cases ($k=2,3$), the
	parity identity \eqref{eq:secant-parity} forces every test to vanish
	identically on the antisymmetric covariance polynomial, implying that
	$\mathfrak M_{k,D}\equiv 0$.
\end{proof}

\subsection{Degreewise correction and formal exactness}
\label{sec:exact-entropy-transfer}

For $k\ge4$, Theorem~\ref{thm:complete-mismatch-rank} shows that each
symmetric coefficient contributes independently to the mismatch class.
The local-correction criterion of Section~\ref{sec:shift-cohomology} thus
takes the following explicit form.

\begin{corollary}[Degreewise correction criterion]
	\label{cor:degreewise-mismatch-criterion}
	Fix $k\ge4$, smooth functions $f,\eta$ near $u_*$, and a Taylor
	degree $D\ge5$.  Set $L=D-5$, and define $F_a,E_a$ by
	\eqref{eq:divided-flux-entropy-jets}.  Then
	\begin{equation}\label{eq:degreewise-jet-criterion}
		\begin{gathered}
			[G_{\eta,k}^{(D)}]=0\ \text{in }H^1_{k-2,D}\\
			\Longleftrightarrow\quad
			\begin{cases}
				F_aE_b+F_bE_a=0,&a<b,\ a+b=L,\\
				F_aE_a=0,&2a=L.
			\end{cases}
		\end{gathered}
	\end{equation}
	Equivalently, there exists a translation-invariant finite-stencil
	homogeneous numerical-entropy-flux correction, vanishing on constant data,
	whose interface difference cancels the degree-$D$ FMT-core cell residual
	on all order-$k$ ENO null modes.
	Furthermore, every degree below five is automatically correctable in the
	same sense.
\end{corollary}
\begin{proof}
	A real cohomology class vanishes if and only if its complexification
	vanishes.  By \eqref{eq:linearized-mismatch-map} and the injectivity of
	$\mathfrak M_{k,D}$, this occurs if and only if
	$\mathbf j_L(f,\eta)=0$, which is precisely the displayed system of
	quadratic relations.  Lemma~\ref{lem:solvability} equates the vanishing
	of the cohomology class with the cancellation of the cell residual via an
	interface flux difference, while
	Lemma~\ref{lem:FMT-consistency-normalization} ensures that the correction
	can be chosen to vanish identically on constant states.  The assertion for
	degrees below five follows directly from
	\eqref{eq:low-degree-mismatch-exact} together with these two lemmas.
\end{proof}

Requiring correction across every homogeneous Taylor degree imposes a
constraint on all Taylor coefficients of $f$ and $\eta$ at $u_*$.  This condition is purely
formal: it does not, by itself, address the analytic convergence of the
assembled correction series.

\begin{corollary}[Formal exactness and the first nonzero degree]
	\label{cor:formal-mismatch-exactness}
	Fix $k\ge4$ and smooth functions $f,\eta$ near $u_*$.  Every
	homogeneous Taylor degree is correctable on the prescribed order-$k$ ENO null modes if and only if
	\begin{equation}\label{eq:formal-mismatch-dichotomy}
		\begin{gathered}
			f^{(r)}(u_*)=0\quad\text{for every }r\ge2,\\
			\text{or}\qquad \eta^{(r)}(u_*)=0\quad\text{for every }r\ge4.
		\end{gathered}
	\end{equation}
	If $a_0:=\operatorname{ord}_{u_*}f''$ and
	$b_0:=\operatorname{ord}_{u_*}\eta^{(4)}$ are finite, then the first
	nonvanishing mismatch class occurs at degree
	\[
	D_{\min}=5+a_0+b_0,
	\]
	where $\operatorname{ord}_{u_*}g:=\min\{r\ge0:g^{(r)}(u_*)\ne0\}$,
	with $\min\varnothing:=\infty$.
\end{corollary}
\begin{proof}
	If every $F_a$ vanishes, the first alternative in
	\eqref{eq:formal-mismatch-dichotomy} holds.  Otherwise, pick an index $a$
	such that $F_a\ne0$.  The diagonal condition at total index $2a$ in
	\eqref{eq:degreewise-jet-criterion} forces $E_a=0$.  The off-diagonal
	equations at total indices $a+b$ then yield $F_aE_b=0$ for every $b\ne a$,
	implying that $E_b=0$ for all $b\ge0$.  Conversely, if either derivative condition holds, every coordinate of
	the symmetric coefficient vector in
	\eqref{eq:symmetric-jet-vector} is zero, and the classes below degree five
	vanish automatically.
	
	If $a_0$ and $b_0$ are finite, every symmetric coefficient at degree
	$L<a_0+b_0$ vanishes by definition of $a_0,b_0$.  At total degree
	$L=a_0+b_0$, the coefficient of $e_{\{a_0,b_0\}}$ is strictly nonvanishing:
	if $a_0=b_0$ it equals $F_{a_0}E_{a_0}\ne0$. Otherwise the cross term
	$F_{b_0}E_{a_0}$ vanishes, since $E_{a_0}=0$ when $a_0<b_0$ and
	$F_{b_0}=0$ when $a_0>b_0$, leaving $F_{a_0}E_{b_0}\ne0$.  Corollary~\ref{cor:degreewise-mismatch-criterion}
	therefore identifies $D_{\min}=5+a_0+b_0$ as the first degree at which
	the mismatch class is nonzero.
\end{proof}

An analytic function whose derivatives of all orders vanish at $u_*$ is
identically zero near $u_*$.  Applying this to $f''$ or $\eta^{(4)}$ shows
that, under analyticity, formal exactness implies that $f$ is affine near
$u_*$ or that $\eta$ is a polynomial of degree at most three there.
For general smooth functions, \eqref{eq:formal-mismatch-dichotomy} concerns
Taylor coefficients alone. It neither implies convergence of the formal
correction series nor controls non-analytic remainders whose derivatives
of all orders vanish at $u_*$.

For polynomial fluxes $f$ and entropies $\eta$, formal exactness is
equivalent to an authentic polynomial correction: there exists a single
translation-invariant, finite-stencil polynomial numerical-entropy-flux
correction, vanishing on constant data, whose interface difference cancels the
entire FMT-core cell residual on these null modes.  Indeed, in the polynomial
setting, only finitely many homogeneous degree residuals can be nonzero.
Corollary~\ref{cor:degreewise-mismatch-criterion} supplies a consistent
correction for each active degree, and their sum yields a polynomial
correction supported on the union of their finite stencils.  Conversely,
taking homogeneous components of any such correction recovers formal
exactness.  This finite superposition upgrades formal exactness to an exact
polynomial transfer in the polynomial case.

\subsection{From exact correction to entropy-stability transfer}
\label{sec:stability-transfer}
\label{sec:cohomology-synthesis}

To pass from exact correction to an entropy inequality, we exploit the odd
degree of the first generic obstruction.  A homogeneous degree-five polynomial
residual is odd under the reflection $z\mapsto-z$. Hence, it can be nonpositive
on a neighborhood of the constant state $z=0$ only if it vanishes identically.
Consequently, a nonvanishing degree-five cohomology class precludes any local
entropy inequality once all lower-degree residuals have been canceled.

\begin{corollary}[The first nonexact term cannot be made one-sided]
	\label{cor:degree-five-one-sided-failure}
	Let $k\ge4$ and suppose
	\[
	f''(u_*)\eta^{(4)}(u_*)\ne0.
	\]
	Consider a translation-invariant finite-stencil entropy-flux correction
	whose restriction to ENO null modes is of class $C^5$ near zero.
	If its corrected cell residual has no Taylor terms below degree five,
	then the degree-five term is nonzero and indefinite (taking both signs).
	Consequently, the full residual cannot be nonpositive on any neighborhood
	of the constant state within the space of order-$k$ ENO null modes.
\end{corollary}

\begin{proof}
	The mismatch classes vanish in all degrees below five.  At $D=5$,
	Corollary~\ref{cor:degreewise-mismatch-criterion} yields
	\[
	[G_{\eta,k}^{(5)}]\ne0
	\qquad\text{in }H^1_{k-2,5}.
	\]
	Fix any degree-five local correction, and let $P_5$ denote the resulting
	homogeneous degree-five Taylor term of the corrected residual.  Because the
	cohomology class is nonvanishing, $P_5$ cannot be a coboundary, and in
	particular $P_5\not\equiv 0$.  As $P_5$ has odd degree, it must take both
	positive and negative values. Hence there exists $z_0$ such that
	$P_5(z_0)>0$.
	
	Once the residuals of degrees below five are eliminated, Taylor expansion
	about $z=0$ gives
	\[
	\mathcal R(tz_0)
	=
	t^5P_5(z_0)+o(t^5)
	\qquad(t\downarrow0).
	\]
	Therefore $\mathcal R(tz_0)>0$ for all sufficiently small $t>0$,
	precluding nonpositivity on any neighborhood of the constant state.
\end{proof}

For an isolated target entropy $\eta$, a negative definite term of lower even
degree (such as degree four) could conceivably dominate the degree-five
residual.  Corollary~\ref{cor:degree-five-one-sided-failure} therefore does
not, by itself, rule out a localized inequality for an arbitrary individual
entropy.

For a family of entropies bifurcating from the exact quadratic balance,
however, first variation yields a decisive obstruction.  If the cell residual
remains nonpositive for both signs of the entropy perturbation and vanishes
identically at the quadratic base point, its first variation with respect to the
perturbation parameter must vanish.  This, in turn, would generate an exact
local correction for the perturbing entropy, directly contradicting the
nontriviality of the degree-five class.  Making this argument rigorous requires
a fixed finite stencil and a uniform neighborhood in the null-mode parameter
space on which the family satisfies mild regularity.

To formulate this condition, fix $k\ge4$, set $m:=k-2$, and parametrize the
space of order-$k$ ENO null modes by
\[
\mathcal Z_k:=\mathbb R^{m+2}=\mathbb R^k,
\qquad
z=(u,c_0,\ldots,c_m).
\]
The constant nodal state $u_j\equiv u_*$ corresponds to $z=0$.  Each
$z\in\mathcal Z_k$ determines the null-mode sequence
$u_{i+\ell}(z)=u_*+U_z(\ell)$ via \eqref{eq:profile-U}, centered at the
cell $i$ under evaluation.  Throughout what follows, neighborhoods of the
constant state refer to neighborhoods of the origin in $\mathcal Z_k$.

\begin{theorem}[No smooth two-sided local transfer family]
	\label{thm:no-two-sided-transfer}
	Let $k\ge4$, let $u_*$ lie in the interior of an interval $I$, and let
	$f,\phi$ be smooth near $I$.  Assume
	\begin{equation}\label{eq:generic-transfer-condition}
		f''(u_*)\phi^{(4)}(u_*)\ne0.
	\end{equation}
	For 
	$
	\eta_\theta(u):=\frac12u^2+\theta\phi(u),
	$ 
	suppose that $\eta_\theta$ is strictly convex on $I$ for every
	$|\theta|<\theta_0$, and let $q_\theta'=\eta_\theta'f'$.  Keep fixed the
	scalar TeCNO numerical flux associated with $\eta_2$
	\eqref{eq:full-scalar-tecno-flux} in the conservative update.
	
	There is no family of numerical entropy fluxes $\mathcal Q^\theta$,
	$|\theta|<\theta_0$, with the following properties:
	\begin{enumerate}[label=\textup{(\roman*)},leftmargin=2.2em]
		\item Each numerical entropy flux is translation invariant, consistent with
		$q_\theta$, and uses one finite stencil radius independent of
		$\theta$. 
At the fixed cell $i$, let $\mathcal I_i$ contain all
		indices entering $\mathcal Q^\theta_{i\pm1/2}$ and the restricted
		conservation-law fluxes $\widehat F^k_{i\pm1/2}$.  The latter equal
		$\widetilde F^k_{i\pm1/2}$ by
		Corollary~\ref{cor:null-mode-restriction}.  
Thus $\mathcal I_i$
		can be chosen finite and independent of $\theta$.
		\item There is one relatively compact neighborhood
		$\mathcal U\Subset\mathcal Z_k$ of $0$ such that the restricted numerical
		entropy fluxes are defined on a neighborhood of $\overline{\mathcal U}$ and
		\[
		\theta\longmapsto\mathcal Q^\theta
		\quad\text{belongs to}\quad
		C^1\!\left((-\theta_0,\theta_0);
		C^5(\overline{\mathcal U})\right);
		\]
		\item There is one neighborhood
		$\mathcal V\Subset\mathcal U$ of $0$, independent of $\theta$, on which
		the following condition holds.  For every $|\theta|<\theta_0$ and
		every $z\in\mathcal V$ satisfying
\[
		u_j(z)\in I
		\qquad\text{for every }j\in\mathcal I_i,
		\]
		define the cell residual of the fixed semidiscrete scheme by
		\begin{equation}\label{eq:two-sided-transfer-residual}
			\begin{split}
				\mathcal R_i^\theta(z)
				:={}-\eta_\theta'(u_i(z))
				\frac{\widehat F^k_{i+1/2}(z)-\widehat F^k_{i-1/2}(z)}{h}
				+\frac{\mathcal Q^\theta_{i+1/2}(z)
					-\mathcal Q^\theta_{i-1/2}(z)}{h}.
			\end{split}
		\end{equation}
		It satisfies
		\begin{equation}\label{eq:two-sided-transfer-inequality}
			\mathcal R_i^\theta(z)\le0
			\qquad (|\theta|<\theta_0),
		\end{equation}
		and the balance is exact at $\theta=0$:
		\begin{equation}\label{eq:base-transfer-equality}
			\mathcal R_i^0(z)=0.
		\end{equation}
	\end{enumerate}
\end{theorem}

\begin{proof}
	Fix the cell $i$ and a point $z\in\mathcal V$ sufficiently close
	to $0$.  
Since $u_*$ lies in the interior of $I$ and $\mathcal I_i$
	is finite, the state constraint in assumption~\textup{(iii)} holds throughout
	a sufficiently small neighborhood.  By
	\eqref{eq:two-sided-transfer-inequality} and
	\eqref{eq:base-transfer-equality}, the differentiable function
	$\theta\mapsto\mathcal R_i^\theta(z)$ has a two-sided maximum at zero.
	Hence
	\begin{equation}\label{eq:first-variation-zero}
		\left.\partial_\theta\mathcal R_i^\theta(z)\right|_{\theta=0}=0.
	\end{equation}
	
	On these modes, the full flux equals the FMT core by
	Corollary~\ref{cor:null-mode-restriction}.  Let
	$\mathcal Q^{\mathrm{cen}}_{\eta_\theta}$ be the centered core reference
	introduced before Lemma~\ref{lem:solvability}, and write
	\[
	\mathcal Q^\theta\big|_{\mathrm{null}}
	=
	\mathcal Q^{\mathrm{cen}}_{\eta_\theta}\big|_{\mathrm{null}}
	-
	\mathcal K^\theta.
	\]
	The core is smooth in its finitely many input states and
	$\eta_\theta$ depends affinely on $\theta$.  The common-stencil and
	regularity assumptions therefore give
	\[
	\theta\longmapsto\mathcal K^\theta
	\in
	C^1\!\left(
	(-\theta_0,\theta_0);
	C^5(\overline{\mathcal U})
	\right).
	\]
	In particular, 
	$
	\dot{\mathcal K}
	:=
	\left.\partial_\theta\mathcal K^\theta\right|_{\theta=0}
	$ 
	defines a translation-invariant finite-stencil local quantity on
	the parameter space of ENO null modes, of class $C^5$ there.
	This regularity suffices to take the homogeneous fifth-order Taylor
	term, without requiring smooth ENO stencil selection away from the
	null modes.
	
	Define $q_\phi'=\phi'f'$ and $\psi_\phi:=\phi'f-q_\phi$.
	Differentiating the centered balance at $\theta=0$ turns
	\eqref{eq:first-variation-zero} into
	\begin{equation}\label{eq:linearized-transfer-balance}
		\begin{split}
			0={}
			\frac{g_{\phi,i+1/2}+g_{\phi,i-1/2}}{2h}
			-\frac{\dot{\mathcal K}_{i+1/2}
				-\dot{\mathcal K}_{i-1/2}}{h},
		\end{split}
	\end{equation}
	where
	\[
	g_{\phi,i+1/2}
	=[\phi']_{i+1/2}\widetilde F^k_{i+1/2}
	-[\psi_\phi]_{i+1/2}.
	\]
	
	Identity~\eqref{eq:linearized-transfer-balance} holds on a neighborhood
	of $0$ in $\mathcal Z_k$.  Taking its degree-five Taylor component and
	applying Lemma~\ref{lem:solvability} gives
	$
	[G_{\phi,k}^{(5)}]=0$ in 
	$H^1_{k-2,5}.
	$ 
	For $D=5$, Corollary~\ref{cor:degreewise-mismatch-criterion} says that
	this class vanishes if and only if
	$f''(u_*)\phi^{(4)}(u_*)=0$, contrary to
	\eqref{eq:generic-transfer-condition}.
\end{proof}

\section{Conclusion}
\label{sec:conclusion}

We have resolved the ENO--TV conjecture for both ENO reconstruction and ENO interpolation on uniform grids and identified the mechanism behind its sharp parity dichotomy. 
The central difficulty is that nonlinear adaptive stencil selection must be controlled uniformly, while ENO null modes allow nonconstant data with vanishing dissipation. Their polynomial characterization identifies polynomial exactness as the source of this degeneracy.
Our key discovery is a localization principle that establishes a two-sided equivalence between ENO dissipation and a weighted finite-difference functional, removing adaptive stencil selection from the coercivity problem.
A hidden quadratic energy, combined with cardinal B-spline lifting and discrete Gagliardo--Nirenberg inequalities, yields the conjectured estimate at every odd order $k\ge3$. Alternating Euler-polynomial blocks yield counterexamples at every even order $k\ge4$.
Together with the second-order estimate, this completes the classification. 
The dichotomy persists on quasi-uniform meshes, while for each $k\ge2$ we have constructed a fixed irregular mesh on which the conjectured estimate fails.

On the same ENO null modes, the quadratic entropy balance is exact, and we have shown that local compatibility with additional entropies is governed by the first cohomology of the cell shift.
Apolar duality and binary covariants make these obstructions computable and determine the dimensions of the resulting obstruction spaces at every Taylor degree. 
For the fixed scalar TeCNO scheme associated with the quadratic entropy, this yields generic obstructions at every order $k\ge4$ to a smooth extension of the exact quadratic balance to local inequalities for both signs of a small entropy perturbation.

Joining ENO null-mode blocks tests global coercivity, while translating their entropy balances tests compatibility with additional entropies.
Through these two uses of the same ENO null modes, we connect nonlinear approximation and discrete interpolation with classical invariant theory.
This connection provides a structural foundation for high-order entropy-stable approximation: it identifies both the quantitative control supplied by ENO dissipation and the algebraic restrictions on extending entropy stability to additional entropies.
It also suggests a broader approach to high-order methods by studying when exact polynomial reconstruction leads to vanishing entropy dissipation.

The localization principle reaches beyond the ENO--TV estimate.
In the companion paper~\cite{LiWuCompanion}, we use it to derive two weaker coercive estimates valid at every reconstruction order and to establish compactness and convergence for arbitrarily high-order entropy-stable ENO schemes.
The resulting theory gives convergence to the entropy solution for scalar conservation laws and subsequential convergence to entropy measure-valued solutions for systems.
These developments connect the structure identified here to the admissibility of limiting solutions, extending the role of polynomial exactness and entropy dissipation from algorithm design to rigorous convergence theory.

\newpage

\appendix

\subsection*{Acknowledgment}  
We thank Professor Ulrik S.~Fjordholm for his helpful suggestion to include numerical illustrations of our counterexamples.

\section{Technical constructions for mesh geometry}

\subsection{Construction of the smooth local-polynomial lift}
\label{app:quasi-patching}

\begin{proof}[Proof of Theorem~\ref{thm:quasi-local-patching}]
	We retain the notation of Subsection~\ref{sec:quasi-uniform-interpolation}.
	For each $j\in\mathbb Z$, let 
	$
	\mathcal W_j:=(x_{j-1/2},x_{j+k-1/2})
	$ 
	be the interval spanned by the $k$ cells
	$I_j,\ldots,I_{j+k-1}$, and let $P_j\in\mathbb P_{k-1}$ be determined by
	\begin{equation}\label{eq:moving-average-interpolant}
		\frac1{h_\ell}\int_{I_\ell}P_j(x)\,dx=\bar u_\ell,
		\qquad \ell=j,\ldots,j+k-1.
	\end{equation}
	To verify uniqueness, suppose that $P\in\mathbb P_{k-1}$ has zero average on
	each of these $k$ cells, and choose a primitive $F\in\mathbb P_k$ with
	$F'=P$.  Then $F$ assumes equal values at the $k+1$ interface points
	bounding the window.  The polynomial $F$ is therefore constant, whence
	$P=0$.  Because the domain and codomain of the cell-average operator both
	have dimension $k$, injectivity implies existence.
	
	After translating the left endpoint of the window to the origin, the local
	geometry is determined by a point in the compact set $[1,\Lambda]^k$.
	The inverse cell-average map depends continuously on this configuration.
	Uniform finite-dimensional norm equivalence therefore yields
	\begin{equation}\label{eq:moving-interpolant-stability}
		\max_{0\le r\le k-1}
		\|P_j^{(r)}\|_{L^\infty(\mathcal W_j)}
		\le C_{k,\Lambda}
		\max_{j\le\ell\le j+k-1}|\bar u_\ell|.
	\end{equation}
	The derivative $P_j'$ vanishes whenever all cell averages in the window
	coincide.  Applying the same argument on the quotient space modulo
	constants gives
	\begin{equation}\label{eq:moving-interpolant-quotient}
		\|P_j'\|_{L^\infty(\mathcal W_j)}
		\le C_{k,\Lambda}
		\max_{j\le\ell\le j+k-1}
		|\bar u_\ell-\bar u_j|
		\le C_{k,\Lambda}A_j^{(k)},
	\end{equation}
	which provides uniform control over the first derivative of the smooth lift.
	
	We next compare the interpolants on two adjacent windows by setting 
	$
	R_j:=P_{j+1}-P_j.
	$ 
	We rely on the following kernel property:
	\begin{equation}\label{eq:adjacent-window-kernel}
		D_j^{(k)}=0
		\quad\Longrightarrow\quad
		P_{j+1}=P_j.
	\end{equation}
	Indeed, by \eqref{eq:nonuniform-primitive-divided-difference},
	$D_j^{(k)}$ is the $(k+1)$-st divided difference of the cumulative
	interface values on the $k+2$ interfaces from $x_{j-1/2}$ through
	$x_{j+k+1/2}$.  If it vanishes, those values are interpolated by a
	polynomial $F\in\mathbb P_k$.  Then $F'\in\mathbb P_{k-1}$ matches the
	prescribed averages on all $k+1$ cells $I_j,\ldots,I_{j+k}$, and uniqueness
	in \eqref{eq:moving-average-interpolant} forces
	$P_j=F'=P_{j+1}$.
	
	For a fixed local geometry, set $b:=(\bar u_j,\ldots,\bar u_{j+k})\in\mathbb R^{k+1}$,
	and regard $D_j^{(k)}$ as a linear functional $d_j(b)$ and $P_{j+1}-P_j$ as
	a linear map $\mathfrak R_j(b)$.  Then \eqref{eq:adjacent-window-kernel}
	translates into the subspace inclusion $\ker d_j\subseteq\ker\mathfrak R_j$.
	
	To make the factorization uniform over all normalized $\Lambda$-quasi-uniform
	local geometries, fix the additive constant in $V$ by setting $V_{j-1/2}=0$.
	The coefficient of the last datum $\bar u_{j+k}$ in $d_j$ is
	\begin{equation}\label{eq:last-datum-coefficient}
		h_{j+k}
		\prod_{q=0}^{k}
		\bigl(x_{j+k+1/2}-x_{j+q-1/2}\bigr)^{-1},
	\end{equation}
	which is bounded away from zero under $1\le h_i\le\Lambda$.  Letting $e_k$
	denote the last standard basis vector of $\mathbb R^{k+1}$, put
	$
	v_j:={e_k}/{d_j(e_k)}.
	$ 
	Then $d_j(v_j)=1$ and $\|v_j\|_{\ell^\infty}\le C_{k,\Lambda}$.  For every
	data vector $b$, the vector $b-d_j(b)v_j$ belongs to $\ker d_j$. Hence
	\[
	\mathfrak R_j(b)=d_j(b)\mathfrak R_j(v_j).
	\]
	Setting $\mathcal R_j:=\mathfrak R_j(v_j)\in\mathbb P_{k-1}$, we obtain the
	factorization
	\begin{equation}\label{eq:adjacent-window-factorization}
		R_j=D_j^{(k)}\mathcal R_j.
	\end{equation}
	Applying \eqref{eq:moving-interpolant-stability} on the two adjacent
	windows to the uniformly bounded vector $v_j$, and recalling the upper and
	lower bounds for $H_j^{(k)}$, gives the following estimate on the cell
	$I_{j+m}$ common to both windows:
	\begin{equation}\label{eq:adjacent-window-bound}
		\|R_j^{(r)}\|_{L^\infty(I_{j+m})}
		\le C_{k,\Lambda}\widehat D_j^{(k)},
		\qquad 0\le r\le k-1.
	\end{equation}
	Furthermore, Lemma~\ref{lem:nonuniform-endpoint-coefficients} and the
	boundedness of $H_j^{(k)}$ in normalized coordinates ensure that
	\begin{equation}\label{eq:normalized-difference-by-amplitude}
		\widehat D_j^{(k)}
		\le C_{k,\Lambda}A_j^{(k)}.
	\end{equation}
	
	It remains to patch the polynomials $P_j$ smoothly.  Set 
	$
	\zeta_j:=x_{j+m-1/2},
	$ 
	which is the interface separating $I_{j+m-1}$ and $I_{j+m}$, so that
	$\zeta_{j+1}-\zeta_j=h_{j+m}\in[1,\Lambda]$.  Fix $0<\rho<1/4$
	independently of the mesh and the data.  Around each interface $\zeta_j$,
	define the interior interval 
	$
	E_j:=(\zeta_j-\rho,\zeta_j+\rho),
	$ 
	and let 
	$
	T_j:=(\zeta_j+\rho,\zeta_{j+1}-\rho)
	$ 
	be the transition interval between $E_j$ and $E_{j+1}$.  These transition
	intervals are mutually disjoint, satisfy $T_j\subset I_{j+m}$, and have
	lengths satisfying $1-2\rho\le|T_j|\le\Lambda-2\rho$.
	
	Fix once and for all a smooth cutoff $\chi\in C^\infty(\mathbb R)$ such that
	$0\le\chi\le1$, $\chi(t)=0$ for $t\le0$, and $\chi(t)=1$ for $t\ge1$.
	The functions
	\[
	\theta_j(x)
	:=\chi\!\left(\frac{x-(\zeta_j+\rho)}{|T_j|}\right)
	\]
	vanish on $(-\infty,\zeta_j+\rho]$, are equal to $1$ on
	$[\zeta_{j+1}-\rho,\infty)$, and have derivatives through order $k$
	bounded uniformly in $j$ by constants depending only on $k$ and $\Lambda$.
	Setting 
$
	\lambda_j:=\theta_{j-1}(1-\theta_j),
	$ 
	we obtain a smooth partition of unity $\sum_j\lambda_j=1$ such that
	$\lambda_j\equiv 1$ on $E_j$, with at most two cutoff functions $\lambda_j$
	nonzero at any given point.
	Because $h_i\ge1$ and $\rho<1/4$, the support of $\lambda_j$ satisfies
	\[
	\operatorname{supp}\lambda_j
	\subset[\zeta_{j-1}+\rho,\zeta_{j+1}-\rho]
	\subset I_{j+m-1}\cup\{\zeta_j\}\cup I_{j+m}
	\subset\mathcal W_j.
	\]
	We now define the smooth lift by
	\begin{equation}\label{eq:smooth-patched-function}
		S_h:=\sum_j\lambda_jP_j.
	\end{equation}
	This sum is locally finite.  For compactly supported data, only finitely many
	$P_j$ are nonzero, so $S_h\in C_c^\infty(\mathbb R)$.  If $P_j\not\equiv 0$,
	its interpolation window meets the set of active cells
	$J:=\{i\in\mathbb Z:\bar u_i\ne0\}$, so the cells supporting $\lambda_j$ lie
	within index distance at most $m$ from $J$.  Finite overlap, the support
	containment $\operatorname{supp}\lambda_j\subset\mathcal W_j$, and
	\eqref{eq:moving-interpolant-stability} then establish
	\eqref{eq:patching-linfty}.
	
	On each interval $E_j$, we have $S_h\equiv P_j$.  Applying the interpolation
	conditions to the two central cells gives
	\begin{equation}\label{eq:central-cell-average-difference}
		a_{j+m-1}
		=
		\frac1{h_{j+m}}\int_{I_{j+m}}P_j
		-
		\frac1{h_{j+m-1}}\int_{I_{j+m-1}}P_j.
	\end{equation}
	The functional on the right-hand side annihilates constants.  On the
	quotient space $\mathbb P_{k-1}/\mathbb R$, the seminorm
	$\|P'\|_{L^{k+1}(E_j)}$ defines a genuine norm and therefore controls every
	linear functional on that quotient.  Translating $\zeta_j$ to the origin,
	this average-difference functional depends continuously on
	$(h_{j+m-1},h_{j+m})\in[1,\Lambda]^2$.  Compactness of this parameter
	domain therefore yields
	\begin{equation}\label{eq:central-functional-bound}
		|a_{j+m-1}|^{k+1}
		\le
		C_{k,\Lambda}
		\int_{E_j}|P_j'(x)|^{k+1}\,dx.
	\end{equation}
	Because the intervals $E_j$ are pairwise disjoint and $j\mapsto j+m-1$ is a
	bijection on $\mathbb Z$, summing \eqref{eq:central-functional-bound}
	establishes \eqref{eq:patching-first-derivative}.
	
	To estimate the highest derivative, observe that $S_h^{(k)}=0$ on every
	$E_j$, since $S_h=P_j$ there and $\deg P_j\le k-1$.  On $T_j$, only the two
	adjacent polynomials contribute:
	\begin{equation}\label{eq:patched-function-on-transition}
		S_h=P_j+\theta_j(P_{j+1}-P_j)
		=P_j+\theta_jR_j.
	\end{equation}
	Leibniz' rule, the difference bound \eqref{eq:adjacent-window-bound}, and
	the uniform derivative bounds on the cutoffs $\theta_j$ give
	\begin{equation}\label{eq:patched-high-derivative-pointwise}
		|S_h^{(k)}|
		\le C_{k,\Lambda}\widehat D_j^{(k)}
		\qquad\text{on }T_j.
	\end{equation}
	Similarly, \eqref{eq:moving-interpolant-quotient},
	\eqref{eq:adjacent-window-bound}, and
	\eqref{eq:normalized-difference-by-amplitude} yield
	\begin{equation}\label{eq:patched-first-derivative-pointwise}
		|S_h'|
		\le C_{k,\Lambda}A_j^{(k)}
		\qquad\text{on }T_j.
	\end{equation}
	Since the transition intervals $T_j$ are pairwise disjoint and have lengths
	bounded by $\Lambda$, multiplying \eqref{eq:patched-high-derivative-pointwise}
	and \eqref{eq:patched-first-derivative-pointwise}, integrating over each
	$T_j$, and summing over $j$ completes the proof of
	\eqref{eq:patching-high-derivative}.
\end{proof}

\subsection{The fixed slowly graded mesh}
\label{app:graded-mesh-construction}

\begin{proof}[Proof of Theorem~\ref{thm:finite-window-mesh-failure}]
	Set
	\begin{equation}\label{eq:graded-mesh-r-parameter}
		r:=\Lambda^{1/(2R)}>1,
		\qquad
		M:=10k.
	\end{equation}
	We use the tent function
	\begin{equation}\label{eq:graded-mesh-tent}
		\phi(x)
		:=
		\begin{cases}
			0,&x\le0,\\
			x,&0\le x\le\frac12,\\
			1-x,&\frac12\le x\le1,\\
			0,&x\ge1.
		\end{cases}
	\end{equation}
	Choose 
	$
	\varepsilon_n:=r^{-n^2-N_0},
	$ 
	where $N_0$ is chosen sufficiently large that Lemma~\ref{lem:graded-mesh-bridge}
	applies with $L=1/2$ and $\varepsilon=\varepsilon_n$ for every $n$.
	
	We first construct the mesh on the $n$-th copy of $[0,1]$.  Apply
	Lemma~\ref{lem:graded-mesh-bridge} separately to $[0,1/2]$ and
	$[1/2,1]$, with endpoint cell size $\varepsilon_n$.  The three
	points $0,1/2,1$ then fall on cell interfaces.  At the midpoint, there are at
	least $M$ consecutive cells of size $\varepsilon_n$ on either side. At the
	endpoints $0$ and $1$, these buffers initially lie only on the inward side.
	We append $M$ additional $\varepsilon_n$-cells in the exterior zero region
	adjacent to each endpoint, ensuring that all three corners possess uniform
	$\varepsilon_n$-buffers on both sides.  Consecutive cell-size ratios remain in
	$[r^{-1},r]$, and each open affine half of $\phi$ contains a cell of length
	at least $c_r/2$ (see the schematic in Figure~\ref{fig:nonuniform-tent-mesh}).
	
	We next embed these finite patterns into a single infinite mesh while
	preserving the adjacent-ratio bound.  Following the right $\varepsilon_n$-buffer,
	append geometrically increasing cells with sizes
	\[
	\varepsilon_nr,
	\varepsilon_nr^2,
	\ldots,
	\varepsilon_nr^{J_n},
	\qquad
	J_n:=\max\{j\in\mathbb N:\varepsilon_nr^j\le1\}.
	\]
	The final cell in this progression has length in $(r^{-1},1]$.  Beginning with
	$\varepsilon_{n+1}r^{J_{n+1}}\in(r^{-1},1]$, append the descending list
	\[
	\varepsilon_{n+1}r^{J_{n+1}},
	\varepsilon_{n+1}r^{J_{n+1}-1},
	\ldots,
	\varepsilon_{n+1}r,
	\]
	followed by the left $\varepsilon_{n+1}$-buffer of the next pattern.
	The ratio of the two peak cell lengths belongs to $[r^{-1},r]$, ensuring that
	every pair of consecutive cell sizes in the resulting mesh satisfies
	\begin{equation}\label{eq:graded-adjacent-r-bound}
		r^{-1}\le\frac{h_{i+1}}{h_i}\le r.
	\end{equation}
	Because each connecting transition contains a cell of size at least $r^{-1}$,
	the mesh interfaces tend to $+\infty$.  A finite geometric transition
	followed by unit cells extends the mesh symmetrically toward $-\infty$ while
	preserving \eqref{eq:graded-adjacent-r-bound}.  This produces one fixed,
	locally finite, strictly increasing mesh containing all the local patterns.
	
	Let $y_n$ denote the left endpoint of the $n$-th copy of $[0,1]$,
	and let $\bar u^{(n)}$ be the cell averages on the fixed mesh of the
	translated tent 
	$
	\phi_n(x):=\phi(x-y_n).
	$ 
	Because the endpoints of the tent are mesh interfaces and the mesh is
	locally finite, these sequences are compactly supported.  The $M$-cell
	zero buffers ensure that every reconstruction stencil meeting their support
	remains inside the corresponding local pattern.  The height of the tent gives
	\begin{equation}\label{eq:graded-mesh-amplitude-bound}
		\|\bar u^{(n)}\|_{\ell^\infty}\le\frac12.
	\end{equation}
	
	On either affine half of the tent, the average of $\phi_n$ over a cell
	equals its value at the center of that cell.  Since each affine half contains a
	cell of length at least $c_r/2$, the jump between the average on this cell
	and either adjacent cell within the same affine segment has magnitude at
	least $c_r/4$.  Selecting one such jump from each affine half yields, for
	every $p\ge2$,
	\begin{equation}\label{eq:graded-mesh-moment-lower-bound}
		\sum_i|a_i^{(n)}|^p
		\ge
		2\left(\frac{c_r}{4}\right)^p,
	\end{equation}
	which proves \eqref{eq:graded-mesh-moment}.
	
	For each $\nu\in\{0,1,2\}$, let $\iota_{n,\nu}\in\mathbb Z$ be the unique
	interface index such that 
	$
	x_{\iota_{n,\nu}+1/2}
	=
	y_n+\frac{\nu}{2}.
	$ 
	Define the exceptional interface set
	\begin{equation}\label{eq:graded-exceptional-interface-set}
		\mathcal A_n
		:=
		\bigcup_{\nu=0}^{2}
		\{\iota_{n,\nu}-k+1,\ldots,\iota_{n,\nu}+k-1\}.
	\end{equation}
	Thus $\mathcal A_n$ is precisely the set of interface indices $i$
	for which the complete reconstruction hull $\mathcal H_{i,k}$
	straddles one of the three corners of $\phi_n$.  In particular, 
	$
	|\mathcal A_n|\le 3(2k-1).
$ 
	
	If $i\notin\mathcal A_n$, all cells in $\mathcal H_{i,k}$ lie within a
	region where $\phi_n$ is a single affine function.  General-mesh polynomial
	exactness \eqref{eq:general-mesh-polynomial-exactness} therefore gives 
	$
	\Delta p_{i+1/2}^{(k,h)}=0.
	$ 
	
	If $i\in\mathcal A_n$, the choice $M=10k$ ensures that $\mathcal H_{i,k}$
	is contained in the uniform $\varepsilon_n-submesh$ surrounding one corner
	$s$ of the tent.  The Lipschitz continuity of $\phi_n$ gives
	\[
	|\bar u_q^{(n)}-\phi_n(s)|
	\le C_k\varepsilon_n,
	\qquad q\in\mathcal H_{i,k}.
	\]
	In particular, 
	$
	|a_i^{(n)}|\le C_k\varepsilon_n.
	$ 
	After the rescaling $x=s+\varepsilon_n z$, only finitely many unit-grid
	stencil positions are possible.  Evaluation of a cell-average
	interpolation polynomial at the interface is a bounded linear functional
	of the cell averages, uniformly over this finite family.  Since constant
	data reconstruct the same constant, both reconstructed traces differ
	from $\phi_n(s)$ by at most $C_k\varepsilon_n$.  Therefore, 
	$
	\bigl|\Delta p_{i+1/2}^{(k,h)}\bigr|
	\le C_k\varepsilon_n
	$ 
	at every exceptional interface.  Summing over the $O(k)$ exceptional
	interfaces yields
	\begin{equation}\label{eq:graded-mesh-source-bound}
		Q_{k,h}^{\mathrm{abs}}(\bar u^{(n)})
		\le C_k\varepsilon_n^2
		\longrightarrow0.
	\end{equation}
	
	Combining \eqref{eq:graded-adjacent-r-bound} with
	\eqref{eq:graded-mesh-r-parameter} gives
	\[
	\Lambda^{-1/(2R)}
	\le
	\frac{h_{i+1}}{h_i}
	\le
	\Lambda^{1/(2R)},
	\]
	which is \eqref{eq:adjacent-mesh-ratio}.  Any two indices in the window
	$|q-i|\le R$ differ by at most $2R$.  Iterating the adjacent-ratio
	bound therefore gives
	\[
	\frac{\max_{|q-i|\le R}h_q}
	{\min_{|q-i|\le R}h_q}
	\le r^{2R}=\Lambda,
	\]
	proving \eqref{eq:sliding-window-mesh-ratio}.  Combining this estimate
	with \eqref{eq:graded-mesh-amplitude-bound},
	\eqref{eq:graded-mesh-moment-lower-bound}, and
	\eqref{eq:graded-mesh-source-bound} completes the proof of Theorem~\ref{thm:finite-window-mesh-failure}.
\end{proof}

\section{Apolar calculations for the FMT entropy defect}
\label{app:invariant-details}

This appendix provides the explicit constructions and coefficient
evaluations underlying the apolar results employed in the entropy-transfer
analysis.

\subsection{The degree-five apolar functional}
\label{app:proof-degree-five-source}

\begin{proof}[Proof of Lemma~\ref{lem:universal-degree-five-source}]
	Consider the connected bipartite multigraph $\Gamma$ with vertex classes
	$\{1,2,3\}$ and $\{4,5\}$, whose nonzero edge multiplicities are
	\[
	m_{14}=n-1,\qquad m_{24}=1,\qquad
	m_{25}=1,\qquad m_{35}=\epsilon-1.
	\]
	It has $j_0=n+\epsilon$ edges and vertex degrees 
	$
	(n-1,2,\epsilon-1,n,\epsilon), 
	$ 
	all at most $n$.  For binary linear forms $\ell_i$, let
	$[\ell_i,\ell_j]$ denote the determinant of their coefficient vectors, and
	form the symbolic bracket monomial
	\begin{equation}\label{eq:degree-five-bracket-graph}
		\begin{split}
			\mathbf B_\Gamma(\ell_1,\ldots,\ell_5;Z)
			:={}[\ell_1,\ell_4]^{n-1}
			[\ell_2,\ell_4][\ell_2,\ell_5]
			[\ell_3,\ell_5]^{\epsilon-1}
			\times
			\prod_{i=1}^{5}
			\ell_i(Z)^{n-\deg_\Gamma(i)}.
		\end{split}
	\end{equation}
	Set 
	$
	q:=5n-2j_0.
	$ 
	Since
	\[
	j_0=
	\begin{cases}
		n+2,&n\text{ odd},\\
		n+3,&n\text{ even},
	\end{cases}
	\]
	we have
	\[
	q=
	\begin{cases}
		3n-4,&n\text{ odd},\\
		3n-6,&n\text{ even},
	\end{cases}
	\]
	which is strictly positive for every $n\ge3$.
	
	The total degree in the output variable $Z$ in
	\eqref{eq:degree-five-bracket-graph} is
	\[
	\sum_{i=1}^{5}\bigl(n-\deg_\Gamma(i)\bigr)
	=
	5n-2j_0
	=
	q,
	\]
	since each edge contributes $1$ to the degree of each endpoint.  Thus
	$\mathbf B_\Gamma(\ell_1,\ldots,\ell_5;Z)$ is a binary form of degree
	$q$ in $Z$.
	
	For each $i\in\{1,\ldots,5\}$, the expression $\mathbf B_\Gamma$ is also
	homogeneous of degree $n$ in the coefficients of $\ell_i$: the brackets
	containing $\ell_i$ contribute degree $\deg_\Gamma(i)$, while
	$\ell_i(Z)^{n-\deg_\Gamma(i)}$ contributes degree
	$n-\deg_\Gamma(i)$.
	
	By the universal property of symmetric powers, this multi-homogeneity
	defines a multilinear map
	\[
	\mathcal T_\Gamma:(V_n)^5\longrightarrow V_q
	\]
	characterized on pure powers by
	\[
	\mathcal T_\Gamma
	\bigl(\ell_1^n,\ldots,\ell_5^n;Z\bigr)
	=
	\mathbf B_\Gamma(\ell_1,\ldots,\ell_5;Z).
	\]
	Symmetrizing $\mathcal T_\Gamma$ over the permutation group $\mathfrak S_5$,
	\[
	\mathcal T_\Gamma^{\mathrm{sym}}(F_1,\ldots,F_5;Z)
	:=
	\frac1{5!}
	\sum_{\sigma\in\mathfrak S_5}
	\mathcal T_\Gamma
	(F_{\sigma(1)},\ldots,F_{\sigma(5)};Z),
	\]
	yields an $SL_2$-equivariant symmetric multilinear map whose diagonal
	restriction coincides with that of $\mathcal T_\Gamma$:
	\[
	\mathcal T_\Gamma^{\mathrm{sym}}(F,\ldots,F;Z)
	=
	\mathcal T_\Gamma(F,\ldots,F;Z).
	\]
	We therefore define the homogeneous degree-five covariant
	\[
	\mathbf B_n(F;Z)
	:=
	\mathcal T_\Gamma(F,\ldots,F;Z).
	\]
	
	We verify that the resulting covariant is nonzero by evaluating it on a
	linear combination of two pure powers,
	\[
	F=c\ell^n+dm^n.
	\]
	By multilinearity, $\mathbf B_n(F;Z)$ is the sum over all assignments of
	either $\ell$ or $m$ to the five vertices of $\Gamma$.  A term can be
	nonzero only if the two endpoints of every edge receive different linear
	forms, since identical forms produce vanishing brackets $[\ell,\ell]=0$ or
	$[m,m]=0$.
	
	Because $\Gamma$ is connected and bipartite with bipartition
	$\{1,2,3\}\sqcup\{4,5\}$, exactly two assignments satisfy this condition:
	\begin{enumerate}[label=\textup{(\alph*)}]
		\item vertices $1,2,3$ receive $\ell$, while vertices $4,5$ receive $m$,
		\item vertices $1,2,3$ receive $m$, while vertices $4,5$ receive $\ell$.
	\end{enumerate}
	
	For the first assignment, every bracket is $[\ell,m]$, and the
	coefficient from the five copies of $F$ is $c^3d^2$.  The remaining
	exponent of $\ell(Z)$ is
	\[
	\sum_{i=1}^{3}\bigl(n-\deg_\Gamma(i)\bigr)
	=
	3n-j_0,
	\]
	while the remaining exponent of $m(Z)$ is
	\[
	\sum_{i=4}^{5}\bigl(n-\deg_\Gamma(i)\bigr)
	=
	2n-j_0.
	\]
	The second assignment gives the same expression with $\ell$ and
	$m$ interchanged.  Since $j_0$ is odd, 
	$
	[m,\ell]^{j_0}
	=
	-[\ell,m]^{j_0}.
	$ 
	It follows that
	\begin{align}
		\mathbf B_n(c\ell^n+dm^n;Z)
		={}&[\ell,m]^{j_0}
		\Bigl[
		c^3d^2
		\ell(Z)^{3n-j_0}
		m(Z)^{2n-j_0}
		\notag\\
		&\hspace{25mm}
		-
		c^2d^3
		\ell(Z)^{2n-j_0}
		m(Z)^{3n-j_0}
		\Bigr].
		\label{eq:graph-covariant-secant}
	\end{align}
	The exponents in this expression are nonnegative.  Indeed,
	\[
	2n-j_0=
	\begin{cases}
		n-2,&n\text{ odd},\\
		n-3,&n\text{ even},
	\end{cases}
	\]
	which is nonnegative for every $n\ge3$ of the corresponding parity.
	The other exponent $3n-j_0$ is strictly larger.  Thus the right-hand side
	of \eqref{eq:graph-covariant-secant} is a nonzero polynomial, proving that
	$\mathbf B_n$ is nonzero.
	
	Write the covariant in output coordinates $Z=(Z_0,Z_1)$ as
	\[
	\mathbf B_n(F;Z)
	=
	\sum_{r=0}^{q}
	B_{n,r}(F)Z_0^{q-r}Z_1^r.
	\]
	By the covariant--test principle in
	Theorem~\ref{thm:cohomology-covariants}, its highest-weight
	coefficient 
	$
	\mathfrak B_n:=B_{n,0}
	$ 
	satisfies 
	$
	\mathcal D\mathfrak B_n=0
	$ 
	and has $\mathsf h_y$-weight $q$.  Since $\mathbf B_n$ has
	degree five in $F$,
	\[
	\mathfrak B_n\in
	\ker\mathcal D\cap W^y_{n,5}[q].
	\]
	Using $q=5n-2j_0$, we conclude that
	\[
	\mathfrak B_n\in\mathscr S_{n,5,j_0}.
	\]
	
	It remains to compute its secant restriction.  In the affine pure-power chart,
	take
	\[
	\ell(Z)=Z_0+tZ_1,
	\qquad
	m(Z)=Z_0+sZ_1.
	\]
	Then 
	$
	[\ell,m]=s-t.
	$ 
	The coefficient of $Z_0^q$ in \eqref{eq:graph-covariant-secant} is
	\[
	(s-t)^{j_0}\bigl(c^3d^2-c^2d^3\bigr)
	=
	(s-t)^{j_0}(cd)^2(c-d).
	\]
	Under the divided-coordinate identification of pure-power binary forms of
	degree $n$ with the evaluation curve $\zeta_n$, coordinate normalization
	introduces only a fixed nonzero scalar.  Hence
	\[
	\mathfrak B_n
	\bigl(c\zeta_n(t)+d\zeta_n(s)\bigr)
	=
	\kappa_n
	(s-t)^{j_0}(cd)^2(c-d)
	\]
	for some $\kappa_n\ne0$.  This proves
	\eqref{eq:universal-degree-five-secant}.
\end{proof}

\subsection{Nondegeneracy of the covariance pairing}
\label{app:proof-covariance-nondegeneracy}

\begin{proof}[Proof of Lemma~\ref{lem:covariance-nondegeneracy}]
	The covariance identity gives the antisymmetry and cubic divisibility of
	$C_{r,s}$.  To prove linear independence, we exhibit a triangular
	coefficient matrix.  Put $R:=\lfloor(D-1)/2\rfloor$, take $1\le r\le s$
	with $r+s=D-1$, and write
	\[
	b_\ell^{(r)}
	:=[X^{D-\ell}Y^\ell]C_{r,s},
	\qquad1\le\ell\le R.
	\]
	For $0\le t\le r$, the coefficient of $X^{r+s-t}Y^t$ in $P_rP_s$ is
	$(t+1)/((r+1)(s+1))$. For $r<t\le s$, it takes the constant value
	$1/(s+1)$.  Since $C_{r,s}=(Y-X)(P_rP_s-P_{r+s})$, subtraction of
	adjacent coefficients gives
	\[
	b_\ell^{(r)}
	=
	\begin{cases}
		-\dfrac1{(r+1)(s+1)},&\ell\le r,\\[2mm]
		0,&\ell>r.
	\end{cases}
	\]
	With rows indexed by $\ell=1,\ldots,R$ and columns by $r=1,\ldots,R$,
	this matrix is upper triangular with diagonal $-1/((r+1)(D-r))\ne0$.
	Thus these $C_{r,s}$ are linearly independent.  Their number equals
	$\dim\mathscr V_D$, so they form a basis.
	
	For the interior nondegeneracy assertion, assume now that $n\ge3$ and
	$D\ge5$, and set $L:=D-5$.  The polynomials
	\[
	(c+d)^{L-2a}(cd)^a,
	\qquad0\le a\le\left\lfloor\frac L2\right\rfloor,
	\]
	form a basis of the symmetric homogeneous polynomials of degree $L$.  By
	Fischer adjointness, multiplication of these basis polynomials by
	$(cd)^2(c-d)$ corresponds to the differential operator
	\begin{equation}\label{eq:covariance-differential-map}
		\mathscr L_DH
		:=\partial_X^2\partial_Y^2(\partial_X-\partial_Y)H.
	\end{equation}
	Put
	$
	\sigma:=X+Y$ and 
	$
	\xi:=Y-X.
	$ 
	Every $H\in\mathscr V_D$ has a unique representation
	\[
	H
	=
	\xi^3
	\sum_{a=0}^{\lfloor(D-3)/2\rfloor}
	h_a\sigma^{D-3-2a}\xi^{2a}.
	\]
	Moreover, up to a fixed nonzero constant, 
	$
	\mathscr L_D
	=
	(\partial_\sigma^2-\partial_\xi^2)^2\partial_\xi.
	$ 
	To compute the rank of this operator, put
	\[
	H_a:=\xi^{3+2a}\sigma^{D-3-2a},
	\qquad
	0\le a\le\left\lfloor\frac{D-3}{2}\right\rfloor,
	\]
	and
	\[
	K_b:=\xi^{2b}\sigma^{D-5-2b},
	\qquad
	0\le b\le\left\lfloor\frac{D-5}{2}\right\rfloor.
	\]
	The $H_a$ form a basis of $\mathscr V_D$, while the $K_b$ form a basis
	of the symmetric homogeneous polynomials of degree $L=D-5$.  For $a\ge1$,
	the term of $\mathscr L_DH_a$ with lowest $\xi$-degree is, up to the same
	fixed nonzero normalization,
	\[
	(2a+3)(2a+2)(2a+1)(2a)(2a-1)
	\xi^{2a-2}\sigma^{D-3-2a}.
	\]
	These lowest-$\xi$-degree terms are nonzero scalar multiples of
	$K_{a-1}$.  The resulting matrix is therefore triangular with nonzero
	diagonal on $H_1,\ldots,H_{\lfloor(D-3)/2\rfloor}$, and $\mathscr L_D$ is
	surjective onto the space of symmetric homogeneous polynomials of
	degree $L$.  Counting basis elements shows that its kernel is
	one-dimensional.  A direct calculation gives $\mathscr L_DC_{1,D-2}=0$, so
	the kernel is precisely the span of this endpoint covariance.
	Restricting the operator to the interior covariance space $\mathscr V_D^{\mathrm{int}}$
	therefore gives an isomorphism, proving the claimed nondegeneracy.
	
	It remains to prove the assertion \eqref{eq:endpoint-all-tests}.
	This part of the argument requires only $D\ge3$ and does not require
	$n\ge3$.  The polynomial $C_{1,D-2}$ contains only pure powers and
	monomials in which one variable has exponent one.
	
	Let $A\in\mathscr S_{n,D,j}$.  If $j$ is even, then \eqref{eq:secant-parity}
	makes $R_A$ symmetric in $c,d$, whereas $C_{1,D-2}$ is antisymmetric.  Hence 
	$
	\langle R_A,C_{1,D-2}\rangle_F=0.
	$ 
	
	Suppose next that $j$ is odd.  Since $j>0$, every monomial occurring in
	$A$ contains at least one positive-index variable.  Therefore 
	$
	R_A(c,0)=0.
	$ 
	By \eqref{eq:secant-parity}, also $R_A(0,d)=0$, so the pure-power
	terms of $C_{1,D-2}$ give no contribution.
	
	The coefficient of $c^{D-1}d$ in $R_A(c,d)$ can be nonzero only if
	$A$ contains a monomial of the form 
	$
	y_0^{D-1}y_j.
	$ 
	In particular, this requires $j\le n$.  If $j=1$, then
	$\mathscr S_{n,D,1}=0$ by Lemma~\ref{lem:all-free-FMT-moments}.  If
	$3\le j\le n$, then $j<j_0=2\mu+1$, and the FMT moment conditions give 
	$
	M_{\mu,j}=0.
	$ 
	If $j>n$, the endpoint pairing already vanishes.  Thus, in every case, 
	$
	M_{\mu,j}
	\langle R_A,C_{1,D-2}\rangle_F=0.
	$ 
	By the secant-covariance formula \eqref{eq:secant-covariance-formula} and
	apolar duality (Lemma~\ref{lem:apolar-principle}), the endpoint direction
	thus contributes the zero class after the full FMT sum.
\end{proof}


\begin{thebibliography}{99}
	
	\bibitem{BLS}
	Y. Brenier, C. De Lellis, and L. Sz\'ekelyhidi Jr.,
	\emph{Weak--strong uniqueness for measure-valued solutions},
	Comm. Math. Phys. 305 (2011), 351--361.

\bibitem{CayleySecondQuantics}
A. Cayley,
\emph{A second memoir upon quantics},
Philos. Trans. Roy. Soc. London 146 (1856), 101--126.

\bibitem{ChenJiaSemi}
W. Y. C. Chen and I. D. D. Jia,
\emph{Semi-invariants of binary forms and Sylvester's theorem},
Ramanujan J. 59 (2022), 297--311.

\bibitem{CFM2AN}
N. Chatterjee and U. S. Fjordholm,
\emph{Convergence of second-order, entropy stable methods for
multi-dimensional conservation laws},
ESAIM Math. Model. Numer. Anal. 54 (2020), 1415--1428.


\bibitem{DiPernaMV}
R. J. DiPerna,
\emph{Measure-valued solutions to conservation laws},
Arch. Rational Mech. Anal. 88 (1985), 223--270.

\bibitem{FjordholmThesis}
U. S. Fjordholm,
\emph{High-order accurate entropy stable numerical schemes for hyperbolic
conservation laws},
Ph.D. thesis, ETH Zurich, Diss. No.~21025, 2013.

\bibitem{FjordholmHandbook}
U. S. Fjordholm,
\emph{Stability properties of the ENO method},
in Handbook of Numerical Methods for Hyperbolic Problems: Basic and
Fundamental Issues, Handb. Numer. Anal. 17, Elsevier, 2016, pp.~123--145.

\bibitem{FKMT}
U. S. Fjordholm, R. K\"appeli, S. Mishra, and E. Tadmor,
\emph{Construction of approximate entropy measure-valued solutions for
hyperbolic systems of conservation laws},
Found. Comput. Math. 17 (2017), 763--827.

\bibitem{FMTtecno}
U. S. Fjordholm, S. Mishra, and E. Tadmor,
\emph{Arbitrarily high-order accurate entropy stable essentially
nonoscillatory schemes for systems of conservation laws},
SIAM J. Numer. Anal. 50 (2012), 544--573.

\bibitem{FMT}
U. S. Fjordholm, S. Mishra, and E. Tadmor,
\emph{ENO reconstruction and ENO interpolation are stable},
Found. Comput. Math. 13 (2013), no.~2, 139--159,
\href{https://doi.org/10.1007/s10208-012-9117-9}{doi:10.1007/s10208-012-9117-9}.

\bibitem{FultonHarris}
W. Fulton and J. Harris,
\emph{Representation Theory: A First Course},
Graduate Texts in Mathematics 129, Springer, 1991.


\bibitem{Harten}
A. Harten, B. Engquist, S. Osher, and S. R. Chakravarthy,
\emph{Uniformly high order accurate essentially non-oscillatory schemes, III},
J. Comput. Phys. 71 (1987), 231--303.

\bibitem{HiltebrandMishra}
A. Hiltebrand and S. Mishra,
\emph{Entropy stable shock capturing space--time discontinuous Galerkin
schemes for systems of conservation laws},
Numer. Math. 126 (2014), 103--151.

\bibitem{Kruzkov}
S. N. Kruzhkov,
\emph{First order quasilinear equations in several independent variables},
Math. USSR-Sb. 10 (1970), 217--243.

\bibitem{LiWuCompanion}
Z. Li and K. Wu,
\emph{Compactness and convergence theory for entropy stable ENO schemes 
	of arbitrarily high order},
preprint, 2026.

\bibitem{LPT}
P.-L. Lions, B. Perthame, and E. Tadmor,
\emph{A kinetic formulation of multidimensional scalar conservation laws and
related equations},
J. Amer. Math. Soc. 7 (1994), 169--191.

\bibitem{Nirenberg}
L. Nirenberg,
\emph{On elliptic partial differential equations},
Ann. Scuola Norm. Sup. Pisa Cl. Sci. (3) 13 (1959), no.~2, 115--162.


\bibitem{Schumaker}
L. L. Schumaker,
\emph{Spline Functions: Basic Theory}, 3rd ed.,
Cambridge Mathematical Library, Cambridge University Press, 2007.


\bibitem{StanleyEC1}
R. P. Stanley,
\emph{Enumerative Combinatorics, Vol. 1}, 2nd ed.,
Cambridge Studies in Advanced Mathematics 49, Cambridge University Press,
2012.

\bibitem{TadmorTao}
E. Tadmor and T. Tao,
\emph{Velocity averaging, kinetic formulations, and regularizing effects in
quasi-linear PDEs},
Comm. Pure Appl. Math. 60 (2007), 1488--1521.

\bibitem{ZakerzadehMay}
M. Zakerzadeh and G. May,
\emph{On the convergence of a shock capturing discontinuous Galerkin method
for nonlinear hyperbolic systems of conservation laws},
SIAM J. Numer. Anal. 54 (2016), 874--898.

\end{thebibliography}
\end{document}